\documentclass[10pt]{amsart}
\usepackage[utf8]{inputenc}
\usepackage{amssymb}
\usepackage{amsfonts}
\usepackage{amsmath}
\usepackage{xcolor}
\usepackage[english]{babel}
\usepackage{amsthm}
\usepackage{setspace}
\usepackage{tikz-cd}
\usepackage{graphicx}
\usepackage[perpage]{footmisc}
\usepackage[shortlabels]{enumitem}
\usepackage{longtable}
\usepackage{comment}
\usepackage{diagbox}
\usepackage{bbm}
\usepackage{booktabs}

\usepackage{mathtools}

\newtheoremstyle{example}
{}                
{}                
{\sffamily}        
{}                
{\bfseries}       
{.}               
{ }               
{}                

\usepackage[
backend=biber,
style=alphabetic
]{biblatex}
\usepackage{biblatex}
\usepackage{hyperref}
\usepackage{xurl}
\hypersetup{breaklinks=true}

\newtheoremstyle{example}
{}                
{}                
{\sffamily}        
{}                
{\bfseries}       
{.}               
{ }               
{}                

\usepackage[
backend=biber,
style=alphabetic
]{biblatex}
\newtheorem{theorem}{Theorem}

\newtheorem{lemma}[theorem]{Lemma}
\newtheorem{corollary}[theorem]{Corollary}

\theoremstyle{definition}
\newtheorem{definition}[theorem]{Definition}
\newtheorem{algorithm}[theorem]{Algorithm}
\newtheorem{remark}[theorem]{Remark}

\newtheorem{conjecture}[theorem]{Conjecture}

\numberwithin{theorem}{section}

\numberwithin{theorem}{section}

\DeclareMathOperator{\im}{im}

\DeclareMathOperator{\Tr}{Tr}

\DeclareMathOperator{\Gal}{Gal}

\DeclareMathOperator{\Aut}{Aut}

\DeclareMathOperator{\Nm}{Nm}

\DeclareMathOperator{\disc}{disc}

\DeclareMathOperator{\Res}{Res}

\DeclareMathOperator{\Part}{Part}
\DeclareMathOperator{\chara}{char}
\DeclareMathOperator{\colspan}{colspan}

\newcommand{\co}{\mathcal{O}}

\newcommand{\Z}{\mathbb{Z}}
\newcommand{\R}{\mathbb{R}}
\newcommand{\Q}{\mathbb{Q}}
\newcommand{\C}{\mathbb{C}}
\newcommand{\F}{\mathbb{F}}

\newcommand{\p}{\mathfrak{p}}

\newcommand{\m}{\widetilde{m}}
\newcommand{\Ext}[1]{\mathbf{Ext}_{#1}}

\newcommand{\Et}[2]{\textnormal{\'Et}_{#1}^{#2}}
\newcommand{\mff}{\mathfrak{f}}

\newcommand{\Split}{\mathrm{Split}}
\newcommand{\pred}{\mathrm{pred}}
\newcommand{\unpred}{\mathrm{unpred}}
\newcommand{\epi}{\mathrm{epi}}

\usepackage[margin=1.2in]{geometry}
\title{$S_n$-extensions with prescribed norms}
\author{Sebastian Monnet}

\begin{document}
\maketitle

\begin{abstract}
    Given a number field $k$, a finitely generated subgroup $\mathcal{A}\subseteq k^\times$, and an integer $n\geq 3$, we study the distribution of degree $n$ field extensions of $k$ with Galois closure group $S_n$, such that the elements of $\mathcal{A}$ are norms. For $n\leq 5$, and conjecturally for $n \geq 6$, we show that the density of such extensions is the product of so-called ``local masses'' at the places of $k$. When $n$ is an odd prime, we give formulae for these local masses, allowing us to express the aforementioned density as an explicit Euler product. For $n=4$, we determine almost all of these masses explicitly and give an efficient algorithm for computing the rest, again giving us the tools to evaluate the density as an Euler product.
\end{abstract}
\tableofcontents

\clearpage
\section{Introduction}
\label{sec-intro}

Fix a number field $k$ and let $\mathcal{A}\subseteq k^\times$ be a finitely generated subgroup. In \cite{newton-nf-prescribed}, Frei, Loughran, and Newton count abelian extensions $K/k$ such that $\mathcal{A}\subseteq N_{K/k}K^\times$. The natural follow-up is: what about non-abelian extensions? Counting non-abelian number fields in general is the subject of Malle's Conjecture (\cite[Conjecture~1.1]{malle-conj}), arguably the biggest open problem in arithmetic statistics, and we do not address it here. Instead, we tackle one of the best-understood classes of non-abelian number field, $S_n$-$n$-ics.

An \emph{$S_n$-$n$-ic extension of $k$} is a degree $n$ field extension $K/k$, whose normal closure $\widetilde{K}$ has $\Gal(\widetilde{K}/k) \cong S_n$. Given a real number $X$, write $N_{k,n}(X)$ for the number of $S_n$-$n$-ics $K/k$ with $\Nm(\disc(K/k)) \leq X$. The ``Malle--Bhargava heuristics'' give conjectural asymptotics for $N_{k,n}(X)$, which have been proved for $n \leq 5$ by Bhargava and collaborators (see \cite{bhargava-mass-formula} for the original heuristics and \cite{geom-of-nums} for the proof when $n\leq 5$). The cases $n\geq 6$ are wide open and very famous. 

Given a real number $X$, write $N_{k,n}(X;\mathcal{A})$ for the number of $S_n$-$n$-ic extensions $K/k$ such that $\Nm(\disc(K/k)) \leq X$ and $\mathcal{A}\subseteq N_{K/k}K^\times$. We are primarily interested in the leading coefficient,
$$
\lim_{X\to\infty} \frac{N_{k,n}(X;\mathcal{A})}{X},
$$
of the asymptotic expansion of $N_{k,n}(X;\mathcal{A})$, as well as the proportion 
$$
\lim_{X\to\infty} \frac{N_{k,n}(X;\mathcal{A})}{N_{k,n}(X)}
$$
of $S_n$-$n$-ics with $\mathcal{A}$ inside their norm groups. 

Since $N_{k,n}(X)$ is only known for $n \leq 5$, barring major breakthroughs in the field, we can only hope to solve our problem completely for $n=3,4,5$. It is common for papers to focus only on one of these three cases, but we have opted to address all three. In fact, contingent on the Malle--Bhargava heuristics, we extend our results to conjectures for $n\geq 6$.
\subsection{Qualitative results}

In Section~\ref{sec-general-n}, we explore our problem in as much generality as possible, namely for all $n\geq 3$. As already discussed, our results are conjectural for $n \geq 6$, since they assume the Malle--Bhargava heuristics. Specialising to the known cases, we isolate the following headline result: 
\begin{theorem}
    \label{thm-prop-between-0-and-1-for-n-leq-5}
    For $n\in \{3,4,5\}$, we have 
    $$
    0 < \lim_{X\to\infty} \frac{N_{k,n}(X;\mathcal{A})}{N_{k,n}(X)} \leq 1, 
    $$
    with equality if and only if $\mathcal{A}\subseteq k^{\times n}$. 
\end{theorem}

\begin{corollary}
    \label{cor-infinitely-many-S-n-n-ics-with-norm-group}
    Let $n \in \{3,4,5\}$. For every finitely generated subgroup $\mathcal{A}\subseteq k^\times$, there are infinitely many $S_n$-$n$-ic extensions $K/k$ with $\mathcal{A}\subseteq N_{K/k}K^\times$. 
\end{corollary}

For arbitrary $n\geq 3$, the more general (conjectural) analogue of Theorem~\ref{thm-prop-between-0-and-1-for-n-leq-5} is contained in Theorems~\ref{thm-prop-between-0-and-1} and \ref{thm-hasse-for-nth-powers}.

Let us unpack the content of Theorem~\ref{thm-prop-between-0-and-1-for-n-leq-5}. It tells us that, for $n=3,4,5$, any finitely generated subgroup $\mathcal{A}\subseteq k^\times$ is contained in the norm group of a positive proportion of $S_n$-$n$-ic extensions. Moreover, it tells us that $\mathcal{A}$ can only be contained in 100\% of such norm groups if it is contained in \emph{every} such norm group. 

\begin{remark}
    In the special case where $\mathcal{A}$ is generated by one element, Corollary~\ref{cor-infinitely-many-S-n-n-ics-with-norm-group} can be obtained by classical methods, using Hilbert's irreducibility theorem. See \cite[Example~1.13]{newton-nf-prescribed} for details. For general subgroups $\mathcal{A}$, ours is the first proof we are aware of. 
\end{remark}

In order to prove Theorem~\ref{thm-prop-between-0-and-1-for-n-leq-5}, we expand the limit $\lim_{X\to\infty}\frac{N_{k,n}(X;\mathcal{A})}{N_{k,n}(X)}$ as an Euler product over all primes of the number field $k$. The factors of this Euler product are expressed in terms of a notion called \emph{mass}, which we now define. 

Throughout this paper, we use the term \emph{$p$-adic field} to mean a finite-degree extension of $\Q_p$, for some rational prime $p$. We use the term \emph{local field} to refer to a field that is either $p$-adic or isomorphic to $\R$ or $\C$. For a $p$-adic field $F$, write $\F_F$ for its residue field and $q_F$ for the size $\# \F_F$. For a local field $F$, write $\Et{n/F}{}$ for the set of isomorphism classes of degree $n$ \'etale algebras over $F$. For a subset $S\subseteq \Et{n/F}{}$, the \emph{mass} $m(S)$ of $S$ is defined as follows:
\begin{itemize}
    \item If $F$ is $p$-adic, then
$$
m(S) = \frac{q_F - 1}{q_F}\cdot \sum_{L \in S} \frac{\Nm(\disc(L/F))^{-1}}{\#\Aut(L/F)}.
$$
\item If $F$ is $\R$ or $\C$, then we define
$$
m(S) = \sum_{L \in S} \frac{1}{\#\Aut(L/F)}.
$$
\end{itemize}
Given a subgroup $\mathcal{A}\subseteq F^\times$, write $\Et{n/F}{\mathcal{A}}$ for the set of \'etale algebras $L \in \Et{n/F}{}$ with $\mathcal{A}\subseteq N_{L/F}L^\times$. For each prime $\p$ of $k$, we may view $\mathcal{A}$ as a subgroup of $k_\p^\times$, which allows us to write $m_{\mathcal{A},\p}$ for the mass 
$$
m_{\mathcal{A},\p} = m\big(\Et{n/k_\p}{\mathcal{A}}\big).
$$
Write $\Pi_k$ for the set of primes of $k$, both finite and infinite. 
\begin{theorem}
    \label{thm-density-as-product-of-masses}
    For $n\in\{3,4,5\}$, we have 
    $$
    \lim_{X\to\infty} \frac{N_{k,n}(X;\mathcal{A})}{X} = \frac{1}{2}\cdot \operatornamewithlimits{Res}_{s = 1}\big(\zeta_k(s)\big)\cdot \prod_{\p \in \Pi_k} m_{\mathcal{A},\p},
    $$
    and the same result holds conjecturally for $n \geq 6$, subject to the Malle--Bhargava heuristics.  
\end{theorem}

\subsection{Computing the masses $m_{\mathcal{A}, \p}$}
Theorem~\ref{thm-density-as-product-of-masses} tells us that we can state the density $\lim_{X\to\infty}\frac{N_{k,n}(X;\mathcal{A})}{X}$ as an explicit Euler product if we can evaluate the masses $m_{\mathcal{A},\p}$ for all primes $\p$. As such, we dedicate much of this paper to a careful study of these masses. The masses for infinite primes are the easiest to state: 
\begin{theorem}
    \label{thm-mass-infinite-prime}
    Let $\p$ be an infinite prime of $k$ and let $f: k\to \C$ be the corresponding embedding. The following two statements are true:
    \begin{itemize} \item If $f$ is a real embedding, then 
    $$
    m\big(\Et{n/k_\p}{\mathcal{A}}\big) = \begin{cases}
        \sum_{s=0}^{\lfloor\frac{n}{2}\rfloor} \frac{1}{s!(n-2s)! 2^s} \quad\text{if $f(\alpha) > 0 $ for all $\alpha \in \mathcal{A}$},
        \\
        \sum_{s=0}^{\lceil \frac{n}{2}\rceil - 1} \frac{1}{s!(n-2s)! 2^s} \quad\text{otherwise}.
    \end{cases}
    $$
    \item If $f$ is a complex embedding, then 
    $$
    m\big(\Et{n/k_\p}{\mathcal{A}}\big) = \frac{1}{n!}. 
    $$
\end{itemize}
\end{theorem}

Recall that we are most interested in the cases $n=3,4,5$, since these are the values of $n$ for which the Malle--Bhargava heuristics have been proved. We address these cases and more, studying $m_{\mathcal{A},\p}$ when $n = 4$ and whenever $n$ is an odd prime.

For $p$-adic fields $F$, the definition of mass has a normalising factor $\frac{q_F - 1}{q_F}$. To avoid this cumbersome notation, we define the \emph{pre-mass} of a set $S \subseteq \Et{n/F}{}$ to be 
$$
\m(S) = \sum_{L\in S}\frac{\Nm(\disc(L/F))^{-1}}{\#\Aut(L/F)}.
$$
Since it is trivial to compute the mass from the pre-mass, all remaining results in the current subsection will be about pre-mass.

Suppose that $F$ is a $p$-adic field. Let $L/F$ be an \'etale algebra, and write $L = L_1\times \ldots \times L_r$, where each $L_i$ is a field extension of $F$ with ramification index $e_i$ and inertia degree $f_i$. The \emph{splitting symbol} $(L,F)$ of $L$ over $F$ is the symbol $(f_1^{e_1}\ldots f_r^{e_r})$, which is defined up to permutation of indices (see Definition~\ref{defi-splitting-symbol} and the subsequent discussion for a more rigorous definition). The \emph{degree} of the splitting symbol $(f_1^{e_1}\ldots f_r^{e_r})$ is defined to be the quantity $\sum_i e_if_i$. Given a degree $n$ splitting symbol $\sigma$, write $\Et{\sigma/F}{\mathcal{A}}$ for the set of $L \in \Et{n/F}{\mathcal{A}}$ with $(L,F) = \sigma$. When $\sigma$ is a splitting symbol of the form $(f^e)$, and for a finite group $G$, write $\Et{\sigma/F}{G/F, \mathcal{A}}$ for the set of field extensions $L \in \Et{\sigma/F}{\mathcal{A}}$ such that the Galois closure $\widetilde{L}$ of $L$ over $F$ has $\Gal(\widetilde{L}/F)\cong G$. 

Our next result, Theorem~\ref{thm-m-A-p-for-prime-n}, gives $\m\big(\Et{n/F}{\mathcal{A}}\big)$ for all $p$-adic fields $F$, whenever $n$ is a rational prime. Consequently, for such $n$, it allows us to write down all the masses $m_{\mathcal{A},\p}$. In particular, this covers the cases $n=3$ and $n=5$, two of the three we are especially interested in. For a $p$-adic field $F$, a finitely generated subgroup $\mathcal{A}\subseteq F^\times$, and a positive integer $n$, define 
$$
\overline{\mathcal{A}}^n = \mathcal{A}F^{\times n}/F^{\times n}. 
$$
For each nonnegative integer $t$, write
$$
\overline{\mathcal{A}}^n_t = \overline{\mathcal{A}}^n\cap \Big(U_F^{(t)}F^{\times n} / F^{\times n}\Big),
$$
where $U_F^{(t)}$ is the $t^\mathrm{th}$ term of the unit filtration, given by $U_F^{(t)} = 1 + \p_F^t$. 
\begin{definition}
    \label{defi-strat-gen-set}
    Let $F$ be a $p$-adic field, let $n$ be a positive integer, and let $\overline{\mathcal{A}}$ be a subgroup of $F^\times / F^{\times n}$. A \emph{stratified generating set} for $\overline{\mathcal{A}}$ is a collection $(A_i)_{i \geq 0}$, indexed by nonnegative integers, where each $A_i$ is a subset of $F^\times$, such that the following three conditions hold:
    \begin{enumerate}
        \item For each $i$ and each $\alpha \in A_i$, we have $v_F(\alpha) = i$.
        \item For $\alpha,\beta \in \bigcup_i A_i$, if $\alpha \neq \beta$ then $[\alpha] \neq [\beta]$ in $F^\times / F^{\times n}$. 
        \item The image of $\bigcup_i A_i$ under the natural map $F^\times \to F^\times / F^{\times n}$ is a minimal generating set for the group $\overline{\mathcal{A}}$. 
    \end{enumerate}
\end{definition}
\begin{theorem}
    \label{thm-m-A-p-for-prime-n}
    Let $p$ be a rational prime, let $F$ be a $p$-adic field with residue field of size $q$, and let $\ell$ be a rational prime. Let $\mathcal{A}\subseteq F^\times$ be a finitely generated subgroup. Let $(A_0, A_1)$ be a stratified generating set for $\overline{\mathcal{A}}^\ell$. The following four statements are true:
    \begin{enumerate}
    \item We have 
    $$
    \m\big(\Et{\ell/F}{\mathcal{A}}\big) = \m\big(\Et{(\ell)/F}{\mathcal{A}}\big) + \m\big(\Et{(1^\ell)/F}{\mathcal{A}}\big) + \big(1 - \ell^{-1}\big) + \sum_{d=1}^{\ell - 2}\frac{\Part(d,\ell - d)}{q^d},
    $$
    where $\Part(d,m)$ is the partition function defined in Section~\ref{sec-general-n}. 
    \item We have
    $$
    \m\big(\Et{(\ell)/F}{\mathcal{A}}\big) = \begin{cases}
        \frac{1}{\ell}\quad\text{if $A_1 = \varnothing$},
        \\
        0\quad\text{otherwise}.
    \end{cases}
    $$
    \item Suppose that $p \neq \ell$. Then we have  
    $$
    \m\big(\Et{(1^\ell)/F}{\mathcal{A}}\big) = \begin{cases}
        \frac{1}{q^{\ell - 1}} \quad \text{if $q\not \equiv 1\pmod{\ell}$ or $\mathcal{A}\subseteq F^{\times \ell}$},
        \\
        \frac{1}{\ell q^{\ell - 1}} \quad\text{if $q\equiv 1\pmod{\ell}$ and $A_0 = \varnothing$ and $\# A_1 = 1$},
        \\
        0 \quad\text{otherwise}. 
    \end{cases}
    $$
    \item Suppose that $p = \ell$. Then we have
    $$
    \m\big(\Et{(1^p)/F}{\mathcal{A}}\big) = \frac{1}{q^{p-1}} - \m\big(\Et{(1^p)/F}{C_p/F}\big) + \m\big(\Et{(1^p)/F}{C_p/F,\mathcal{A}}\big).
    $$ 
    The two missing ingredients, $\m\big(\Et{(1^p)/F}{C_p/F}\big)$ and $\m\big(\Et{(1^p)/F}{C_p/F,\mathcal{A}}\big)$, are as follows. We have 
    $$
    \m\big(\Et{(1^p)/F}{C_p/F}\big) = \frac{q-1}{p-1}\cdot q^{-2}\cdot \Big(\mathbbm{1}_{e_F \geq p-1}\cdot A\Big(\Big\lceil\frac{pe_F}{p-1}\Big\rceil\Big) + \mathbbm{1}_{(p-1)\nmid e_F} \cdot B\Big(\Big\lceil \frac{pe_F}{p-1}\Big\rceil\Big)\Big) + \mathbbm{1}_{\mu_p \subseteq F} \cdot q^{-(p-1)(e_F+1)} ,
    $$
    where $A$ and $B$ are the explicit functions defined in Appendix~\ref{appendix-helpers}, and 
    $$
    \m\big(\Et{(1^p)/F}{C_p/F,\mathcal{A}}\big) = \mathbbm{1}_{\mu_p\subseteq F} \cdot \mathbbm{1}_{\overline{\mathcal{A}}^p_{\frac{pe_F}{p-1}} = 0}\cdot \frac{q^{-(p-1)(e_F+1)}}{\#\overline{\mathcal{A}}^p} +  \frac{1}{(p-1)\#\overline{\mathcal{A}}^p}\cdot q^{-2}\cdot \sum_{\substack{1\leq c \leq \lceil \frac{pe_F}{p-1}\rceil \\ c\not\equiv 1\pmod{p}}} \frac{q\cdot \#\overline{\mathcal{A}}^p_c - \#\overline{\mathcal{A}}^p_{c-1}}{q^{(p-2)c + \lfloor\frac{c-2}{p}\rfloor}}.
    $$
\end{enumerate}
\end{theorem}
Since the sizes $\#\overline{\mathcal{A}}^p_c$ can all be different, the series for $\m\big(\Et{(1^p)/F}{C_p/F,\mathcal{A}}\big)$ does not have a closed form. In the special case where $\mathcal{A}$ is generated by one element, we can write down such a closed form, which is stated in Theorem~\ref{thm-closed-form-for-m-alpha-p-with-prime-n}. Along the way, we find the exact number of extensions $L \in \Et{(1^p)/F}{C_p/F}$ with each possible discriminant, which we state in Theorem~\ref{thm-num-Cp-exts-with-disc-val}.

We now turn our attention to the final case of particular interest: $n=4$. When $\p$ is a prime of $k$ not lying over $2$, the pre-masses $\m\big(\Et{4/k_\p}{\mathcal{A}}\big)$ take a relatively simple, explicit form, stated in Theorem~\ref{thm-n=4-for-odd-primes}. 

When $\p \mid 2$, the situation is more difficult; since $4$ is not prime, wildly ramified quartic extensions behave significantly worse than wildly ramified prime-degree extensions, so our methods from the prime case do not generalise. However, since there are only finitely many primes $\p\mid 2$, it suffices to have a practicable method for computing $\m\big(\Et{4/k_\p}{\mathcal{A}}\big)$ for each such $\p$. In principle, this can be done by brute-force, since there are finitely many isomorphism classes of quartic \'etale algebras over $k_\p$, and we can find all of them e.g. using the methods of \cite{compute-extensions}. Unfortunately, \cite[Theorem~2]{krasner-num-exts} tells us that $\#\Et{4/k_\p}{}$ is on the order of $2^{3[k_\p : \Q_2]}$, which becomes very large very quickly, so the brute-force approach is unfeasible for all but very low-degree extensions $k_\p/\Q_2$. One of our main results is a significant improvement over this brute-force approach, whose time complexity we will state shortly, in Theorem~\ref{thm-time-complexity-mass-all-quartics-with-norms}; before the statement, we introduce some notation.

Let $F$ be a $p$-adic field for some rational prime $p$. In most computer algebra packages, $p$-adic fields are defined up to a certain \emph{precision}, which refers to the length of $p$-adic series representations stored. We will be working with field extensions $E/F$ with $[E : F] \leq 4$, and we need enough precision to construct the quotient $E^\times / E^{\times 4}$ of units modulo fourth powers. It follows that we need precision on the order of $e_F$. Suppose that we are working with precision $m$. In any sensible implementation, the computation time for a field operation in $F$ will be polynomial in $[F : \Q_p]$ and $m$. Since for our applications we have $m = O(e_F)$, a field operation in $F$ can be performed with time complexity $O(t([F:\Q_p]))$, for some polynomial function $t$. Moreover, since we are bounding the degree by a constant, any field operation in any extension $E/F$ with $[E : F] \leq 4$ can also be performed with time complexity $O(t([F:\Q_p]))$. Given a function $f$ in unspecified variables, we write 
$$
O_F(f) = O(f\cdot t([F:\Q_p])).
$$
Thus, the notation $O_F$ should be thought of as Big O notation, where we are suppressing the polynomial $t([F:\Q_2])$ to account for field operations in $F$, and quadratic and quartic extensions thereof. 

Let $\mathcal{G}_4(\mathcal{A})\subseteq k_\p^\times$ be a subset whose image in $k_\p^\times / k_\p^{\times 4}$ generates the group $\overline{\mathcal{A}}^4$. By the discussion above, the brute-force computation of $\m(\Et{4/k_\p}{\mathcal{A}})$ has time complexity of at least $O_{k_\p}(\#\mathcal{G}_4(\mathcal{A})\cdot 2^{3[k_\p : \Q_2]})$. We present the following improvements:
\begin{theorem}
    \label{thm-time-complexity-mass-all-quartics-with-norms}
    Let $\p$ be a prime of $k$ lying over $2$, and choose a set $\mathcal{G}_4(\mathcal{A})$ as above. There are two algorithms for computing $\m(\Et{4/k_\p}{\mathcal{A}})$, whose time complexities respectively are as follows:
    \begin{enumerate}
        \item $$O_{k_\p}\big(e_{k_\p}\cdot \#\mathcal{G}_4(\mathcal{A})\cdot 2^{2[k_\p:\Q_2]}\cdot [k_\p:\Q_2]^3\big).$$
        \item $$O_{k_\p}\big(e_{k_\p}\cdot 2^{\#\mathcal{G}_4(\mathcal{A})}\cdot 2^{[k_\p:\Q_2]}\cdot [k_\p:\Q_2]^3\big).$$
    \end{enumerate}
\end{theorem}
\begin{remark}
    The first time complexity in Theorem~\ref{thm-time-complexity-mass-all-quartics-with-norms} is unconditionally a big improvement over the brute-force algorithm. The second time complexity is a significantly bigger improvement, conditional on the number of generators of $\overline{\mathcal{A}}^4$ being kept reasonably small. Thus, the algorithms are both useful, and the choice between them will depend on the application. 
\end{remark}
We now discuss the ingredients of the algorithms in Theorem~\ref{thm-time-complexity-mass-all-quartics-with-norms}. For $\p\mid 2$, Theorem~\ref{thm-n=4-for-odd-primes} gives an explicit formula for $\m\big(\Et{\sigma/k_\p}{\mathcal{A}}\big)$, whenever $\sigma \not\in \{(1^21^2), (2^2), (1^4)\}$. Thus, we need only address these three cases. For $\sigma = (1^21^2)$ and $\sigma = (2^2)$, we state the required pre-masses in Theorems~\ref{thm-1^21^2-2-adic} and \ref{thm-2^2-2-adic}, respectively. 

The last hurdle is to compute $\m\big(\Et{(1^4)/k_\p}{\mathcal{A}}\big)$ for primes $\p \mid 2$. This final case is significantly harder, because there is so much wild ramification at play. We now give a brief overview of the results, which are stated in full technicality in Theorem~\ref{thm-1^4-2-adic}. We partition by Galois closure group $G$, considering the sets $\Et{(1^4)/k_\p}{G/k_\p,\mathcal{A}}$ separately for each $G$. For $G \in \{A_4,S_4\}$, we can state the pre-mass explicitly. For other $G$, we will instead consider the quantities 
$$
\# \Et{(1^4)/k_\p, m}{G/k_\p,\mathcal{A}} = \#\Big\{L \in \Et{(1^4)/k_\p}{G/k_\p,\mathcal{A}} : v_{k_\p}(d_{L/k_\p}) = m\Big\}
$$
for each possible $m$, where $d_{L/k_\p}$ is the discriminant ideal of the extension $L/k_\p$. From these quantities, it is trivial to compute the pre-mass by summing over $m$. We give explicit expressions for $\# \Et{(1^4)/k_\p, m}{G/k_\p,\mathcal{A}}$ in the cases $G \in \{D_4, V_4\}$.

So far, every quantity we have stated is reasonably explicit, and we have not yet bifurcated into the two algorithms promised in Theorem~\ref{thm-time-complexity-mass-all-quartics-with-norms}. It is only the final piece, the pre-mass of wildly ramified $C_4$-extensions, that requires genuine algorithms to compute. Given a quadratic extension $E/k_\p$, write $\Et{2/E}{C_4/k_\p}$ for the set of quadratic field extensions $L/E$ such that $L/k_\p$ is a $C_4$-extension. We say that a quadratic extension $E/k_\p$ is \emph{$C_4$-extendable} if $\Et{2/E}{C_4/k_\p}\neq \varnothing$.  Given a $C_4$-extendable extension $E/k_\p$, the set $\Et{2/E}{C_4/k_\p}$ can be parametrised by the quotient group $k_\p^\times / k_\p^{\times 2}$. Under this parametrisation, the set $\Et{2/E}{C_4/k_\p,\mathcal{A}}$, which consists of the extensions $L \in \Et{2/E}{C_4/k_\p}$ with $\mathcal{A}$ in their norm group, corresponds to a certain subset $N_E^\mathcal{A}$ of $k_\p^\times / k_\p^{\times 2}$. In particular, we can easily express the quantities $\#\Et{(1^2)/E,m_2}{C_4/k_\p,\mathcal{A}}$ in terms of the sizes of the subsets 
$$
N_{E,c}^\mathcal{A} = N_E^\mathcal{A} \cap\big(U_{k_\p}^{(c)}k_\p^{\times 2}/k_\p^{\times 2}\big),
$$ 
for $-1 \leq c \leq 2e_{k_\p}$. When $[k_\p : \Q_2]$ is small, the set $k_\p^\times / k_\p^{\times 2}$ is small, so we can compute the sets $N_{E,c}^\mathcal{A}$ by brute-force, iterating over the $2^{2+[k_\p : \Q_2]}$ elements of $k_\p^\times / k_\p^{\times 2}$. This brute-force method yields the first time complexity in Theorem~\ref{thm-time-complexity-mass-all-quartics-with-norms}. On the other hand, if $[k_\p : \Q_2]$ is large, we can use a more sophisticated algorithm whose running time is dominated by $2^{\#\mathcal{G}_4(\mathcal{A})}$. This more sophisticated method gives the second time complexity in Theorem~\ref{thm-time-complexity-mass-all-quartics-with-norms}. Once the sizes $\# N_{E,c}^\mathcal{A}$ have been computed, we can quickly write down the pre-mass 
$$
\m\big(\big\{L \in \Et{(1^4)/k_\p}{C_4/k_\p,\mathcal{A}} : \text{$E$ embeds into $L$}\big\}\big).
$$
Iterating over the $2^{2+[k_\p:\Q_2]}$ quadratic extensions $E$ of $k_\p$, we can compute the pre-mass of $\Et{(1^4)/k_\p}{C_4/k_\p,\mathcal{A}}$, and hence of $\Et{4/k_\p}{\mathcal{A}}$, yielding Theorem~\ref{thm-time-complexity-mass-all-quartics-with-norms}.

\subsection{Acknowledgements}

I would like to thank my supervisor, Rachel Newton, for suggesting the initial research direction and for providing a great deal of support and encouragement throughout. I am also grateful to Brandon Alberts, Christopher Keyes, Ross Paterson, Ashvin Swaminathan, John Voight, and Jiuya Wang for helpful conversations. 

This work was supported by the Engineering and Physical Sciences Research Council [EP/S021590/1], via the EPSRC Centre for Doctoral Training in Geometry and Number Theory (The London School of Geometry and Number Theory), University College London. 

\subsection{Index of notation}

\begin{enumerate}
    \item We will always use the letters $K/k$ for extensions of number fields, and $L/F$ for \'etale algebras over $p$-adic fields.
    \item Let $\mathcal{E}$ be a number field or $p$-adic field, and let $\mathcal{M}/\mathcal{E}$ be a finite \'etale algebra. We write
    \begin{enumerate}
        \item $\mathcal{O}_\mathcal{E}$ for the ring of integers of $\mathcal{E}$. 
        \item $[\mathcal{M} : \mathcal{E}]$: The dimension of $\mathcal{M}$ as an $\mathcal{E}$-algebra. 
        \item $\disc(\mathcal{M}/\mathcal{E})$: The discriminant of the extension, which is an ideal of $\mathcal{O}_\mathcal{E}$. 
        \item $d_{\mathcal{M}/\mathcal{E}}$: Shorthand for $\disc(\mathcal{M}/\mathcal{E})$.
        \item $N_{\mathcal{M}/\mathcal{E}}$: The norm map $\mathcal{M}^\times \to \mathcal{E}^\times$. 
        \item $\widetilde{\mathcal{M}}$: The normal closure of $\mathcal{M}$ over $\mathcal{E}$, where the choice of base field $\mathcal{E}$ will be clear from context. Note that this is only defined when $\mathcal{M}$ is a field. 
        \item $\Aut(\mathcal{M}/\mathcal{E})$: The group of $\mathcal{E}$-algebra automorphisms of $\mathcal{M}$. 
        \item $\Gal(\mathcal{M}/\mathcal{E})$: The Galois group of the extension $\mathcal{M}/\mathcal{E}$, if it is a Galois extension of fields. 
    \end{enumerate}
    \item $\Nm$: This can represent two things:
    \begin{itemize}
        \item We write $\Nm(I)$ for the norm $\#(\co_{\mathcal{M}}/I)$ of an integral ideal $I$ of a $p$-adic field or number field $\mathcal{M}$.
        \item We write $\Nm(\mathcal{M})$ for the norm group $N_{\mathcal{M}/\mathcal{E}}\mathcal{M}^\times$ of an \'etale algebra $\mathcal{M}$ over a $p$-adic field or number field $\mathcal{E}$. When we use this shorthand, the base field will be clear from context, so we will omit it from the notation. 
    \end{itemize}
    \item $k$: A fixed number field. 
    \item $K_\p$: The completion $K\otimes_k k_\p$ of a finite-degree field extension $K/k$ over a prime $\p$ of $k$.
    \item $\mathcal{A}$: A subgroup of the unit group of a number field or local field. 
    \item $N_{k,n}(X)$: The number of $S_n$-$n$-ic extensions $K/k$ with $\Nm(\disc(K/k)) \leq X$. 
    \item $N_{k,n}(X;\mathcal{A})$: The number of $S_n$-$n$-ic extensions $K/k$ with $\Nm(\disc(K/k))\leq X$ and $\mathcal{A}\subseteq N_{K/k}K^\times$. 
    \item For a $p$-adic field $F$, we write:
    \begin{enumerate}
        \item $\p_F$ for the maximal ideal of $\co_F$.
        \item $\F_F$ for the residue field $\co_F/\p_F$ of $F$. 
        \item $q_F$ for the size $\# \F_F$ of the residue field. 
        \item $\pi_F$ for a uniformiser of $F$.
    \end{enumerate}
    \item $(L,F)$: The splitting symbol of an \'etale algebra $L$ over a $p$-adic field $F$, defined in Section~\ref{subsec-etale-algebras}.
    \item $\Et{n/F}{}$: For a local field $F$, this denotes the set of isomorphism classes of degree $n$ \'etale algebras over $F$. There is a variety of decorators that impose additional conditions on these \'etale algebras. The possible decorators are:
    \begin{enumerate}
    \item $\Et{n/F,m}{}$: The set of $L \in \Et{n/F}{}$ such that $v_F(d_{L/F}) = m$. 
    \item $\Et{n/F,\leq m}{}$: The set of $L \in \Et{n/F}{}$ such that $v_F(d_{L/F}) \leq m$. 
    \item $\Et{n/F}{\pred}$: The set of $L\in \Et{n/F}{}$ such that $(L,F) = (f_1^{e_1}\ldots f_r^{e_r})$ for mutually coprime $e_i$. 
    \item $\Et{n/F}{\unpred}$: The complement $\Et{n/F}{} \setminus \Et{n/F}{\pred}$.
    \item $\Et{n/F}{\epi}$: The set of $L\in \Et{n/F}{}$ such that $(L,F) = (f_1^{e_1}\ldots f_r^{e_r})$ for mutually coprime $e_i$ and mutually coprime $f_i$. 
    \item $\Et{n/F}{\mathcal{A}}$: The set of $L \in \Et{n/F}{}$ such that $\mathcal{A}\subseteq N_{L/F}L^\times$.
    \end{enumerate}
    If multiple decorators are present, then we take the intersection of the corresponding sets of \'etale algebras. 
    \item $\Et{\sigma/F}{}$: For a splitting symbol $\sigma$, this is the set of isomorphism classes of \'etale algebras $L/F$ with $(L,F) = \sigma$. This has all the same decorators as $\Et{n/F}{}$, as well as the following additional one:
    \begin{enumerate}
        \item $\Et{\sigma/E}{G/F}$: This is only defined if $\sigma$ is of the form $(f^e)$, so that the elements of $\Et{\sigma/E}{}$ are fields. In this case, $G$ is a finite group, $F$ is still a local field and $E/F$ is a finite field extension. Then $\Et{\sigma/E}{G/F}$ is the set of $L \in \Et{\sigma/E}{}$ such that $\Gal(\widetilde{L}/F) \cong G$, where $\widetilde{L}$ is the normal closure of $L$ over $F$. 
    \end{enumerate}
    \item $m(S)$: The mass of a subset $S$ of $\Et{n/F}{}$.
    \item $\widetilde{m}(S)$: The pre-mass of a subset $S$ of $\Et{n/F}{}$.
    \item $m_{\mathcal{A},\p}$: The mass $m\big(\Et{n/k_\p}{\mathcal{A}}\big)$, where $\p$ is a prime of $k$ and $\mathcal{A}$ is viewed as a subgroup of $k_\p$ via the natural inclusion $k \subseteq k_\p$. 
    \item $\Res_{s = 1}\big(\zeta_k(s)\big)$: The residue of the Dedekind zeta function of $k$ at $1$. 
    \item $\overline{\mathcal{A}}^n$: For a $p$-adic field $F$ and a subgroup $\mathcal{A}\subseteq F^\times$, this is the subgroup $\mathcal{A}F^{\times n} / F^{\times n}$ of $F^\times / F^{\times n}$. 
    \item $U_F^{(t)}$: For a $p$-adic field $F$ and a nonnegative integer $t$, this is $1 + \p_F^t$.
    \item $\mathcal{G}_4(\mathcal{A})$: A subset of $F^\times$ whose image in $F^\times / F^{\times 4}$ generates the group $\overline{\mathcal{A}}^4$. We tend to assume that a choice of $\mathcal{G}_4(\mathcal{A})$ is fixed implicitly. 
    \item $O_F$: Modified Big O notation that suppresses the time complexity of field operations in the $p$-adic field $F$. 
    \item $\overline{\mathcal{A}}_t^n$: For a $p$-adic field $F$, a subgroup $\mathcal{A}\subseteq F^\times$, and a nonnegative integer $t$, this is the intersection
    $$
    \overline{\mathcal{A}}^n \cap \Big(U_F^{(t)}F^{\times n} / F^{\times n}\Big)
    $$
    \item $\ell$: A rational prime. 
    \item $\Part(d,m)$: The partition function, as in Definition~\ref{defi-partition}.
    \item $\Split_n$: Set of degree $n$ splitting symbols.
    \item $\Split_n^{\pred}$: Set of predictable degree $n$ splitting symbols. 
    \item $\Split_n^{\epi}$: Set of epimorphic degree $n$ splitting symbols. 
    \item $\Ext{k,n}$: The set of isomorphism classes of $S_n$-$n$-ic extensions $K/k$. 
    \item $\Ext{k,n,\leq X}$: The set of $K \in \Ext{k,n, \leq X}$ with $\Nm(\disc(K/k)) \leq X$. 
    \item $N_{k,n}(X; \Sigma)$: The number of fields $K \in \Ext{k,n,\leq X}$ such that $K$ satisfies the family of local conditions $\Sigma$.
    \item $\sum \pi$: For a partition $\pi = (a_i)$, this is the sum $\sum_i a_i$. 
    \item For a splitting symbol $\sigma = (f_1^{e_1}\ldots f_r^{e_r})$, we write:
    \begin{enumerate}
        \item $\pi(\sigma)$ for the partition associated to $\sigma$.
        \item $d_\sigma$ for the integer $\sum_i f_i(e_i - 1)$. 
        \item $\#\Aut(\sigma)$ for the quantity 
        $$
        \Big(\prod_i f_i\Big) \cdot \# \{\sigma \in S_n : (e_{\sigma(i)}, f_{\sigma(i)}) = (e_i, f_i) \text{ for all $i$}\}.
        $$
    \end{enumerate}
    \item $k(\mu_m)$: The extension obtained from $k$ by adjoining all $m^\mathrm{th}$ roots of unity. 
    \item $\mu_m \subseteq F$: For a field $F$, this is shorthand for the statement that $F$ contains all $m^\mathrm{th}$ roots of unity.
    \item $\Pi_k$: The set of primes of $k$, both finite and infinite. 
    \item $\mff(L/F)$: The conductor of an extension $L/F$ of $p$-adic fields. 
    \item $\mff(\chi)$: The Artin conductor of a character $\chi$. 
    \item $G_i$: The $i^\mathrm{th}$ ramification subgroup of a Galois group $G$. 
    \item $G^\vee$: The set of characters $G\to \C^\times$, for a Galois group $G$.  
    \item $\operatorname{Epi}_{\F_p}(V, W)$: The set of $\F_p$-linear epimorphisms $V \to W$, where $V,W$ are $\F_p$-vector spaces. 
    \item $W_i$: Given an implicit choice of $p$-adic field $F$, this is the group $U_F^{(i)}F^{\times p} / U_F^{(i+1)}F^{\times p}$. 
    \item $A(t)$ and $B(t)$: Explicit functions defined in Appendix~\ref{appendix-helpers}. 
    \item $\sqcup$: Disjoint union.
    \item $c_\alpha$: Largest integer such that $\alpha \in U_F^{(c)}F^{\times p}$. If $\alpha \in F^{\times p}$, then $c_\alpha = \infty$.
    \item $N_\alpha$: For $\alpha \in F^\times \setminus F^{\times 2}$, this is the group $N_{F(\sqrt{\alpha})/F}F(\sqrt{\alpha})^\times$. 
    \item Given a $C_4$-extendable quadratic extension $E/F$, we have the following notation:
    \begin{enumerate}
    \item $N_\omega$: Given an element $\omega \in E$ such that $E(\sqrt{\omega})/F$ is a $C_4$-extension, this is the norm group $N_{E(\sqrt{\omega})/F}E(\sqrt{\omega})^\times$. 
    \item $N_\omega^\mathcal{A}$: The set 
    $$
    \bigcap_{\alpha \in \mathcal{G}_4(\mathcal{A}) \cap N_\omega} \overline{N}_\alpha^2 \setminus \bigcup_{\alpha \in \mathcal{G}_4(\mathcal{A}) \setminus N_\omega} \overline{N}_\alpha^2 \subseteq F^\times / F^{\times 2}.
    $$
    \item $N_E^\mathcal{A}$: The set $N_\omega^{\mathcal{A}}$ for a choice of element $\omega \in E$ such that the $C_4$-extension $E(\sqrt{\omega})/E$ has minimal discriminant.
    \item $N_{E,c}^\mathcal{A}$: The intersection 
    $
    N_E^\mathcal{A} \cap \big(U_F^{(c)}F^{\times 2}/F^{\times 2}\big).
    $
    \end{enumerate}
    \item $\mathcal{B}_0^E$: For a $p$-adic field $E$, this is a certain subset of $\co_E^\times$ that descends to an $\F_p$-basis for $\F_E$.
    \item $\operatorname{colspan}(M)$: The column span of a matrix $M$.
\end{enumerate}
\section{Results for general $n$}

The main goal of this section is to prove Theorem~\ref{thm-prop-between-0-and-1-for-n-leq-5}, along with its conjectural generalisations in Theorems~\ref{thm-prop-between-0-and-1} and \ref{thm-hasse-for-nth-powers}.

\label{sec-general-n}
\subsection{\'Etale algebras, splitting symbols, and local conditions}
\label{subsec-etale-algebras}
Let $F$ be a local field. Recall from Section~\ref{sec-intro} that $\Et{n/F}{}$ is the set of isomorphism classes of degree $n$ \'etale algebras over $F$, and that we are interested in the masses and pre-masses of subsets $S \subseteq \Et{n/F}{}$, denoted by $m(S)$ and $\m(S)$, respectively.

\begin{proof}[Proof of Theorem~\ref{thm-mass-infinite-prime}]
    Suppose that $f$ is a real embedding. Then 
    $$
    \Et{n/k_\p}{} = \{\R^r \times \C^s : r + 2s = n\text{ and $r,s \geq 0$}\}.  
    $$
    If $L = \R^r \times \C^s$, we have 
    $$
    N_{L/k_\p}L^\times = \begin{cases}
        \R^\times \quad\text{if $r > 0$},
        \\
        \R_{> 0}\quad\text{if $r = 0$},
    \end{cases}
    $$
    and 
    $$
    \# \Aut(L/k_\p) = 2^s \cdot r!\cdot s!.
    $$
    The result for real $f$ follows. If $f$ is a complex embedding, then the only element of $\Et{n/k_\p}{}$ is $\C^n$, which has $n!$ automorphisms and norm group $\C^\times$, so we are done.
\end{proof}
\begin{definition}
    \label{defi-splitting-symbol}
    Let $n$ be a positive integer. A \emph{degree $n$ splitting symbol} is a symbol $(f_1^{e_1}\ldots f_r^{e_r})$, where the $f_i$ and $e_i$ are positive integers. We identify splitting symbols that are permutations of each other, i.e. $(1^22^3) = (2^31^2)$. The superscripts are purely symbolic, and do not represent exponentiation, so for example $(1^2)$ and $(1^3)$ are different splitting symbols. Finally, we suppress exponents with value $1$, writing e.g. $(2)$ instead of $(2^1)$. Write $\Split_n$ for the set of degree $n$ splitting symbols. 
\end{definition}

Let $F$ be a local field and let $L \in \Et{n/F}{}$. Then we have 
$$
L = L_1 \times \ldots \times L_r,
$$
for field extensions $L_i/F$ with inertia degree $f_i$ and ramification index $e_i$. The \emph{splitting symbol $(L,F)$ of $L$ over $F$} is defined to be the splitting symbol
$$
(L,F) = (f_1^{e_1}\ldots f_r^{e_r}). 
$$
For a splitting symbol $\sigma \in \Split_n$, write 
$
\Et{\sigma/F}{}
$
for the set of $L \in \Et{n/F}{}$ with $(L,F) = \sigma$. 

\begin{definition}
    We introduce our key terminology for local conditions.
    \begin{enumerate}
    \item Let $k$ be a number field and let $\p \in \Pi_k$. A \emph{degree $n$ local condition at $\p$} is a set $\Sigma_\p \subseteq \Et{n/k_\p}{}$. 
    \item A \emph{degree $n$ collection of local conditions on $k$} is a collection $(\Sigma_\p)_\p$ of local conditions at each prime $\p\in \Pi_k$. 
    \item Let $\Sigma$ be a degree $n$ collection of local conditions on $k$. We call $\Sigma$ \emph{acceptable} if, for all but finitely many $\p$, the set $\Sigma_\p$ contains all \'etale algebras $L/k_\p$ with $v_{k_\p}(d_{L/k_\p}) \leq 1$. 
    \item Let $K/k$ be a finite field extension and let $\Sigma_\p$ be a local condition at the prime $\p$ of $k$. The extension $K/k$ \emph{satisfies} the condition $\Sigma_\p$ if the completion $K_\p := K\otimes_k k_\p$ is in $\Sigma_\p$. 
    \item The extension $K/k$ \emph{satisfies} a collection $\Sigma = (\Sigma_\p)_\p$ of local conditions if it satisfies $\Sigma_\p$ for every prime $\p \in \Pi_k$.
    \end{enumerate}
\end{definition}
Write $\Ext{k,n}$ for the set of isomorphism classes of degree $n$ extensions $K/k$ with Galois closure group $S_n$. For each real number $X$, define 
$$
\Ext{k, n,\leq X} = \{K \in \Ext{k, n} : \Nm(\disc(K/k)) \leq X\}.
$$
For a system $\Sigma$ of local conditions, define 
$$
N_{k,n}(X;\Sigma) = \#\{K \in \Ext{k, n,\leq X} : \text{$K$ satisfies $\Sigma$}\}.
$$

\begin{conjecture}[Malle--Bhargava heuristics]
\label{conj-N(X;Sigma)-asymptotics}
Let $\Sigma = (\Sigma_\p)_\p$ be an acceptable degree $n$ collection of local conditions on $k$. Then 
$$
\lim_{X\to\infty} \frac{N_{k,n}(X;\Sigma)}{X} = \frac{1}{2}\cdot \Res_{s=1}\big(\zeta_k(s)\big)\cdot \prod_{\p \in \Pi_k} m(\Sigma_\p).
$$
\end{conjecture}
\begin{theorem}
    \label{thm-malle-bhargava-true-for-n-leq-5}
    Conjecture~\ref{conj-N(X;Sigma)-asymptotics} is true for $n=2,3,4,5$. 
\end{theorem}
\begin{proof}
    This is \cite[Theorem~2]{geom-of-nums}. 
\end{proof}

\subsection{Expressing our problem in terms of local conditions}

Let $F$ be a local field and let $\mathcal{A}\subseteq F^\times$ be a finitely generated subgroup. Write $\Et{n/F}{\mathcal{A}}$ for the set of $L \in \Et{n/F}{}$ with $\mathcal{A}\subseteq N_{L/F}L^\times$. Define $\Et{\sigma/F}{\mathcal{A}}$ analogously.

Recall that in our original setting, $k$ is a number field and $\mathcal{A}\subseteq k^\times$ is a finitely generated subgroup. For each prime $\p$ of $k$, we view $\mathcal{A}$ naturally as a subgroup of $k_\p^\times$, and therefore it makes sense to talk about the set $\Et{n/k_\p}{\mathcal{A}}$. In particular, we have a system of local conditions $\big(\Et{n/k_\p}{\mathcal{A}}\big)_\p$. 

The following theorem says that for an extension $K/k$, the elements of $\mathcal{A}$ are norms globally if and only if they are norms everywhere locally.

\begin{theorem}[Hasse Norm Principle holds for $S_n$-$n$-ics]
    \label{thm-hasse-norm-principle}
    Let $n$ be a positive integer, let $K/k$ be an $S_n$-$n$-ic extension of number fields, and let $\mathcal{A}\subseteq k^\times$ be a finitely generated subgroup. Then $\mathcal{A} \subseteq N_{K/k}K^\times$ if and only if $K/k$ satisfies the system $\big(\Et{n/k_\p}{\mathcal{A}}\big)_\p$ of local conditions. 
\end{theorem}
\begin{proof}
    This is \cite[Corollary to Theorem~4]{Voskresenski}.
\end{proof}

\begin{lemma}
    \label{lem-disc-val-of-tame-extension}
    Let $E/F$ be a tamely ramified extension of $p$-adic fields, with ramification index $e$ and inertia degree $f$. Then 
    $$
    v_F(d_{E/F}) = f(e-1). 
    $$
\end{lemma}
\begin{proof}
    This follows easily from \cite[Page~199, Theorem~2.6]{neukirch2013algebraic} and \cite[Page~201, Theorem~2.9]{neukirch2013algebraic}. 
\end{proof}

\begin{lemma}
    \label{lem-norm-conds-acceptable}
    Let $k$ be a number field, let $n$ be an integer with $n \geq 3$, and let $\mathcal{A}\subseteq k^\times$ be a finitely generated subgroup. Then the system $\big(\Et{n/k_\p}{\mathcal{A}}\big)_\p$ of local conditions is acceptable.
\end{lemma}
\begin{proof}
    Let $\p$ be a finite prime such that $\gcd(\Nm(\p), n) = 1$ and $v_\p(\alpha)=0$ for all $\alpha\in \mathcal{A}$. Note that all but finitely many primes satisfy these conditions. Let $L \in \Et{n/k_\p}{}$ such that $v_{k_\p}(d_{L/k_\p}) \leq 1$. Writing $L = L_1\times\ldots \times L_r$ for field extensions $L_i/k_\p$, at most one of the extensions $L_i/k_\p$ is ramified, so either $L = L_1$ or $L_i/k_\p$ is unramified for some $i$. Suppose that $L = L_1$. Let $(L,k_\p) = (f^e)$. Since $\Nm(\p)$ is coprime to $n$, it follows that $L/k_\p$ is tamely ramified, so Lemma~\ref{lem-disc-val-of-tame-extension} tells us that $v_{k_\p}(d_{L/k_\p}) = f(e-1)$, and therefore $f = 1$ and $e \leq 2$, so $n \leq 2$, contradicting the assumption that $n \geq 3$. Therefore, $L_i/k_\p$ is unramified for some $i$, so $\co_{k_\p}^\times \subseteq N_{L/k_\p}L^\times$, and hence $L \in \Et{n/k_\p}{\mathcal{A}}$. The result follows.
\end{proof}

\begin{proof}
    [Proof of Theorem~\ref{thm-density-as-product-of-masses}]

    This is immediate from Theorem~\ref{thm-malle-bhargava-true-for-n-leq-5}, Theorem~\ref{thm-hasse-norm-principle}, and Lemma~\ref{lem-norm-conds-acceptable}.
\end{proof}

Call a splitting symbol $(f_1^{e_1}\ldots f_r^{e_r})$ \emph{predictable} if the integers $e_i$ are mutually coprime. Call a symbol $(f_1^{e_1}\ldots f_r^{e_r})$ \emph{epimorphic} if it is predictable and the integers $f_i$ are mutually coprime. For a $p$-adic field $F$, write $\Et{n/F}{\pred}$ and $\Et{n/F}{\epi}$ for the sets of $L \in \Et{n/F}{}$ such that $(L,F)$ is predictable and epimorphic, respectively. Write $\Split_n^{\pred}$ and $\Split_n^{\epi}$ for the sets of predictable and epimorphic degree $n$ splitting symbols, respectively.

\begin{lemma}
    \label{lem-norm-group-of-pred-symbol}
    Let $F$ be a $p$-adic field and let $\sigma$ be a predictable splitting symbol. Write $\sigma = (f_1^{e_1}\ldots f_r^{e_r})$ and let $L \in \Et{\sigma/F}{}$. Let $g = \gcd(f_1,\ldots, f_r)$. Then we have
    $$
    N_{L/F}L^\times = \{x \in F^\times : g\mid v_F(x)\}. 
    $$
    In particular, if $L \in \Et{n/F}{\epi}$, then 
    $$
    N_{L/F}L^\times = F^\times. 
    $$
\end{lemma}
\begin{proof}
    Write $L = L_1 \times \ldots \times L_r$, where each $L_i/F$ is a field extension with ramification index $e_i$ and inertia degree $f_i$. For each $i$, let $L_i^{\mathrm{ur}}$ be the maximal unramified subextension of $L_i/F$. By considering the towers $L_i/L_i^{\mathrm{ur}}/F$, it is easy to see that we have 
    $$
    \co_F^{\times e_i}\subseteq N_{L_i/F}L_i^\times 
    $$
    for each $i$. Moreover, for each $i$, there is some $x_i \in N_{L_i/F}L_i^\times$ with $v_F(x_i) = f_i$. Since the $e_i$ are mutually coprime, there exist integers $r_i$ such that $\sum_i r_i e_i = 1$, so for each $\alpha \in \co_F^\times$, we have 
    $$
    \alpha = \prod_i (\alpha^{e_i})^{r_i} \in \prod_i N_{L_i/F}L_i^\times = N_{L/F}L^\times. 
    $$
    Therefore, $\co_F^\times \subseteq N_{L/F}L^\times$. On the other hand, there are integers $s_i$ such that $\sum_i s_if_i = g$. Setting
    $$
    x = \prod_i x_i^{s_i} \in N_{L/F}L^\times,
    $$
    we obtain $v_F(x) = g$, so 
    $$
    \{x \in F^\times : g\mid v_F(x)\} \subseteq N_{L/F}L^\times. 
    $$
    Conversely, we have
    $$
    N_{L_i/F}L_i^\times \subseteq N_{L_i^{\mathrm{ur}}/F} L_i^{\mathrm{ur}, \times} = \{x \in F^\times : f_i\mid v_F(x)\} \subseteq \{x \in F^\times : g\mid v_F(x)\}
    $$
    for each $i$, so we are done. 
\end{proof}

\begin{corollary}
    \label{cor-Et-pred-subset-Et-A}
    Let $F$ be a $p$-adic field and let $\mathcal{A} = \langle \alpha_1,\ldots, \alpha_d\rangle$ be a finitely generated subgroup of $F^\times$. Let $n$ be a positive integer and suppose that $n \mid v_F(\alpha_i)$ for each $i$. For every finite prime $\p$ of $k$, we have 
    $$
    \Et{n/k_\p}{\pred}\subseteq \Et{n/k_\p}{\mathcal{A}}.
    $$
\end{corollary}
\begin{proof}
    This is immediate from Lemma~\ref{lem-norm-group-of-pred-symbol}. 
\end{proof}
\begin{lemma}
    \label{lem-overram-O-q^{-2}}
    Let $n$ be an integer with $n \geq 3$, let $p$ be a rational prime with $p > n$, let $F$ be a $p$-adic field, and let $L \in \Et{n/F}{}\setminus \Et{n/F}{\pred}$. Then $v_F(d_{L/F}) \geq 2$. 
\end{lemma}
\begin{proof}
    Let $(L,F) = (f_1^{e_1}\ldots f_r^{e_r})$. Since $p > n$, we have $p\nmid e_i$ for each $i$, so by Lemma~\ref{lem-disc-val-of-tame-extension} we have 
    $$
    v_F(d_{L/F}) = \sum_i f_i(e_i-1). 
    $$
    Since $(L,F)$ is not predictable, the $e_i$ are not mutually coprime, so there is some $d \geq 2$ with $d \mid e_i$ for all $i$, and therefore $e_i \geq 2$ for each $i$. If any of the integers $e_i$ is at least $3$, then we are done. On the other hand, if $e_i = 2$ for every $i$, then
    $$
    v_F(d_{L/F}) = \sum_i f_i = \frac{n}{2} > 1,
    $$
    so we are done. 
\end{proof}
\begin{definition}[Partitions]
    \label{defi-partition}
Let $d$ be a nonnegative integer and let $m$ be a positive integer. The symmetric group $S_m$ acts by permutation of coordinates on the set $\Z_{\geq 0}^m$ of nonnegative integers. We define a \emph{partition of $d$ into $m$ parts} to be an equivalence class 
$$
[(a_i)] \in \Z_{\geq 0}^m / S_m,
$$
such that $\sum_i a_i = d$. This is often referred to as a partition into ``at most'' $m$ parts, but for our purposes it is more convenient to allow parts to be zero. Write $\Part(d, m)$ for the number of partitions of $d$ into $m$ parts. Given a partition 
$$
\pi = (a_1,\ldots, a_m)
$$
of $d$ into $m$ parts, write $\sum \pi$ for the sum $\sum_i a_i$. 
\end{definition}
Let $\sigma = (f_1^{e_1}\ldots f_r^{e_r})$ be a degree $n$ splitting symbol with 
$$
\sum_i f_i(e_i - 1) = d. 
$$
Then we obtain a partition 
$$
\pi(\sigma) := (e_1 - 1, \ldots, e_1 - 1, e_2 - 1,\ldots, e_2 - 1, \ldots, e_r - 1, \ldots, e_r - 1),
$$
of $d$ into $n - d$ parts, where each term $e_i - 1$ appears $f_i$ times. For a splitting symbol $\sigma = (f_1^{e_1}\ldots f_r^{e_r})$, define 
$$
d_\sigma := \sum_i f_i(e_i - 1)
$$
and 
$$
\#\Aut(\sigma) = \Big(\prod_i f_i \Big)\cdot \#\{\sigma \in S_n : (e_{\sigma(i)},f_{\sigma(i)}) = (e_i, f_i)\text{ for all $i$}\}.
$$
Note that $d_\sigma$ can be read off from $\pi(\sigma)$. Recall that, given a $p$-adic field $F$, we write $q$ as shorthand for the size $q_F$ of the residue field of $F$. 

\begin{lemma}
    \label{lem-num-exts-with-ramification-partition}
    Let $F$ be a $p$-adic field for some rational prime $p$. Let $n$ and $d$ be integers with $0 \leq d \leq n - 1$, and let $\pi_0$ be a partition of $d$ into $n-d$ parts. We have 
    $$
    \m\big(\{L \in \Et{n/F}{} : \pi(L,F) = \pi_0\}\big) = \frac{1}{q^d}. 
    $$
\end{lemma}
\begin{proof}
    In the case $F = \Q_p$, this is \cite[Proposition~2.2]{bhargava-mass-formula}. The appendix of \cite{bhargava-mass-formula} explains how to extend the proof to arbitrary $F$.
\end{proof}
\begin{corollary}
    \label{cor-mass-of-etale-algebras-with-given-symbol-disc}
    Let $F$ be a $p$-adic field. Let $n$ and $d$ be integers with $0 \leq d \leq n - 1$. Then 
    $$
    \m\big(\{L \in \Et{n/F}{} : d_{(L,F)} = d\}\big) = \frac{\Part(d,n-d)}{q^d}.
    $$
\end{corollary}
\begin{proof}
    This follows easily from Lemma~\ref{lem-num-exts-with-ramification-partition}.
\end{proof}

We also record the following lemma, which will be useful later:
\begin{lemma}
    \label{lem-mass-of-all-algebras-with-symbol}
    Let $n$ be a positive integer, let $F$ be a $p$-adic field, and let $\sigma \in \Split_n$. We have 
    $$
    \m\big(\Et{\sigma/F}{}\big) = \frac{1}{q^{d_\sigma}} \cdot \frac{1}{\#\Aut(\sigma)}.
    $$
\end{lemma}
\begin{proof}
    As with Lemma~\ref{lem-num-exts-with-ramification-partition}, the case $F = \Q_p$ is \cite[Proposition~2.1]{bhargava-mass-formula}. For our statement, the necessary modifications to the proof are explained in the appendix of \cite{bhargava-mass-formula}. 
\end{proof}
    
\begin{lemma}
    \label{lem-positive-prop}
    For $n \geq 3$, we have  
    $$
    \prod_{\p \in \Pi_k} \frac{\m\big(\Et{n/k_\p}{\mathcal{A}}\big)}{\m\big(\Et{n/k_\p}{}\big)} > 0.
    $$
\end{lemma}
\begin{proof}
    Since the factor 
    $$
    \frac{\m\big(\Et{n/k_\p}{\mathcal{A}}\big)}{\m\big(\Et{n/k_\p}{}\big)}
    $$
    is strictly positive for each $\p$, without loss of generality we may assume that $\p$ is a finite prime with $\chara(\co_k/\p) > n$ and $v_\p(\alpha_i) = 0$ for all $i$. Then all degree $n$ \'etale algebras $L/k_\p$ are tamely ramified, so Lemma~\ref{lem-disc-val-of-tame-extension} tells us that 
    $$
    v_{k_\p}(d_{L/k_\p}) = d_{(L,k_\p)},
    $$
    for each $L \in \Et{n/k_\p}{}$. For $0 \leq d \leq n-1$, let 
    $
    a_d = \Part(d, n - d). 
    $
    Corollary~\ref{cor-Et-pred-subset-Et-A}, Lemma~\ref{lem-overram-O-q^{-2}}, and Corollary~\ref{cor-mass-of-etale-algebras-with-given-symbol-disc} tell us that 
    $$
    \m\big(\Et{n/k_\p}{\mathcal{A}}\big) \geq \m\big(\Et{n/k_\p}{\pred}\big) \geq a_0 + a_1 q^{-1} 
    $$
    and 
    $$
    \m\big(\Et{n/k_\p}{}\big) = a_0 + a_1q^{-1} + \ldots + a_{n-1}q^{-(n-1)}. 
    $$
    It is easy to write down a positive real number $a$, independent of $\p$, such that 
    $$
    1 - aq^{-2} \leq \frac{\m\big(\Et{n/k_\p}{\mathcal{A}}\big)}{\m\big(\Et{n/k_\p}{}\big)} \leq 1,
    $$
    and the result follows. 
\end{proof}
\begin{lemma}
    \label{lem-unit-norm-of-all-tot-ram-implies-nth-power}
    Let $F$ be a local field and let $n$ be any positive integer. Then 
    $$
    \bigcap_{L \in \Et{n/F}{}} N_{L/F}L^\times = F^{\times n}. 
    $$
\end{lemma}
\begin{proof}
    If $F$ is $\R$ or $\C$, then the result is easy to see. Therefore, we will assume that $F$ is $p$-adic. By the structure theorem for finitely generated abelian groups, we may write 
    $$
    F^\times / F^{\times n} = \bigoplus_{i=1}^r (\Z/d_i\Z)e_i, 
    $$
    for basis elements $e_i \in F^\times / F^{\times n}$ and integers $d_i\geq 2$ with $d_i \mid n$. For each $i$, let 
    $$
    A_i = \bigoplus_{j\neq i} (\Z/d_j\Z)e_j,
    $$
    and let $E_i/F$ be the degree $d_i$ abelian field extension with 
    $$
    \big(N_{E_i/F}E_i^\times\big) / F^{\times n} = A_i. 
    $$
    For each $i$, let $L_i/E_i$ be any field extension of degree $\frac{n}{d_i}$, so that 
    $$
    \big(N_{L_i/F}L_i^\times\big) / F^{\times n} \subseteq A_i.
    $$
    The result follows since $\bigcap_i A_i = 0$. 
\end{proof}

\begin{theorem}
    \label{thm-prop-between-0-and-1}
    Let $n \geq 3$, let $k$ be a number field, and let $\mathcal{A}\subseteq k^\times$ be a finitely generated subgroup. Assuming Conjecture~\ref{conj-N(X;Sigma)-asymptotics} is true for $n$ (as is the case for $n\leq 5$), we have
    $$
    0 < \lim_{X\to\infty}\frac{N_{k,n}(X;\mathcal{A})}{N_{k,n}(X)} \leq 1,
    $$
    with equality if and only if 
    $$
    \mathcal{A} \subseteq \ker\Big(
        k^\times \to \prod_{\p\in \Pi_k} k_\p^\times / k_\p^{\times n} 
    \Big).
    $$
\end{theorem}
\begin{proof}
    Since we are assuming that Conjecture~\ref{conj-N(X;Sigma)-asymptotics} holds for $n$, Lemma~\ref{lem-positive-prop} tells us that 
    $$
    0 < \lim_{X\to\infty}\frac{N_{k,n}(X;\mathcal{A})}{N_{k,n}(X)} \leq 1,
    $$
    with equality if and only if $\Et{n/k_\p}{\mathcal{A}} = \Et{n/k_\p}{}$ for all $\p \in \Pi_k$. The result then follows from Lemma~\ref{lem-unit-norm-of-all-tot-ram-implies-nth-power}. 
\end{proof}
\begin{definition}
    Given a number field $k$, we say that a positive integer $n$ is \emph{power-pathological in $k$} if the following three statements are true:
    \begin{enumerate}
        \item $n = 2^rn'$ for an odd integer $n'$ and an integer $r \geq 3$.
        \item The extension $k(\mu_{2^r})/k$ is not cyclic.
        \item All primes $\p$ of $k$ lying over $2$ decompose in $k(\mu_{2^r})/k$. 
    \end{enumerate}
\end{definition}
\begin{theorem}
    [Hasse principle for $n^\mathrm{th}$ powers]
    \label{thm-hasse-for-nth-powers}
    Let $\Pi_k$ be the set of primes of $k$. We have 
    $$
    \# \ker\Big(
        k^\times / k^{\times n} \to \prod_{\p \in \Pi_k} k_\p^\times / k_\p^{\times n} 
    \Big) = \begin{cases}
        1 \quad\text{if $n$ is not power-pathological in $k$},
        \\
        2\quad\text{if $n$ is power-pathological in $k$}. 
    \end{cases}
    $$
\end{theorem}
\begin{proof}
    This is the special case $T = \Pi_k$ of \cite[Theorem~9.1.11]{neukirch-cohomology-of-nfs}. 
\end{proof}
\begin{corollary}
    If $n$ is not power-pathological in $k$, then 
    $$
    \ker\Big(
        k^\times \to \prod_{\p \in \Pi_k} k_\p^\times / k_\p^{\times n} 
    \Big) = k^{\times n}. 
    $$
    If $n$ is power-pathological in $k$, then 
    $$
    \ker\Big(
        k^\times \to \prod_{\p \in \Pi_k} k_\p^\times / k_\p^{\times n} 
    \Big) = k^{\times n} \cup u k^{\times n},
    $$
    for some $u \in k^{\times (n/2)} \setminus k^{\times n}$.
\end{corollary}
\begin{proof}
    [Proof of Theorem~\ref{thm-prop-between-0-and-1-for-n-leq-5}]

    This is immediate from Theorems~\ref{thm-malle-bhargava-true-for-n-leq-5}, \ref{thm-prop-between-0-and-1}, and \ref{thm-hasse-for-nth-powers}.
\end{proof}
\begin{corollary}
    If $n$ is not power-pathological in $k$ and Conjecture~\ref{conj-N(X;Sigma)-asymptotics} is true for $n$, then $\mathcal{A}$ is in the norm group of $100\%$ of $S_n$-$n$-ics if and only if $\mathcal{A}$ is in the norm group of all $S_n$-$n$-ics. 
\end{corollary}

\section{Prime degree extensions}
\label{sec-prime-degree-exts}

The main purpose of this section is to prove Theorem~\ref{thm-m-A-p-for-prime-n}. We start by fixing some notation: Let $p$ and $\ell$ be rational primes; let $F$ be a $p$-adic field with residue field of size $q$; let $\pi_F$ be a uniformiser of $F$; let $\mathcal{A}\subseteq F^\times$ be a finitely generated subgroup. 
\begin{lemma}
    \label{lem-epi-symbols-for-ell}
    We have  
    $$
    \Split_\ell = \{(1^\ell), (\ell)\} \sqcup \Split_\ell^{\epi}. 
    $$
\end{lemma}
\begin{proof}
    Let $\sigma \in \Split_\ell \setminus \{(1^\ell), (\ell)\}$. Writing $\sigma = (f_1^{e_1}\ldots f_r^{e_r})$, we have
    $
    \ell = \sum_i f_ie_i,
    $
    so the $e_i$ (respectively the $f_i$) are mutually coprime, and therefore $\sigma \in \Split_\ell^\epi$. 
\end{proof}

\begin{corollary}
    \label{cor-m-et-ell-F-A}
    We have 
    $$
    \m\big(\Et{\ell/F}{\mathcal{A}}\big) = \m\big(\Et{(1^\ell)/F}{\mathcal{A}}\big) + \m\big(\Et{(\ell)/F}{\mathcal{A}}\big) + \m\big(\Et{\ell/F}{\epi}\big).  
    $$
\end{corollary}
\begin{proof}
    This follows easily from Lemmas~\ref{lem-norm-group-of-pred-symbol} and \ref{lem-epi-symbols-for-ell}.
\end{proof}

\begin{lemma}
    \label{lem-mass-Et-ell-epi}
    We have 
    $$
    \m\big(\Et{\ell/F}{\epi}\big) = \Big(1 - \frac{1}{\ell}\Big) + \sum_{d=1}^{\ell - 2} \frac{\Part(d,\ell - d)}{q^d}.
    $$
\end{lemma}
\begin{proof}
    Since $(1^\ell)$ is the only degree $\ell$ splitting symbol $\sigma$ with $d_\sigma = \ell - 1$, we have  
    $$
    \m\big(\Et{\ell/F}{\epi}\big) = \sum_{d = 0}^{\ell - 2} \m\big(
    \{L \in \Et{\ell / F}{\epi} : d_{(L,F)} = d\}    
    \big).
    $$
    Since $d_{(\ell)} = 0$, it follows from Lemma~\ref{lem-epi-symbols-for-ell} that 
    $$
    \{L \in \Et{\ell / F}{\epi} : d_{(L,F)} = d\} = \{L \in \Et{\ell / F}{} : d_{(L,F)} = d\}
    $$
    whenever $1 \leq d \leq \ell - 2$. The result follows from Corollary~\ref{cor-mass-of-etale-algebras-with-given-symbol-disc}. 
\end{proof}

Recall from Section~\ref{sec-intro} that for each positive integer $n$, we write $\overline{\mathcal{A}}^n$ for the group $\mathcal{A}F^{\times n}/F^{\times n}$. Recall the notion of a stratified generating set from Definition~\ref{defi-strat-gen-set}. 
\begin{lemma}
    \label{lem-mass-of-l-ic-unram}
    Let $(A_0, A_1)$ be a stratified generating set for $\overline{\mathcal{A}}^\ell$. Then we have 
    $$
    \m\big(\Et{(\ell)/F}{\mathcal{A}}\big) = \begin{cases}
        \frac{1}{\ell}\quad\text{if $A_1 = \varnothing$},
        \\
        0 \quad\text{otherwise}. 
    \end{cases}
    $$
\end{lemma}
\begin{proof}
    This follows from the fact that the unramified degree $\ell$ extension of $F$ has norm group $\co_F^\times F^{\times \ell}$. 
\end{proof}
Because of Corollary~\ref{cor-m-et-ell-F-A} and Lemmas~\ref{lem-mass-Et-ell-epi} and \ref{lem-mass-of-l-ic-unram}, it remains only to find the pre-mass $\m\big(\Et{(1^\ell)/F}{\mathcal{A}}\big)$ of totally ramified degree $\ell$ extensions with $\mathcal{A}$ in their norm group. We address the tamely ramified case $p \neq \ell$ and the wildly ramified case $p = \ell$ separately. 

\subsection{Finding $\m\big(\Et{(1^\ell)/F}{\mathcal{A}}\big)$ when $p \neq \ell$}

For each positive integer $m$, fix a primitive $m^\mathrm{th}$ root of unity $\zeta_m$ in the algebraic closure of $F$.

\begin{lemma}
    \label{lem-tot-tame-ram-classification}
    Let $e$ be a positive integer coprime to $q$ and define $g = \gcd(e,q-1)$. For $0 \leq j \leq g-1$, define 
    $$
    L_j = F\Big(\sqrt[e]{\zeta_{q-1}^j\pi_F}\Big). 
    $$ 
    The extensions $L_j/F$ are all nonisomorphic, and 
    $$
    \Et{(1^e)/F}{} = \big\{L_j : j = 0,1,\ldots, g-1\big\}.
    $$
\end{lemma}
\begin{proof}
    This is essentially \cite[Theorem~7.2]{compute-extensions}, where we replace the polynomial $X^e + \zeta_{q-1}^j\pi_F$ with $X^e - \zeta_{q-1}^j\pi_F$. The modifications to the proof are trivial. 
\end{proof}
\begin{lemma}
    \label{lem-norm-groups-of-tot-tame-ram-l-exts}
    Suppose that $p\neq \ell$. The following two statements are true:
    \begin{enumerate}
        \item If $\ell \mid q-1$, then $\Et{(1^\ell)/F}{} = \{L_0,\ldots, L_{\ell -1}\}$, where $L_j = F\Big(\sqrt[\ell]{\zeta_{q-1}^j\pi_F}\Big)$. Moreover, for each $j$ we have $\Aut(L_j/F) \cong C_\ell$ and
        $$
        N_{L_j/F}L_j^\times = \big\{u^\ell ((-1)^{\ell + 1}\zeta_{q-1}^j\pi_F)^m : u \in \co_F^\times, m \in \Z\big\}.
        $$
        \item If $\ell \nmid q-1$, then $\Et{(1^\ell)/F}{} = \{L\}$, where $L = F(\sqrt[\ell]{\pi_F})$. Moreover, we have $\Aut(L/F) = 1$ and $N_{L/F}L^\times = F^\times$.
    \end{enumerate}
\end{lemma}
\begin{proof}
    In both cases, the classification of extensions comes from Lemma~\ref{lem-tot-tame-ram-classification}. To prove the statements about norms and automorphisms, we consider the two cases separately. 
    \begin{enumerate}
        \item Suppose that $\ell \mid q-1$, so that $\zeta_\ell\in F^\times$. Since the elements $\zeta_\ell^j\pi_F$ are all uniformisers of $F$, and $\pi_F$ is an arbitrary uniformiser, we only need to prove the result for $L_0 = F(\sqrt[\ell]{\pi_F})$. The element $\sqrt[\ell]{\pi_F}$ has minimal polynomial $X^\ell - \pi_F$ over $F$, which splits in $L_0$, so $L_0/F$ is Galois, hence cyclic. By class field theory, we have
        $$
        [F^\times : N_{L_0/F}L_0^\times] = \ell.
        $$
        Moreover, it is clear that 
        $$
        \big\{u^\ell ((-1)^{\ell + 1}\pi_F)^m : u \in \co_F^\times, m \in \Z\big\} \subseteq N_{L_0/F}L_0^\times. 
        $$
        By \cite[Part II, Proposition~3.7]{neukirch-bonn}, we have 
        $$
        [\co_F^\times : \co_F^{\times \ell}] = \# \mu_\ell(F) = \ell.
        $$
        It follows that  
        $$
        \big[F^\times : \big\{u^\ell ((-1)^{\ell + 1}\pi_F)^m : u \in \co_F^\times, m \in \Z\big\}\big] = \ell,
        $$
        and the result follows. 
    \item Suppose that $\ell \nmid q-1$. Since $\ell \neq p$ and $\ell \nmid q-1$, the result \cite[Page~140, Proposition~5.7]{neukirch2013algebraic} tells us that $\zeta_\ell \not \in L$. It follows that the polynomial $X^\ell - \pi_F$ has only one root in $L$, and therefore $L/F$ is not Galois, hence it has only one automorphism. The maximal abelian subextension of $L/F$ is $F$, so it has trivial norm group. 
    \end{enumerate}
\end{proof}
\begin{lemma}
    \label{lem-size-of-Et-1^ell-A-for-l-divides-(q-1)}
    Suppose that $p \neq \ell$ and $\ell \mid q-1$. Let $(A_0, A_1)$ be a stratified generating set for $\overline{\mathcal{A}}^\ell$. 
    Then 
    $$
    \# \Et{(1^\ell)/F}{\mathcal{A}} = \begin{cases}
        \ell \quad\text{if $\mathcal{A}\subseteq F^{\times \ell}$} ,
        \\
        1 \quad\text{if $A_0 = \varnothing$ and $\# A_1 = 1$},
        \\
        0 \quad\text{otherwise}. 
    \end{cases}
    $$
\end{lemma}
\begin{proof}
    This follows easily from Lemmas~\ref{lem-tot-tame-ram-classification}~and~\ref{lem-norm-groups-of-tot-tame-ram-l-exts}.
\end{proof}

\begin{corollary}
    \label{cor-mass-1^ell-tame-ram}
    Suppose that $p \neq \ell$, and let $(A_0, A_1)$ be a stratified generating set for $\overline{\mathcal{A}}^\ell$. The following two statements are true:
    \begin{enumerate}
        \item If $\ell \nmid q-1$, then 
        $$
        \m\big(\Et{(1^\ell)/F}{\mathcal{A}}\big) = \frac{1}{q^{\ell - 1}}.
        $$
        \item If $\ell \mid q-1$, then  
        $$
        \m\big(\Et{(1^\ell)/F}{\mathcal{A}}\big) = \begin{cases}
            \frac{1}{q^{\ell-1}} \quad\text{if $\mathcal{A}\subseteq F^{\times \ell}$} ,
            \\
            \frac{1}{\ell q^{\ell - 1}} \quad\text{if $A_0 = \varnothing$ and $\#A_1 = 1$},
            \\
            0 \quad\text{otherwise}. 
        \end{cases}
        $$
    \end{enumerate}
\end{corollary}
\begin{proof}
    This is immediate from Lemmas~\ref{lem-disc-val-of-tame-extension}, \ref{lem-norm-groups-of-tot-tame-ram-l-exts} and \ref{lem-size-of-Et-1^ell-A-for-l-divides-(q-1)}. 
\end{proof}

\subsection{Finding $\m\big(\Et{(1^\ell)/F}{\mathcal{A}}\big)$ when $p = \ell$}
\label{subsec-wild-ram-Cp-exts}
In the current subsection, we will address the wildly ramified case $p=\ell$. 
\begin{lemma}
    \label{lem-m-Et-1^p-A-in-terms-of-other-quantities}
    We have 
    $$
    \m\big(\Et{(1^p)/F}{\mathcal{A}}\big) = \m\big(\Et{(1^p)/F}{}\big) - \m\big(\Et{(1^p)/F}{C_p/F}\big) + \m\big(\Et{(1^p)/F}{C_p/F,\mathcal{A}}\big).
    $$
\end{lemma}
\begin{proof}
    Let 
    $
    L \in \Et{(1^p)/F}{} \setminus \Et{(1^p)/F}{C_p/F}. 
    $
    Then the maximal abelian subextension of $L/F$ is $F$, so by class field theory we have $N_{L/F}L^\times = F^\times$. Therefore, we have
    $$
    \Et{(1^p)/F}{\mathcal{A}} = \Et{(1^p)/F}{C_p/F,\mathcal{A}} \cup \Big(\Et{(1^p)/F}{} \setminus \Et{(1^p)/F}{C_p/F}\Big),
    $$
    and the result follows. 
\end{proof}

Let $L/F$ be an abelian extension of $p$-adic fields, and let $G = \Gal(L/F)$. Write $\mff(L/F)$ for the conductor of the extension $L/F$, defined to be the smallest integer $m$ such that $U_F^{(m)}\subseteq N_{L/F}L^\times$.
For a character $\chi:G \to \C^\times$, the \emph{Artin conductor} of $\chi$ is defined to be 
$$
\mff(\chi) = \sum_{i\geq 0}\frac{\lvert G_i\rvert}{\lvert G_0\rvert}\Big(\chi(1) - \chi(G_i)\Big),
$$
where $(G_i)_i$ is the ramification filtration of $G$, defined by
$$
G_i = \Big\{\varphi \in G: v_L\Big(\frac{\varphi(\pi_L) - \pi_L}{\pi_L}\Big) \geq i\Big\},
$$
and $\chi(G_i)$ is the average value of $\chi(g)$ as $g$ ranges over $G_i$. Write $G^\vee$ for the set of characters $\chi:G\to\C^\times$.
\begin{theorem}
    \label{thm-conductor-discriminant}
    Let $L/F$ be an abelian extension of $p$-adic fields, and let $G = \Gal(L/F)$. We have \begin{enumerate}
        \item $\mff(L/F) = \max_{\chi \in G^\vee} \mff(\chi)$.
        \item $v_F(d_{L/F}) = \sum_{\chi \in G^\vee} \mff(\chi)$. 
    \end{enumerate}
\end{theorem}
\begin{proof}
    The first item is \cite[Page 135, Corollary to Theorem~14]{artin-CFT}. The second is \cite[Theorem 17.50]{keune-nf}.
\end{proof}
\begin{lemma}
    \label{lem-cond-disc-for-Cp}
    Let $L/F$ be a $C_p$-extension of $p$-adic fields. Then 
    $$
    \mff(L/F) = \frac{v_F(d_{L/F})}{p-1}. 
    $$
\end{lemma}
\begin{proof}
Take $G = \Gal(L/F) = \langle g \rangle$, and write $c$ for the unique integer with $G_{c-1} \neq G_c$. Writing $\omega$ for the complex number $e^{\frac{2\pi i}{p}}$, we have 
$$
G^\vee = \{\chi_j : j=0,1,2,\ldots, p-1\},
$$
where $\chi_j(g) = \omega^j$. We have 
$$
\mff(\chi_j) = \begin{cases}
    0\quad\text{if $j=0$},
    \\
    c\quad\text{otherwise},
\end{cases}
$$
so Theorem~\ref{thm-conductor-discriminant} tells us that $\mff(L/F) = c$ and $v_F(d_{L/F}) = (p-1)c$, and the result follows. 
\end{proof}

Given a set $\Et{\bullet/F}{\bullet}$ of \'etale algebras over $F$, write $\Et{\bullet/F, m}{\bullet}$ and $\Et{\bullet/F, \leq m}{\bullet}$ for the sets of $L \in \Et{\bullet/F}{\bullet}$ with $v_F(d_{L/F}) = m$ and $v_F(d_{L/F}) \leq m$, respectively. 

\begin{lemma}
    \label{lem-chars-to-Cp-exts-(p-1)-to-1}
    Let $m$ be a positive integer. There is a surjective $(p-1)$-to-$1$ map 
    $$
    \operatorname{Epi}_{\F_p}\Big(F^\times / U_F^{(\lfloor\frac{m}{p-1}\rfloor)}F^{\times p} , \F_p\Big) \to \Et{p/F, \leq m}{C_p/F},
    $$
    where the map $\chi : F^\times / U_F^{(\lfloor\frac{m}{p-1}\rfloor)}F^{\times p} \twoheadrightarrow \F_p$ is sent to the unique abelian extension $L/F$ with 
    $$
    N_{L/F}L^\times = \ker \Big(F^\times \to F^\times / U_F^{(\lfloor\frac{m}{p-1}\rfloor)}F^{\times p} \overset{\chi}{\to} \F_p\Big).
    $$
\end{lemma}
\begin{proof}
    This follows from Lemma~\ref{lem-cond-disc-for-Cp} by basic class field theory. 
\end{proof}
\begin{corollary}
    \label{cor-num-(1^p)-exts-in-terms-of-quotient}
    For every positive integer $m$, we have 
    $$
    \#\Et{p/F,\leq m}{C_p/F} = \frac{1}{p-1}\Big(\#\Big(F^\times / U_F^{(\lfloor \frac{m}{p-1}\rfloor)}F^{\times p}\Big) - 1\Big).
    $$
\end{corollary}
\begin{proof}
    This is immediate from Lemma~\ref{lem-chars-to-Cp-exts-(p-1)-to-1}. 
\end{proof}

\begin{lemma}
    \label{lem-lift-squares-mod-p-n}
    Let $u \in \co_F^\times$ and let $n$ be an integer with $n > \frac{pe_F}{p-1}$. Then $U_F^{(n)}F^{\times p} = U_F^{(n+1)}F^{\times p}$. 
\end{lemma}
\begin{proof}
    Let $u \in U_F^{(n)}F^{\times p}$. Then there exist elements $a \in F^\times$ and $x \in \p_F^n$ with
    $
    u = a^p(1+x).
    $
    Since $v_F(x) \geq n$ and $n > \frac{pe_F}{p-1}$, it is easy to see that
    $$
    \Big(1 + \frac{x}{p}\Big)^p \equiv 1 + x \pmod{\p_F^{n+1}}, 
    $$
    and it follows that $u \in U_F^{(n+1)}F^{\times p}$, as required. 
\end{proof}
\begin{corollary}
    {\label{cor-p-power-iff-p-power-mod-p-pi}}
    We have $U_F^{(\lfloor \frac{pe_F}{p-1}\rfloor + 1)}\subseteq F^{\times p}$. 
\end{corollary}
\begin{proof}
    This follows from Lemma~\ref{lem-lift-squares-mod-p-n} by Hensel's lemma.
\end{proof}

\begin{lemma}
    \label{lem-iso-to-direct-sum-of-Wi}
    For $0 \leq i \leq \lfloor \frac{pe_F}{p-1}\rfloor$, let $W_i = U_F^{(i)}F^{\times p} / U_F^{(i+1)}F^{\times p}$. For each positive integer $c$, there is a group isomorphism
    $$
    F^\times / U_F^{(c)}F^{\times p} \cong \Big(F^\times / U_F^{(0)}F^{\times p}\Big) \oplus \bigoplus_{i=0}^{\min\{c-1,\lfloor\frac{pe_F}{p-1}\rfloor\}} W_i. 
    $$
\end{lemma}
\begin{proof}
    Corollary~\ref{cor-p-power-iff-p-power-mod-p-pi} tells us that 
    $$
    F^\times / U_F^{(c)}F^{\times p} = F^\times / U_F^{(\lfloor \frac{pe_F}{p-1}\rfloor + 1)}F^{\times p}
    $$
    whenever $c \geq \lfloor\frac{pe_F}{p-1}\rfloor + 1$, so we only need to prove the result for $c \leq \lfloor \frac{pe_F}{p-1}\rfloor + 1$.  
    
    The left- and right-hand side have the same cardinality by definition of the $W_i$. Moreover, both are $p$-torsion groups, hence $\F_p$-vector spaces, so they are isomorphic as groups. 
\end{proof}
\begin{lemma}
    \label{lem-UF0F*p-eq-UF1F*p}
    We have 
    $$
    U_F^{(0)}F^{\times p} = U_F^{(1)}F^{\times p}. 
    $$
\end{lemma}
\begin{proof}
    The $p$-power map 
    $$
    \F_F^\times \to \F_F^\times,\quad x \mapsto x^p 
    $$
    is injective, hence bijective. Thus, every element of $U_F^{(0)}$ is congruent to a $p^\mathrm{th}$ power modulo $\p_F$, and the result follows. 
\end{proof}
\begin{lemma}
    \label{lem-lift-to-pth-power-mod-n}
    Let $0 \leq i \leq \frac{pe_F}{p-1} $, and let $m \in U_F^{(i)}$. 
    \begin{enumerate}
        \item If $p \nmid i$, then $m \in U_F^{(i+1)}F^{\times p}$ if and only if $m \in U_F^{(i+1)}$.
        \item If $p\mid i$, then $m \in U_F^{(i+1)}F^{\times p}$ if and only if there is some $x \in \p_F^{\frac{i}{p}}$ such that $(1 + x)^p \equiv m\pmod{\p_F^{i+1}}$.
        \item If $p\mid i$ and $i < \frac{pe_F}{p-1}$, then we always have $m \in U_F^{(i+1)}F^{\times p}$, and in particular 
        $$
        m \equiv \Big(1 + \pi_F^{\frac{i}{p}}y\Big)^p\pmod{\p_F^{i+1}},
        $$
         where $[y] \in \co_F/\p_F$
        is the unique element with $[y]^p = \Big[\frac{m-1}{\pi_F^{i}}\Big],$ which exists by Lemma~\ref{lem-UF0F*p-eq-UF1F*p}.
        \item If $i = \frac{pe_F}{p-1}$, then $m \in U_F^{(i+1)}F^{\times p}$ if and only if $\Big[\frac{m-1}{{\pi_F^{e_F/(p-1)}p}}\Big]\in \co_F/\p_F$ is in the image of the map 
        $$
        \co_F/\p_F \to \co_F/\p_F, \quad y \mapsto y + \frac{\pi_F^{e_F}}{p}y^p,
        $$
        and in that case $m \equiv \Big(1 + \pi_F^{\frac{i}{p}}y\Big)^p \pmod{\p_F^{i+1}}$, for each $y$ in the preimage. 
    \end{enumerate}
\end{lemma}
\begin{proof}
    By Lemma~\ref{lem-UF0F*p-eq-UF1F*p}, the case $i=0$ is trivial, so we assume that $i\geq 1$. For the first two statements, the ``if'' directions are trivial, so we focus on the ``only if''.
    
    Suppose that $m \in U_F^{(i+1)}F^{\times p} \setminus U_F^{(i+1)}$, so $v_F(m-1) = i$. We will show that this implies $p\mid i$ and there is some $x \in \p_F^{\frac{i}{p}}$ with $(1 + x)^p \equiv m\pmod{\p_F^{i+1}}$, proving the first two statements. 
    
    Since $m \in U_F^{(i+1)}F^{\times p}$, there is some $c \in F^\times$ such that $v_F(c) = 0$ and $m\equiv c^p\pmod{\p_F^{i+1}}$. Write $c = 1 + x$ for $x \in \co_F$, so that 
    $$
    m - 1 \equiv \sum_{j=1}^{p-1} \binom{p}{j}x^j + x^p\pmod{\p_F^{i+1}}. 
    $$
    Since $v_F(m-1) > 0$, we have $v_F(x) > 0$, so 
    \begin{equation}
    \label{eqn-binom-expansion-valuations}
    \begin{cases}
    v_F\Big(\sum_{j=1}^{p-1}\binom{p}{j}x^j\Big) = e_F + v_F(x) ,
    \\
    v_F(x^p) = pv_F(x). 
    \end{cases}
\end{equation}
    Suppose for contradiction that $v_F(x) > \frac{e_F}{p-1}$. Then $v_F(x) + e_F < pv_F(x)$, so 
    $$
    i = v_F(m-1) = v_F(x) + e_F > \frac{pe_F}{p-1},
    $$
    which is impossible since by assumption $i \leq \frac{pe_F}{p-1}$. 
    
    Therefore, $v_F(x) \leq \frac{e_F}{p-1}$. We will consider the cases $v_F(x) = \frac{e_F}{p-1}$ and $v_F(x) < \frac{e_F}{p-1}$ separately. Suppose first that $v_F(x) = \frac{e_F}{p-1}$. Then $pv_F(x) = e_F + v_F(x) = \frac{pe_F}{p-1}$, so Equation~(\ref{eqn-binom-expansion-valuations}) tells us that
    $$
    i = v_F(m-1) \geq \frac{pe_F}{p-1},
    $$
    and therefore $i = \frac{pe_F}{p-1}$ and $v_F(x) \geq \frac{i}{p}$, as required. 
    
    Suppose instead that $v_F(x) < \frac{e_F}{p-1}$. Then $pv_F(x) < v_F(x) + e_F$, so Equation~(\ref{eqn-binom-expansion-valuations}) tells us that
    $$
    i = v_F(m-1) = pv_F(x),
    $$
    and therefore $p\mid i$ and $v_F(x) \geq \frac{i}{p}$. Thus we have proved Statements $(1)$ and $(2)$. 
    
    Suppose that $p\mid i$ and $i < \frac{pe_F}{p-1}$. By Lemma~\ref{lem-UF0F*p-eq-UF1F*p}, there is a $y \in \co_F$ with $y^p \equiv \frac{m-1}{\pi_F^i}\pmod{\p_F}$. Let $x = \pi_F^{\frac{i}{p}}y$. Then we have 
    $$
    (1 + x)^p \equiv 1 + x^p \equiv m\pmod{\p_F^{i+1}},
    $$
    so Statement $(3)$ follows from Statement $(2)$. 

    Suppose that $i = \frac{pe_F}{p-1}$. Statement $(2)$ tells us that $m \in U_F^{(i+1)}F^{\times p}$ if and only if there is some $x \in \p_F^{\frac{e_F}{p-1}}$ with 
    $$
    m-1 \equiv px + x^p\pmod{\p_F^{\frac{pe_F}{p-1} + 1}},
    $$
    and Statement $(4)$ follows easily. 
\end{proof}
We note the following algorithm as an immediate consequence of Lemma~\ref{lem-lift-to-pth-power-mod-n}:
\begin{algorithm}
    \label{algo-c-alpha}
    \hfill 

    \textbf{Input:} $\alpha \in \co_F^\times$. 

    \textbf{Output:} Returns a pair $(i,\lambda)$. If $\alpha \in F^{\times p}$, then $i=\infty$ and $\lambda \in \co_F^\times$ is such that 
    $$
    \alpha\equiv \lambda^p\pmod{\p_F^{\lfloor\frac{pe_F}{p-1}\rfloor + 1}}.
    $$
    Otherwise, $i$ is the largest integer with $\alpha \in U_F^{(i)}F^{\times p}$, and $\lambda \in \co_F^\times$ is an element such that $\alpha \equiv \lambda^p\pmod{\p_F^{i}}$. 

    \textbf{Algorithm:} 
    
    \begin{enumerate}
        \item Set $m_0 = \alpha$ and $\lambda_0 = 1$. 
        \item For $0 \leq i \leq \frac{pe_F}{p-1}$, do the following:
        
        \hfill

        \noindent If $p \nmid i$, then:
        \begin{itemize}\item If $m_i \equiv 1\pmod{\p_F^{i+1}}$, then set $m_{i+1} = m_i$ and $\lambda_{i+1} = \lambda_i$. \item Otherwise, return $(i, \lambda_i)$ and break the for loop. 
        \end{itemize}

        \noindent If $p\mid i$ and $i < \frac{pe_F}{p-1}$, then: \begin{itemize}
            \item Let $[y] \in \co_F/\p_F$ be the unique element with $[y]^p = \Big[\frac{m_i-1}{\pi_F^{i}}\Big]$, and set $\lambda_{i+1} = \lambda_i(1+\pi_F^{i/p}y)$ and $m_{i+1} = \frac{m_i}{(1 + \pi_F^{i/p}y)^p}$. 
        \end{itemize}

        \noindent If $i = \frac{pe_F}{p-1}$, then:
        \begin{itemize}\item  If $\Big[\frac{m_i-1}{{\pi_F^{e_F/(p-1)}p}}\Big]\in \co_F/\p_F$ is in the image of the map 
        $$
        \co_F/\p_F \to \co_F/\p_F, \quad y \mapsto y + \frac{\pi_F^{e_F}}{p}y^p,
        $$
        then take $y$ in the preimage and set $\lambda_{i+1} = \lambda_{i}(1 + \pi_F^{i/p}y)$ and $m_{i + 1} = \frac{m_{i}}{(1 + \pi_F^{i/p}y)^p}$.
        \item Otherwise, return $(\frac{pe_F}{p-1}, \lambda_{\frac{pe_F}{p-1}})$. 
    \end{itemize}
    By Lemma~\ref{lem-lift-to-pth-power-mod-n}, we see inductively that $m_i \equiv 1\pmod{\p_F^i}$ and $\alpha = \lambda_i^p m_i$ for all $i$ that it is defined.
        \item If the for loop from the previous step finishes, then return $(\infty, \lambda_{\frac{pe_F}{p-1}+1})$. 
    \end{enumerate}
\end{algorithm}

\begin{lemma}
    \label{lem-im-of-coef-times-y^p+y-map}
    Suppose that $(p-1) \mid e_F$. Define $\varphi$ to be the map  
    $$
    \varphi: \co_F/\p_F \to \co_F/\p_F, \quad y \mapsto y + \frac{\pi_F^{e_F}}{p}y^p.
    $$
    We have 
    $$
    \#\im\varphi =\begin{cases}
        q/p\quad\text{if $\mu_p\subseteq F$},
        \\
        q\quad\text{if $\mu_p\not\subseteq F$}. 
    \end{cases}
    $$
\end{lemma}
\begin{proof}
    The map is $\F_p$-linear, and its kernel consists of the roots of the polynomial 
    $$
    X\Big(X^{p-1} + \frac{\pi_F^{e_F}}{p}\Big) \in \F_F[X].
    $$
    Since $\F_F$ contains all $(p-1)^\mathrm{st}$ roots of unity, the polynomial has either $1$ or $p$ roots. By Hensel's lemma, any root of $X^{p-1} + \frac{\pi_F^{e_F}}{p}$ in $\F_F$ lifts to a root in $F$, which exists if and only if $-p \in F^{\times (p-1)}$. The result \cite[Lemma~14.6]{washington1997introduction} states that $\Q_p(\sqrt[p-1]{-p}) = \Q_p(\zeta_p)$, and the result follows. 
\end{proof}
Recall from Lemma~\ref{lem-iso-to-direct-sum-of-Wi} that we are interested in the groups $W_i = U_F^{(i)}F^{\times p} / U_F^{(i+1)} F^{\times p}$, for integers $i$ with $ 0 \leq i \leq \lfloor \frac{pe_F}{p-1}\rfloor$.
\begin{corollary}
    \label{cor-size-of-Wi}
    Let $i$ be an integer with $0 \leq i \leq \frac{pe_F}{p-1}$. We have group isomorphisms
    $$
     W_i \cong \begin{cases}
        \F_F \quad\text{if $i< \frac{pe_F}{p-1}$ and $p\nmid i$},
        \\
        \F_p \quad\text{if $i = \frac{pe_F}{p-1}$ and $\mu_p\subseteq F$},
        \\
        1 \quad\text{otherwise}.
    \end{cases}
    $$
\end{corollary}
\begin{proof}
    Suppose first that $p\nmid i$. Then Lemma~\ref{lem-lift-to-pth-power-mod-n}, Part (1) tells us that the natural sequence 
    $$
    1 \to U_F^{(i+1)} \to U_F^{(i)} \to W_i \to 1,
    $$
    is exact, so $W_i \cong \F_F$. If $p \mid i$ and $i < \frac{pe_F}{p-1}$, then Lemma~\ref{lem-lift-to-pth-power-mod-n}, Part (3) tells us that $W_i = 0$. Finally, suppose that $i = \frac{pe_F}{p-1}$. Write $\varphi$ for the map 
    $$
    \varphi: \co_F/\p_F \to \co_F/\p_F, \quad y \mapsto y + \frac{\pi_F^{e_F}}{p}y^p.
    $$
    By Lemma~\ref{lem-lift-to-pth-power-mod-n}, Part (4), we have an exact sequence 
    $$
    1 \to \im\varphi \to U_F^{(i)}/U_F^{(i+1)} \to W_i \to 1,
    $$
    where the map $\im\varphi \to U_F^{(i)}/U_F^{(i+1)}$ is given by $[x] \mapsto [1 + \pi_F^{\frac{e_F}{p-1}}px]$. The result then follows by Lemma~\ref{lem-im-of-coef-times-y^p+y-map}.
\end{proof}
\begin{corollary}
    \label{cor-dim-of-F-mod-Ut-and-pth-powers}
    Let $c$ be a nonnegative integer. If $0 \leq c \leq \lceil\frac{pe_F}{p-1}\rceil$, then we have 
    $$
    \#\big(F^\times / U_F^{(c)}F^{\times p}\big) = pq^{c - 1 - \lfloor\frac{c-1}{p}\rfloor}.
    $$
    If $c > \lceil \frac{pe_F}{p-1}\rceil$, then we have
    $$
    \#\big(F^\times / U_F^{(c)}F^{\times p}\big) = \begin{cases}p^2q^{e_F}\quad\text{if $\mu_p\subseteq F$} 
        \\
        pq^{e_F}\quad\text{otherwise}. 
    \end{cases}
    $$
\end{corollary}
\begin{proof}
    This follows easily from Lemma~\ref{lem-iso-to-direct-sum-of-Wi} and Corollary~\ref{cor-size-of-Wi}.
\end{proof}
\begin{corollary}
    \label{cor-num-local-Cp-exts-with-bounded-disc}
    Let $L \in \Et{p/F}{C_p/F}$. Then $v_F(d_{L/F}) = (p-1)c$ for some integer $c$ with $0 \leq c \leq \lfloor\frac{pe_F}{p-1}\rfloor + 1$. 
  
    If $1 \leq c \leq \lceil \frac{pe_F}{p-1} \rceil$, then we have 
    $$
    \# \Et{p/F, \leq (p-1)c}{C_p/F} = 
        \frac{1}{p-1}\Big(pq^{c-1-\lfloor \frac{c-1}{p}\rfloor} - 1\Big).
    $$  
    If $(p-1)\mid e_F$, then we have 
    $$
    \# \Et{p/F, \leq pe_F + p - 1}{C_p/F} = \begin{cases}
        \frac{1}{p-1}\Big(p^2q^{e_F} - 1\Big)\quad\text{if $\mu_p\subseteq F$},
        \\
        \frac{1}{p-1}\Big(pq^{e_F} - 1\Big)\quad\text{otherwise}. 
    \end{cases}
    $$
\end{corollary}
\begin{proof}
    This follows easily from Corollaries~\ref{cor-num-(1^p)-exts-in-terms-of-quotient} and \ref{cor-dim-of-F-mod-Ut-and-pth-powers}.
\end{proof}

\begin{theorem}
    \label{thm-num-Cp-exts-with-disc-val}
    If $\Et{(1^p)/F,m}{C_p/F}$ is nonempty, then $m = (p-1)c$ for an integer $c$ with $1 \leq c \leq \lfloor\frac{pe_F}{p-1}\rfloor+1$, and 
    $$
    \#\Et{(1^p)/F,m}{C_p/F} = \begin{cases}
        \frac{p(q-1)}{p-1}\cdot q^{c-2-\lfloor\frac{c-2}{p}\rfloor}\quad\text{if $c \not \equiv 1\pmod{p}$},
        \\
        pq^{e_F}\quad\text{if $c = \frac{pe_F}{p-1} + 1$ and $\mu_p\subseteq F$},
        \\
        0\quad\text{otherwise}. 
    \end{cases}
    $$
\end{theorem}
\begin{proof}
    This follows easily from Corollary~\ref{cor-num-local-Cp-exts-with-bounded-disc}. 
\end{proof}
\begin{lemma}
    \label{lem-preliminary-sum-for-Cp-mass}
    Let $p$ be an integer with $p \geq 2$ and let $q$ be a positive rational number. In Appendix~\ref{appendix-helpers}, we define functions $A(t)$ and $B(t)$ in terms of $p$ and $q$. For any integer $t$ with $t \geq 2$, we have
    $$
    \sum_{\substack{1 \leq c \leq t \\ c \not \equiv 1\pmod{p}}} q^{-(p-2)c - \lfloor\frac{c-2}{p}\rfloor} = \mathbbm{1}_{t \geq p}\cdot A(t) + \mathbbm{1}_{t \not\equiv 0,1\pmod{p}}\cdot B(t).
    $$
\end{lemma}
\begin{proof}
    The proof is a straightforward computation. To eliminate the possibility of a manipulation error, we have checked the identity numerically (see the Python notebook in the Github repository \url{https://github.com/Sebastian-Monnet/Sn-n-ics-paper-checks}). 
\end{proof}

\begin{corollary}
    \label{cor-mass-1^p-Cp}
    Recall the explicit functions $A(t)$ and $B(t)$ from Appendix~\ref{appendix-helpers}. We have 
    $$
    \m\big(\Et{(1^p)/F}{C_p/F}\big) = \frac{q-1}{p-1}q^{-2}\Big(\mathbbm{1}_{e_F \geq p-1}\cdot A\Big(\Big\lceil\frac{pe_F}{p-1}\Big\rceil\Big) + \mathbbm{1}_{(p-1)\nmid e_F} \cdot B\Big(\Big\lceil \frac{pe_F}{p-1}\Big\rceil\Big)\Big) + \mathbbm{1}_{\mu_p \subseteq F} \cdot q^{-(p-1)(e_F+1)}. 
    $$
\end{corollary}
\begin{proof}
    By Theorem~\ref{thm-num-Cp-exts-with-disc-val}, the mass $\m\big(\Et{(1^p)/F}{C_p/F}\big)$ is the sum of the following two quantities:
    \begin{enumerate}
        \item $$
            \frac{q-1}{p-1}\cdot q^{-2} \cdot \sum_{\substack{1 \leq c < \frac{pe_F}{p-1} + 1 \\ c \not \equiv 1\pmod{p}}} q^{- \big((p-2)c + \lfloor\frac{c-2}{p}\rfloor\big)} .
        $$
        \item $$\mathbbm{1}_{\mu_p\subseteq F}\cdot q^{-(p-1)(e_F+1)}.$$
    \end{enumerate}
    For $c \in \Z$, we have $c < \frac{pe_F}{p-1} + 1$ if and only if $c \leq \lceil \frac{pe_F}{p-1}\rceil$. Setting $t = \lceil \frac{pe_F}{p-1}\rceil$, Lemma~\ref{lem-preliminary-sum-for-Cp-mass} tells us that 
    $$
    \sum_{\substack{1 \leq c < \frac{pe_F}{p-1} + 1 \\ c \not \equiv 1\pmod{p}}} q^{- \big((p-2)c + \lfloor\frac{c-2}{p}\rfloor\big)} = \mathbbm{1}_{\lceil \frac{pe_F}{p-1}\rceil \geq p}\cdot A\Big(\Big\lceil \frac{pe_F}{p-1}\Big\rceil\Big) + \mathbbm{1}_{\lceil \frac{pe_F}{p-1}\rceil \not \equiv 0,1\pmod{p}}\cdot B\Big(\Big\lceil \frac{pe_F}{p-1}\Big\rceil\Big) .
    $$
    It is easy to see that $\lceil \frac{pe_F}{p-1}\rceil \geq p$ if and only if $e_F \geq p-1$. We claim that $\lceil \frac{pe_F}{p-1}\rceil\equiv 0,1\pmod{p}$ if and only if $(p-1)\mid e_F$. To see this, write $e_F = m(p-1) + r$ for integers $m$ and $r$ with $0 \leq r \leq p-2$. Then 
    $$
    \Big\lceil\frac{pe_F}{p-1}\Big\rceil = pm + \Big\lceil \frac{pr}{p-1}\Big\rceil.
    $$
    If $(p-1)\mid e_F$, then $r = 0$, so $\big\lceil\frac{pe_F}{p-1}\big\rceil \equiv 0\pmod{p}$. Otherwise, we have $r \geq 1$, so $1 < \frac{pr}{p-1} \leq p-1$, hence $\big\lceil\frac{pe_F}{p-1}\big\rceil \not\equiv 0,1\pmod{p}$. The result follows.
\end{proof}
Recall that, given nonnegative integers $n$ and $t$ with $n\geq 1$ and $t \geq 0$, we write 
$$
\overline{\mathcal{A}}^n = \mathcal{A}F^{\times n}/ F^{\times n} 
$$
and 
$$
\overline{\mathcal{A}}^n_t = \overline{\mathcal{A}}^n \cap \Big(U_F^{(t)} F^{\times n}/F^{\times n}\Big).
$$
\begin{lemma}
    \label{lem-wildly-ram-Cp-exts-with-norms-and-bounded-disc-basic-form}
    Let $c$ be an integer with $0 \leq c \leq \frac{pe_F}{p-1} + 1$. We have 
    $$
    \# \Et{p/F, \leq (p-1)c}{C_p/F,\mathcal{A}} = \frac{1}{p-1}\cdot \Big(
    \frac{\#\overline{\mathcal{A}}^p_c}{\#\overline{\mathcal{A}}^p} \cdot \#\big(F^\times / U_F^{(c)}F^{\times p}\big) - 1    
    \Big).
    $$
\end{lemma}
\begin{proof}
    By Lemma~\ref{lem-chars-to-Cp-exts-(p-1)-to-1}, we need to count $\F_p$-linear transformations  
    $$
    \chi : F^\times / F^{\times p} \to \F_p
    $$
    such that 
    $$
    \chi\big(\overline{\mathcal{A}}^p\big) = \chi\big(U_F^{(c)}F^{\times p}/F^{\times p}\big) = 0.
    $$
    In other words, we need to compute the size of the annihilator 
    $$
    \Big(\overline{\mathcal{A}}^p + \Big(U_F^{(c)}F^{\times p}/F^{\times p}\Big)\Big)^\bot.
    $$
    It is easy to see that 
    $$
    \dim_{\F_p}\Big(\big(\overline{\mathcal{A}}^p + (U_F^{(c)}F^{\times p}/F^{\times p})\big)^\bot\Big) = \dim_{\F_p}\Big(F^\times / U_F^{(c)}F^{\times p}\Big) - \dim_{\F_p}\big(\overline{\mathcal{A}}^p/\overline{\mathcal{A}}^p_c\big),
    $$
    so 
    $$
    \#\Big(\overline{\mathcal{A}}^p + (U_F^{(c)}F^{\times p}/F^{\times p})\Big)^\bot = \frac{\#\overline{\mathcal{A}}^p_c}{\#\overline{\mathcal{A}}^p} \cdot \#\Big(F^\times / U_F^{(c)}F^{\times p}\Big). 
    $$
    The result then follows by Lemma~\ref{lem-chars-to-Cp-exts-(p-1)-to-1}. 
\end{proof}
\begin{corollary}
    \label{cor-size-of-wildly-ram-Cp-exts-with-norms}
    Let $c$ be an integer with $0 \leq c \leq \frac{pe_F}{p-1} + 1$. We have 
    $$
    \# \Et{p/F, \leq (p-1)c}{C_p/F,\mathcal{A}} = \frac{1}{p-1}\cdot \Big(
    \frac{\#\overline{\mathcal{A}}^p_c}{\#\overline{\mathcal{A}}^p}\Big(1 + (p-1)\#\Et{p/F,\leq (p-1)c}{C_p/F}\Big) - 1
    \Big).
    $$
\end{corollary}
\begin{proof}
    This is immediate from Corollary~\ref{cor-num-(1^p)-exts-in-terms-of-quotient} and Lemma~\ref{lem-wildly-ram-Cp-exts-with-norms-and-bounded-disc-basic-form}.
\end{proof}
\begin{corollary}
    \label{cor-num-wild-Cp-exts-with-A-and-disc-val-less-explicit-form}
    Let $c$ be an integer with $1 \leq c \leq \frac{pe_F}{p-1} + 1$. We have 
    $$
    \# \Et{p/F, (p-1)c}{C_p/F,\mathcal{A}} = \frac{\#\big(F^{\times p}/ U_F^{(c-1)}F^{\times p}\big)}{(p-1)\#\overline{\mathcal{A}}^p}\cdot \Big(\#\overline{\mathcal{A}}_c^p \cdot \# W_{c-1} - \#\overline{\mathcal{A}}^p_{c-1}\Big)
    $$
\end{corollary}
\begin{proof}
    This follows immediately from Lemma~\ref{lem-wildly-ram-Cp-exts-with-norms-and-bounded-disc-basic-form}, together with the definition of the groups $W_i$. 
\end{proof}

\begin{corollary}
    \label{cor-num-wild-Cp-exts-with-A-and-disc-val}
    For $1 \leq c \leq \frac{pe_F}{p-1} + 1$, we have 
    $$
    \#\Et{p/F,(p-1)c}{C_p/F,\mathcal{A}} = \begin{cases}
        \frac{p}{(p-1)\#\overline{\mathcal{A}}^p} q^{c - 2- \lfloor\frac{c-2}{p}\rfloor} \Big(q\cdot \#\overline{\mathcal{A}}^p_c - \#\overline{\mathcal{A}}^p_{c-1}\Big)\quad\text{if $c \not\equiv 1\pmod{p}$},
        \\
        \frac{pq^{e_F}}{\#\overline{\mathcal{A}}^p}\quad\text{if $c = \frac{pe_F}{p-1} + 1$ and $\mu_p\subseteq F$ and $\overline{\mathcal{A}}^p_{\frac{pe_F}{p-1}} = 0$},
        \\
        0\quad\text{otherwise}. 
    \end{cases}
    $$
\end{corollary}
\begin{proof}
    This follows easily from Corollary~\ref{cor-dim-of-F-mod-Ut-and-pth-powers}, Theorem~\ref{thm-num-Cp-exts-with-disc-val}, and Lemma~\ref{lem-wildly-ram-Cp-exts-with-norms-and-bounded-disc-basic-form}. 
\end{proof}
\begin{corollary}
    \label{cor-mass-1^p-Cp-A}
    The mass $\m\big(\Et{(1^p)/F}{C_p/F,\mathcal{A}}\big)$ is given by
    $$
    \m\big(\Et{(1^p)/F}{C_p/F,\mathcal{A}}\big) = \mathbbm{1}_{\mu_p\subseteq F} \cdot \mathbbm{1}_{\overline{\mathcal{A}}^p_{\frac{pe_F}{p-1}} = 0}\cdot \frac{q^{-(p-1)(e_F+1)}}{\#\overline{\mathcal{A}}^p} +  \frac{1}{(p-1)\#\overline{\mathcal{A}}^p}\cdot q^{-2}\cdot\sum_{\substack{1 \leq c \leq \lceil \frac{pe_F}{p-1}\rceil \\ c \not \equiv 1 \pmod{p}}} \frac{q\#\overline{\mathcal{A}}^p_c - \#\overline{\mathcal{A}}^p_{c-1}}{q^{(p-2)c + \lfloor\frac{c-2}{p}\rfloor}}.
        $$
\end{corollary}
\begin{proof}
    This is immediate from Theorem~\ref{thm-num-Cp-exts-with-disc-val} and Corollary~\ref{cor-num-wild-Cp-exts-with-A-and-disc-val}.
\end{proof}

\begin{proof}
    [Proof of Theorem~\ref{thm-m-A-p-for-prime-n}] 
    We prove the statements one by one. 
    \begin{enumerate}
        \item This is immediate from Corollary~\ref{cor-m-et-ell-F-A} and Lemma~\ref{lem-mass-Et-ell-epi}. 
        \item This is precisely Lemma~\ref{lem-mass-of-l-ic-unram}.
        \item This is precisely Corollary~\ref{cor-mass-1^ell-tame-ram}.
        \item 
        Serre's mass formula \cite[Theorem~2]{serre} tells us that
        $$
        \m\big(\Et{(1^p)/F}{}\big) = \frac{1}{q^{p-1}},
        $$
        and the first part of the statement follows from Lemma~\ref{lem-m-Et-1^p-A-in-terms-of-other-quantities}. The rest of the theorem is given by Corollaries~\ref{cor-mass-1^p-Cp} and  \ref{cor-mass-1^p-Cp-A}.
    \end{enumerate}
\end{proof}

\begin{definition}
    \label{defi-c-alpha}
    Let $\alpha \in F^\times$. We define $c_\alpha \in \Z \cup \{\infty\}$ as follows:
    \begin{itemize}
        \item If $\alpha \in F^{\times p}$, then $c_\alpha = \infty$. 
        \item Otherwise, adopting the convention that $U_F^{(-1)} = F^\times$, we define $c_\alpha$ to be the largest integer $c$ such that $\alpha \in U_F^{(c)}F^{\times p}$. 
    \end{itemize}
\end{definition}
\begin{remark}
    We can compute $c_\alpha$ using Algorithm~\ref{algo-c-alpha}. We will state the time complexity of this computation in Lemma~\ref{lem-time-complexity-of-c-alpha-algo}.
\end{remark}

\begin{lemma}
    \label{lem-classification-of-c-alpha}
    Let $\alpha \in F^\times \setminus F^{\times p}$. Then we have $-1 \leq c_\alpha \leq \lfloor \frac{pe_F}{p-1} \rfloor$. Moreover, if $c_\alpha < \frac{pe_F}{p-1}$ then $p\nmid c_\alpha$, and if $c_\alpha = \frac{pe_F}{p-1}$ then $\mu_p\subseteq F$.
\end{lemma}
\begin{proof}
    This follows from Corollaries~\ref{cor-p-power-iff-p-power-mod-p-pi} and \ref{cor-size-of-Wi}. 
\end{proof}
\begin{lemma}
    \label{lem-num-Cp-exts-with-disc-and-single-elt-norm}
    Let $\alpha \in F^\times$ and let $c$ be an integer with $1 \leq c \leq \frac{pe_F}{p-1} + 1$. Then we have 
    \begin{enumerate}
        \item If $c \leq c_\alpha$, then we have 
        $$
        \#\Et{p/F,(p-1)c}{C_p/F,\alpha} = \begin{cases}
            \frac{p(q-1)}{p-1}\cdot q^{c-2-\lfloor\frac{c-2}{p}\rfloor}\quad\text{if $c \not \equiv 1\pmod{p}$},
            \\
            pq^{e_F}\quad \text{if $c = \frac{pe_F}{p-1} + 1$ and $\mu_p \subseteq F$},
            \\
            0\quad\text{otherwise}. 
        \end{cases}
        $$
        \item If $1 \leq c_\alpha < \frac{pe_F}{p-1}$, then  
        $$
        \#\Et{p/F,(p-1)(c_\alpha+1)}{C_p/F,\alpha} = \frac{q - p}{p-1} \cdot q^{c_\alpha - 1 - \lfloor\frac{c_\alpha}{p}\rfloor}. 
        $$
        \item If $c_\alpha = \frac{pe_F}{p-1}$, then 
        $$
        \# \Et{p/F, (p-1)(c_\alpha+1)}{C_p/F, \alpha} = 0.
        $$
        \item If $c_\alpha + 2 \leq c \leq \frac{pe_F}{p-1} + 1$, then 
        $$
        \#\Et{p/F,(p-1)c}{C_p/F,\alpha} = \begin{cases}
            \frac{q-1}{p-1}q^{c-2-\lfloor\frac{c-2}{p}\rfloor}\quad\text{if $c \not \equiv 1\pmod{p}$},
            \\
            q^{e_F}\quad\text{if $c = \frac{pe_F}{p-1} + 1$ and $\mu_p\subseteq F$},
            \\
            0\quad\text{otherwise}. 
        \end{cases}
        $$
    \end{enumerate}
\end{lemma}
\begin{proof} 
    This follows easily from Corollary~\ref{cor-size-of-Wi}, Corollary~\ref{cor-dim-of-F-mod-Ut-and-pth-powers}, Corollary~\ref{cor-num-wild-Cp-exts-with-A-and-disc-val-less-explicit-form}, and Lemma~\ref{lem-classification-of-c-alpha}. 
\end{proof}

\begin{theorem}
    \label{thm-closed-form-for-m-alpha-p-with-prime-n}
    Let $\alpha \in F^\times \setminus F^{\times p}$. If $p\nmid v_F(\alpha)$, then 
    $$
    \m\big(\Et{(1^p)/F}{C_p/F,\alpha}\big) = \frac{1}{p}\cdot \m\big(\Et{(1^p)/F}{C_p/F}\big).
    $$
    If $p \mid v_F(\alpha)$, then $\m\big(\Et{(1^p)/F}{C_p/F,\alpha}\big)$ is the sum of the following seven quantities:
    \begin{enumerate}
        \item $$\mathbbm{1}_{c_\alpha \geq p}\cdot \frac{q-1}{p-1}\cdot q^{-2}\cdot A(c_\alpha).$$
        \item $$\mathbbm{1}_{c_\alpha\not\equiv0,1\pmod{p}}\cdot \frac{q-1}{p-1}\cdot q^{-2}\cdot B(c_\alpha).$$
        \item $$\mathbbm{1}_{c_\alpha < \frac{pe_F}{p-1}}\cdot \frac{q-p}{p-1}\cdot\frac{1}{p}\cdot q^{-2 - (p-2)(c_\alpha+1) - \lfloor\frac{c_\alpha}{p}\rfloor}.$$
        \item $$\mathbbm{1}_{c_\alpha < \frac{pe_F}{p-1} - 1}\cdot \frac{q-1}{pq^2(p-1)} \cdot \Big(\mathbbm{1}_{e_F \geq p-1}\cdot A\Big(\Big\lceil \frac{pe_F}{p-1} \Big\rceil\Big) - \mathbbm{1}_{c_\alpha \geq p - 1}\cdot A(c_\alpha + 1)\Big).$$
        \item $$\mathbbm{1}_{c_\alpha < \frac{pe_F}{p-1} - 1}\cdot \mathbbm{1}_{(p-1)\nmid e_F}\cdot \frac{q-1}{pq^2(p-1)}\cdot B\Big(\Big\lceil\frac{pe_F}{p-1}\Big\rceil\Big).$$
        \item $$-\mathbbm{1}_{c_\alpha < \frac{pe_F}{p-1} - 1}\cdot \mathbbm{1}_{c_\alpha\not\equiv -1,0\pmod{p}}\cdot \frac{q-1}{pq^2(p-1)}\cdot B(c_\alpha + 1).$$
        \item $$ 
        \mathbbm{1}_{c_\alpha < \frac{pe_F}{p-1}}\cdot \mathbbm{1}_{\mu_p\subseteq F}\cdot \frac{1}{p}\cdot q^{-(p-1)(e_F + 1)}.
        $$
    \end{enumerate}
\end{theorem}
\begin{proof}
    If $p\nmid v_F(\alpha)$, then the result follows from Corollary~\ref{cor-size-of-wildly-ram-Cp-exts-with-norms}, so we will assume that $p \mid v_F(\alpha)$. Lemma~\ref{lem-classification-of-c-alpha} tells us that $c_\alpha \geq 1$. By Lemma~\ref{lem-num-Cp-exts-with-disc-and-single-elt-norm}, the pre-mass is the sum of the following four quantities:
    \begin{enumerate}
        \item $$
        \frac{q-1}{p-1}\cdot q^{-2}\cdot \sum_{\substack{1 \leq c \leq c_\alpha \\ c\not\equiv 1\pmod{p}}} q^{-(p-2)c - \lfloor \frac{c-2}{p}\rfloor}.
        $$
        \item $$
        \mathbbm{1}_{c_\alpha < \frac{pe_F}{p-1}}\cdot 
        \frac{q-p}{p-1}\cdot\frac{1}{p}\cdot q^{-2 - (p-2)(c_\alpha+1) - \lfloor\frac{c_\alpha}{p}\rfloor}.
        $$
        \item $$
        \frac{q-1}{p-1}\cdot \frac{1}{p}\cdot q^{-2} \cdot \sum_{\substack{c_\alpha + 2 \leq c < \frac{pe_F}{p-1} + 1 \\ c \not \equiv 1\pmod{p}}} q^{-(p-2)c - \lfloor\frac{c-2}{p}\rfloor}.
        $$
        \item $$
        \mathbbm{1}_{(p-1)\mid e_F}\cdot \mathbbm{1}_{c_\alpha < \frac{pe_F}{p-1}}\cdot \mathbbm{1}_{\mu_p\subseteq F}\cdot \frac{1}{p}\cdot q^{-(p-1)(e_F + 1)}.
        $$  
    \end{enumerate}
    By Lemma~\ref{lem-preliminary-sum-for-Cp-mass}, the first three quantities rearrange to:
    \begin{enumerate}
        \item $$
        \frac{q-1}{p-1}\cdot q^{-2} \cdot (\mathbbm{1}_{c_\alpha \geq p}\cdot A(c_\alpha) + \mathbbm{1}_{c_\alpha\not\equiv 0,1\pmod{p}}B(c_\alpha)).
        $$
        \item $$
        \mathbbm{1}_{c_\alpha < \frac{pe_F}{p-1}}\cdot \frac{q-p}{p-1}\cdot\frac{1}{p}\cdot q^{-2 - (p-2)(c_\alpha+1) - \lfloor\frac{c_\alpha}{p}\rfloor}.
        $$
        \item \begin{align*}
        \mathbbm{1}_{c_\alpha < \frac{pe_F}{p-1} - 1}\cdot \frac{q-1}{p-1}\cdot \frac{1}{p}\cdot q^{-2}\cdot \Big(&\mathbbm{1}_{e_F\geq p-1}\cdot A\Big(\Big\lceil \frac{pe_F}{p-1}\Big\rceil\Big) - \mathbbm{1}_{c_\alpha \geq p-1}A(c_\alpha + 1) 
        \\
        &+ \mathbbm{1}_{\lceil\frac{pe_F}{p-1}\rceil \not\equiv 0,1\pmod{p}}B\Big(\Big\lceil\frac{pe_F}{p-1}\Big\rceil\Big) - \mathbbm{1}_{c_\alpha\not\equiv -1,0\pmod{p}}B(c_\alpha + 1)\Big).
        \end{align*}
    \end{enumerate}
    We saw in the proof of Corollary~\ref{cor-mass-1^p-Cp} that $\big\lceil\frac{pe_F}{p-1}\big\rceil \equiv 0,1\pmod{p}$ if and only if $(p-1)\mid e_F$, so 
    $$
    \mathbbm{1}_{\lceil\frac{pe_F}{p-1}\rceil \not\equiv 0,1\pmod{p}} = \mathbbm{1}_{(p-1)\nmid e_F}. 
    $$ 
    The result \cite[Lemma~14.6]{washington1997introduction} states that $\Q_p(\sqrt[p-1]{-p}) = \Q_p(\zeta_p)$, and it follows that 
    $$
    \mathbbm{1}_{(p-1)\mid e_F}\cdot \mathbbm{1}_{\mu_p\subseteq F} = \mathbbm{1}_{\mu_p\subseteq F},
    $$
    and the result follows. 
\end{proof}
\section{$S_4$-quartic extensions}
\label{sec-S4-quartics}

Throughout Section~\ref{sec-S4-quartics}, let $p$ be a rational prime, let $F$ be a $p$-adic field, and let $\mathcal{A}\subseteq F^\times$ be a finitely generated subgroup. With $F$ fixed, write $q$ for the size $\# \F_F$ of the residue field of $F$. The goal of this section is to understand the pre-mass $\m\big(\Et{4/F}{\mathcal{A}}\big)$. The main results are Theorems~\ref{thm-n=4-for-odd-primes}, \ref{thm-1^21^2-2-adic}, \ref{thm-2^2-2-adic}, and \ref{thm-1^4-2-adic}.  

\subsection{Tamely ramified parts}
The goal of this subsection is to prove the following theorem: 

\begin{theorem}
    \label{thm-n=4-for-odd-primes}
    For all $p$, we have 
    $$
    \m\big(\Et{4/F}{\mathcal{A}}\big) = \frac{5q^2 + 8q + 8}{8q^2} + \sum_{\sigma \in \{(4),(22),(1^21^2),(2^2),(1^4)\}}\m\big(\Et{\sigma/F}{\mathcal{A}}\big).
    $$
    Moreover, the following statements are true: 
    \begin{enumerate}
        \item For all $p$, we have 
        $$
        \m\big(\Et{(4)/F}{\mathcal{A}}\big) = \begin{cases}
            \frac{1}{4} \quad\text{if $4 \mid v_F(\alpha)$ for all $\alpha \in \mathcal{A}$},
            \\
            0 \quad\text{otherwise},
        \end{cases}
        $$
        and
        $$
        \m\big(\Et{(22)/F}{\mathcal{A}}\big) = \begin{cases}
            \frac{1}{8} \quad\text{if $2 \mid v_F(\alpha)$ for all $\alpha \in \mathcal{A}$},
            \\
            0 \quad\text{otherwise}.
        \end{cases}
        $$
        \item Suppose that $p \neq 2$. Let $(A_0, A_1)$ be a stratified generating set for $\overline{\mathcal{A}}^2$. 
        Then we have 
        $$
        \m\big(\Et{(1^21^2)/F}{\mathcal{A}}\big) = \begin{cases}
            \frac{1}{2q^2}\quad\text{if $\mathcal{A}\subseteq F^{\times 2}$},
            \\
            \frac{3}{8q^2}\quad\text{else if $A_0 = \varnothing$ and $\# A_1 = 1$},
            \\
            \frac{1}{4q^2}\quad\text{otherwise}.
        \end{cases}
        $$
        \item \label{point-defi-of-A_r-for-r=0,1,2} Suppose that $p \neq 2$. Let $(A_0, A_1, A_2)$ be a stratified generating set for $\overline{\mathcal{A}}^4$. 
        Then we have 
        $$
        \m\big(\Et{(2^2)/F}{\mathcal{A}}\big) = \begin{cases}
            \frac{1}{2q^2}\quad\text{if $A_0 \subseteq F^{\times 2}$ and $A_1 = A_2 = \varnothing$},
            \\
            \frac{1}{4q^2}\quad\text{else if $A_0\subseteq F^{\times 2}$ and $A_1 = \varnothing$ and $\frac{\alpha_i}{\alpha_j}\in F^{\times 2}$ for all $\alpha_i,\alpha_j\in A_2$},
            \\
            0\quad\text{otherwise}.
        \end{cases}
        $$
        If $q \equiv 1\pmod{4}$, then we have 
        $$
        \m\big(\Et{(1^4)/F}{\mathcal{A}}\big) =\begin{cases}
            \frac{1}{q^3} \quad\text{if $\mathcal{A}\subseteq F^{\times 4}$},
            \\
            \frac{1}{2q^3}\quad\text{if $A_0 = A_1 = \varnothing$ and $\#A_2 = 1$ and $A_2 \subseteq F^{\times 2}$,}
            \\
            \frac{1}{4q^3}\quad\text{if $A_0 = A_2 = \varnothing$ and $\# A_1 = 1$},
            \\
            0 \quad\text{otherwise}. 
        \end{cases}
        $$
        If $q \equiv 3\pmod{4}$, then 
        $$
        \m\big(\Et{(1^4)/F}{\mathcal{A}}\big) = \begin{cases}
            \frac{1}{q^3}\quad\text{if $\mathcal{A}\subseteq F^{\times 2}$},
            \\
            \frac{1}{2q^3}\quad\text{if $A_0 = A_2 = \varnothing$ and $\# A_1 = 1$},
            \\
            0\quad\text{otherwise}. 
        \end{cases}
        $$
    \end{enumerate}
\end{theorem}

\begin{lemma}
    \label{lem-list-quartic-splitting-symbols}
    We have 
    $$
    \Split_4^{\epi} = \{(13),(1^22), (112), (1^211), (1^31), (1111)\}
    $$
    and 
    $$
    \Split_4 = \{(4), (22), (1^21^2), (2^2), (1^4)\} \cup \Split_4^{\epi}. 
    $$
\end{lemma}
\begin{proof}
    The eleven possible splitting symbols are listed on \cite[Page~1353]{bhargava-hcl3}, and it is clear that the epimorphic ones are as stated.
\end{proof}

\begin{lemma}
    \label{lem-mass-of-quartic-epis}
    For all $p$, we have 
    $$
    \m\big(\Et{4/F}{\epi}\big) = \frac{5q^2 + 8q + 8}{8q^2}. 
    $$
\end{lemma}
\begin{proof}
    This follows from Lemmas~\ref{lem-mass-of-all-algebras-with-symbol} and \ref{lem-list-quartic-splitting-symbols}. 
\end{proof}
\begin{lemma}
    \label{lem-mass-of-(4)-and-(22)}
    For all $p$, we have 
    $$
    \m\big(\Et{(4)/F}{\mathcal{A}}\big) = \begin{cases}
        \frac{1}{4}\quad\text{if $4 \mid v_F(\alpha)$ for all $\alpha \in \mathcal{A}$},
        \\
        0 \quad\text{otherwise},
    \end{cases}
    $$
    and 
    $$
    \m\big(\Et{(22)/F}{\mathcal{A}}\big) = \begin{cases}
        \frac{1}{8}\quad\text{if $2 \mid v_F(\alpha)$ for all $\alpha \in \mathcal{A}$},
        \\
        0 \quad\text{otherwise}.
    \end{cases}
    $$
\end{lemma}
\begin{proof}
    This follows easily from Lemma~\ref{lem-norm-group-of-pred-symbol}. 
\end{proof}
\begin{lemma}
    \label{lem-1^21^2-odd-p}
    Suppose that $p$ is odd. Let $(A_0, A_1)$ be a stratified generating set for $\overline{\mathcal{A}}^2$.  
    Then we have 
    $$
    \m\big(\Et{(1^21^2)/F}{\mathcal{A}}\big) = \begin{cases}
        \frac{1}{2q^2}\quad\text{if $\mathcal{A}\subseteq F^{\times 2}$},
        \\
        \frac{3}{8q^2}\quad \text{if $A_0 = \varnothing$ and $\# A_1 = 1$},
        \\
        \frac{1}{4q^2}\quad\text{otherwise}. 
    \end{cases}
    $$
\end{lemma}
\begin{proof}
    By Lemma~\ref{lem-tot-tame-ram-classification}, we have 
    $$
    \Et{(1^21^2)/F}{} = \{L_0\times L_0, L_0\times L_1, L_1\times L_1\}, 
    $$
    where 
    $$
    L_j = F\Big(\sqrt{\zeta_{q-1}^j\pi_F}\Big),\quad j=0,1.
    $$
    Lemma~\ref{lem-disc-val-of-tame-extension} tells us that $v_F(d_{L_j/F}) = 1$ for each $j$. It follows that, for $i,j \in \{0,1\}$, we have  
    $$
    \m\big(\{L_i \times L_j\}\big) = \begin{cases}
        \frac{1}{8q^2}\quad\text{if $i = j$},
        \\
        \frac{1}{4q^2}\quad\text{if $i\neq j$}. 
    \end{cases}
    $$
    The result then follows from the fact that 
    $$
    N_{L_j/F}L_j^\times = \langle \co_F^{\times 2}, -\zeta_{q-1}^j\pi_F\rangle.
    $$
\end{proof}
\begin{lemma}
    \label{lem-2^2-and-1^4-odd-p}
    Suppose that $p$ is odd. Let $(A_0, A_1, A_2)$ be a stratified generating set for $\overline{\mathcal{A}}^4$. 
    Then we have 
    $$
    \m\big(\Et{(2^2)/F}{\mathcal{A}}\big) = \begin{cases}
        \frac{1}{2q^2} \quad\text{if $A_0 \subseteq F^{\times 2}$ and $A_1 = A_2 = \varnothing$},
        \\
        \frac{1}{4q^2} \quad\text{else if $A_0 \subseteq F^{\times 2}$ and $A_1 = \varnothing$ and $\frac{\alpha_i}{\alpha_j} \in F^{\times 2}$ for all $\alpha_i,\alpha_j \in A_2$},
        \\
        0 \quad\text{otherwise}.
    \end{cases}
    $$
    If $q \equiv 1 \pmod{4}$, then we have 
    $$
    \m\big(\Et{(1^4)/F}{\mathcal{A}}\big) = \begin{cases}
        \frac{1}{q^3}\quad\text{if $\mathcal{A}\subseteq F^{\times 4}$},
        \\
        \frac{1}{2q^3}\quad\text{if $A_0 = A_1 = \varnothing$ and $\# A_2 = 1$ and $A_2 \subseteq F^{\times 2}$},
        \\
        \frac{1}{4q^3}\quad\text{if $A_0 = A_2 = \varnothing$ and $\# A_1 = 1$},
        \\
        0\quad\text{otherwise}.
    \end{cases}
    $$
    If $q \equiv 3\pmod{4}$, then we have 
    $$
    \m\big(\Et{(1^4)/F}{\mathcal{A}}\big) = \begin{cases}
        \frac{1}{q^3}\quad\text{if $\mathcal{A}\subseteq F^{\times 2}$},
        \\
        \frac{1}{2q^3}\quad\text{else if $A_0 = A_2 = \varnothing$ and $\# A_1 = 1$},
        \\
        0\quad\text{otherwise}. 
    \end{cases}
    $$
\end{lemma}

\begin{proof}
    Let $E/F$ be the quadratic unramified extension. Lemma~\ref{lem-tot-tame-ram-classification} tells us that 
    $$
    \Et{(1^2)/E}{} = \{L_0, L_1\},
    $$
    where 
    $$
    L_j = E\Big(\sqrt{\zeta_{q^2-1}^j\pi_F}\Big),
    $$
    for each $j$. On the other hand, we have two index $4$ subgroups 
    $$
    \langle \co_F^{\times 2}, \pi_F^2\rangle, \quad \langle \co_F^{\times 2}, \zeta_{q-1}\pi_F^2\rangle
    $$
    of $F^\times$. Clearly the quartic abelian extensions corresponding to these subgroups have splitting symbol $(2^2)$, so they must be equal to $L_0$ and $L_1$ in some order, which implies that $L_0$ and $L_1$ are nonisomorphic abelian extensions of $F$. Lemma~\ref{lem-disc-val-of-tame-extension} tells us that $v_F(d_{L_j/F}) = 2$ for each $j$, so $\m(\{L_j\})= \frac{1}{4q^2}$, and the formula for $\m\big(\Et{(2^2)/F}{\mathcal{A}}\big)$ follows. 

    Lemma~\ref{lem-tot-tame-ram-classification} tells us that 
    $$
    \Et{(1^4)/F}{} = \begin{cases}
        \{L_0, L_1,L_2,L_3\} \quad\text{if $q \equiv 1\pmod{4}$},
        \\
        \{L_0, L_1\}\quad\text{if $q\equiv 3\pmod{4}$},
    \end{cases}
    $$
    where 
    $$
    L_j = F\Big(\sqrt[4]{\zeta_{q-1}^j\pi_F}\Big). 
    $$
    Suppose first that $q \equiv 1 \pmod{4}$. Then the minimal polynomial of $\sqrt[4]{\pi_F}$ over $F$ splits in $L_0$, so $L_0/F$ is Galois. Since $\pi_F$ is an arbitrary choice of uniformiser, it follows that $L_j/F$ is Galois for each $j$, and therefore 
    $$
    N_{L_j/F}L_j^\times = \langle \co_F^{\times 4}, -\zeta_{q-1}^j\pi_F\rangle.
    $$
    Lemma~\ref{lem-disc-val-of-tame-extension} tells us that 
    $
    \m(\{L_j\}) = \frac{1}{4q^3}
    $ 
    for each $j$, and the result follows. Finally, suppose that $q\equiv 3\pmod{4}$. Then $\mu_4\not\subseteq F$, so $\sqrt[4]{\pi_F}$ has only two conjugates in $L_0$, and therefore $L_0$ is non-Galois with a maximal abelian subextension $F(\sqrt{\pi_F})$, so 
    $$
    N_{L_0/F}L_0^\times = \langle \co_F^{\times 2},-\pi_F\rangle,
    $$
    and similarly $L_1/F$ is non-Galois with 
    $$
    N_{L_1/F} L_1^\times = \langle \co_F^{\times 2}, - \zeta_{q-1}\pi_F\rangle. 
    $$
    By Lemma~\ref{lem-disc-val-of-tame-extension}, we have $v_F(d_{L_j/F}) = 3$ for each $j$, and the result follows. 
\end{proof}

\begin{proof}[Proof of Theorem~\ref{thm-n=4-for-odd-primes}]
    This is immediate from Lemmas~\ref{lem-list-quartic-splitting-symbols} to \ref{lem-2^2-and-1^4-odd-p} inclusive. 
\end{proof}

\subsection{$2$-adic fields} 

In this subsection, we specialise to the case $p=2$, so that $F$ is a $2$-adic field. Recall that, for an integer $m$, the decorators $\Et{\bullet/F, m}{\bullet}$ and $\Et{\bullet/F,\leq m}{\bullet}$ denote the sets of $L \in \Et{\bullet/F}{\bullet}$ with $v_F(d_{L/F}) = m$ and $v_F(d_{L/F}) \leq m$, respectively.  By Theorem~\ref{thm-n=4-for-odd-primes}, in order to compute $\m\big(\Et{4/F}{\mathcal{A}}\big)$, all we need is to compute $\m\big(\Et{\sigma/F}{\mathcal{A}}\big)$ for each $\sigma \in \{(1^21^2), (2^2), (1^4)\}$. We state the first two of these quantities in the next two theorems:

\begin{theorem}
    \label{thm-1^21^2-2-adic}
    The pre-mass $\m\big(\Et{(1^21^2)/F}{\mathcal{A}}\big)$ is equal to the sum of the following quantities:
    \begin{enumerate}
        \item $$
            \frac{1}{8}\cdot \sum_{4 \leq m \leq 4e_F + 2}q^{-m}\#\Et{(1^2)/F,\frac{m}{2}}{\mathcal{A}}, 
        $$
        which can be made explicit using Corollary~\ref{cor-num-wild-Cp-exts-with-A-and-disc-val}.
        \item $$
        \frac{1}{4}\cdot \sum_{4 \leq m \leq 4e_F + 2} q^{-m} N_{(1^21^2)}^{\neq}(m),
        $$
        where $N_{(1^21^2)}^{\neq}$ is the explicit function defined in Appendix~\ref{appendix-helpers}. 
    \end{enumerate}
\end{theorem}
\begin{proof}
    Postponed.
\end{proof}
Call a quadratic extension $E/F$ \emph{$C_4$-extendable} if it can be embedded into a $C_4$-extension of $F$. Let $E \in \Et{2/F}{C_2/F, \mathcal{A}}$ be $C_4$-extendable, and let $\omega \in E^\times$ be such that $E(\sqrt{\omega})$ is an element of $\Et{2/E}{C_4/F}$ with minimal discriminant valuation. For each $\alpha \in F^\times$, write 
$$
N_\alpha = N_{F(\sqrt{\alpha})/F}F(\sqrt{\alpha})^\times,
$$
and define
$$
N_\omega = N_{E(\sqrt{\omega})/F}E(\sqrt{\omega})^\times. 
$$
Moreover, given a choice of finitely generated subgroup $\mathcal{A}\subseteq F^\times$, write $\mathcal{G}_4(\mathcal{A})$ for a subset of $F^\times$, such that the image of the map $\mathcal{G}_4(\mathcal{A}) \to F^\times / F^{\times 4}$ generates $\overline{\mathcal{A}}^4$. We fix an arbitrary choice of $\mathcal{G}_4(\mathcal{A})$ for each possible subgroup $\mathcal{A}$. Define the set 
$$
N_\omega^{\mathcal{A}} = \bigcap_{\alpha \in \mathcal{G}_4(\mathcal{A}) \cap N_\omega} \overline{N}_\alpha^2 \setminus \bigcup_{\alpha \in \mathcal{G}_4(\mathcal{A}) \setminus N_\omega} \overline{N}_\alpha^2 \subseteq F^\times / F^{\times 2}. 
$$
For each such $E$, fix once and for all a choice of element $\omega$, and write 
$$
N_{E}^{\mathcal{A}} = N_\omega^{\mathcal{A}}.
$$
For each nonnegative integer $c$, define 
$$
N_{E,c}^\mathcal{A} = N_E^{\mathcal{A}} \cap \Big(U_F^{(c)}F^{\times 2}/F^{\times 2}\Big).
$$
Note that the sets $N_{E,c}^\mathcal{A}$ depend on our choices of element $\omega$ and $\mathcal{G}_4(\mathcal{A})$. However, we will see in Theorem~\ref{thm-2^2-2-adic} and \ref{thm-1^4-2-adic} that (at least for $C_4$-extendable $E$) their \emph{sizes} depend only on $E$ and $\mathcal{A}$, hence our notation.

\begin{theorem}
    \label{thm-2^2-2-adic}
    If $2\nmid v_F(\alpha)$ for some $\alpha\in\mathcal{A}$, then $\Et{(2^2)/F}{\mathcal{A}} = \varnothing$. Otherwise, the following four statements are true:
    \begin{enumerate}
        \item For $G \in \{S_4, A_4\}$, we have $\Et{(2^2)/F}{G/F} = \varnothing$, so 
        $$
        \m\big(\Et{(2^2)/F}{G/F,\mathcal{A}}\big) = 0.
        $$
        \item We have 
        $$
        \m\big(\Et{(2^2)/F}{D_4/F,\mathcal{A}}\big) = \frac{1}{2}\cdot\Big(q^{-2} - q^{-2e_F - 2} - \frac{1}{q^2+q+1}\Big(q^{-1} - q^{-3e_F-1}\Big) + q^{-3e_F - 2}\Big(q^{e_F} - 1\Big)\Big). 
        $$
        \item We have 
        $$
        \m\big(\Et{(2^2)/F}{V_4/F,\mathcal{A}}\big) = \frac{1}{4\#\overline{\mathcal{A}}^2}\Big( \mathbbm{1}_{\overline{\mathcal{A}}^2_{2e_F}=0}\cdot q^{-3e_F - 2} + 
        \sum_{c=1}^{e_F}q^{-3c-1}\Big(q\cdot \#\overline{\mathcal{A}}^2_{2c} - \#\overline{\mathcal{A}}^2_{2c-1}\Big)    
        \Big).
        $$
        \item  
        Let $E$ be the unique unramified quadratic extension of $F$. The pre-mass $\m\big(\Et{(2^2)/F}{C_4/F,\mathcal{A}}\big)$ is the sum of the following two quantities:
        \begin{enumerate}
            \item $$
            \frac{1}{8}\cdot \sum_{c=1}^{e_F}q^{-4c}(\# N_{E,2e_F - 2c}^{\mathcal{A}} - \# N_{E,2e_F - 2c + 2}^{\mathcal{A}}),
            $$
            \item $$
            \frac{1}{8} \cdot q^{-4e_F - 2} \cdot (\# N_E^\mathcal{A} - \#N_{E,0}^\mathcal{A}).  
            $$
        \end{enumerate}
    \end{enumerate}
\end{theorem}
\begin{proof}
    Postponed.
\end{proof}

For $\sigma = (1^4)$, the actual mass is too cumbersome to state in closed form, so we instead study the quantities $\#\Et{(1^4)/F,m}{G/F,\mathcal{A}}$ for each $G$. For $G \neq C_4$, we state these quantities explicitly in Theorem~\ref{thm-1^4-2-adic}. For $G = C_4$, also in Theorem~\ref{thm-1^4-2-adic}, we give a formula for $\#\Et{(1^2)/E,m_2}{C_4/F,\mathcal{A}}$, whenever $E$ is a totally ramified quadratic extension of $F$ and $m_2$ is an integer. This formula depends on $E$ in a nontrivial way, but it can be evaluated efficiently by Lemma~\ref{lem-complexity-of-finding-top-half-of-tower}, allowing us to compute $\#\Et{(1^4)/F,m}{C_4/F,\mathcal{A}}$ for each $m$, and hence the pre-mass $\m(\Et{(1^4)/F}{C_4/F,\mathcal{A}})$.

\begin{theorem}
    \label{thm-1^4-2-adic}
    The following four statements are true:
    \begin{enumerate}
        \item For $G \in \{S_4, A_4\}$, we have 
        $$
        \m\big(\Et{(1^4)/F}{G/F,\mathcal{A}}\big) = \m\big(\Et{(1^4)/F}{G/F}\big),
        $$
        which is known from \cite[Corollaries~1.6 and 1.7]{monnet2024counting}. 
        \item If $\Et{(1^4)/F,m}{D_4/F,\mathcal{A}}$ is nonempty, then one of the following three statements is true: $m$ is even with $6 \leq m \leq 8e_F + 2$; or $m \equiv 1\pmod{4}$ with $4e_F + 5 \leq m \leq 8e_F + 1$; or $m = 8e_F + 3$. For such $m$, we have  
        $$
        \#\Et{(1^4)/F,m}{D_4/F,\mathcal{A}} = \frac{1}{2}\sum_{0 < m_1 < m/2} \#\Et{(1^2)/F,m_1}{\mathcal{A}} \cdot\Big( N^{C_2}(m-2m_1) - N^{C_4}(m_1,m-2m_1) - N^{V_4}(m_1,m-2m_1)\Big),
        $$
        where $N^{C_2}$, $N^{C_4}$, and $N^{V_4}$ are the functions defined in Appenix~\ref{appendix-helpers}. We can compute this quantity explicitly using Corollary~\ref{cor-num-wild-Cp-exts-with-A-and-disc-val}. 
        \item If $\Et{(1^4)/F, m}{V_4/F, \mathcal{A}}$ is nonempty, then $m$ is an even integer with $6 \leq m \leq 6e_F + 2$. For such $m$, the quantity
        $
        \#\Et{(1^4)/F, m}{V_4/F, \mathcal{A}}
        $
        is equal to the sum of the following two quantities:
        \begin{enumerate}
            \item 
        $$
        \frac{1}{2}\cdot \sum_{\substack{m_1 < m_2 \\ m_1 + 2m_2 = m}} \big(\#\Et{(1^2)/F,m_1}{\mathcal{A}}\big)\big(\#\Et{(1^2)/F,m_2}{\mathcal{A}}\big).
        $$
        \item $$ \mathbbm{1}_{3\mid m}\cdot 
        \frac{2}{3(\#\overline{\mathcal{A}}^2)^2}\cdot q^{\frac{m}{3} - 2}\Big(q\# \overline{\mathcal{A}}^2_{m/3} - \#\overline{\mathcal{A}}^2_{m/3-1}\Big)\Big(q\# \overline{\mathcal{A}}^2_{m/3} - 2\#\overline{\mathcal{A}}^2_{m/3-1}\Big).
        $$
        \end{enumerate}
        This expression can be made explicit using Corollary~\ref{cor-num-wild-Cp-exts-with-A-and-disc-val}.
        \item 
        Let $E \in \Et{(1^2)/F}{\mathcal{A}}$ be $C_4$-extendable, and let $m_1 = v_F(d_{E/F})$. If $m_1 > e_F$, then 
        $$
        \#\Et{(1^2)/E, m_2}{C_4/F,\mathcal{A}} = \begin{cases}
            \frac{1}{2}\cdot\# N_E^{\mathcal{A}} \quad\text{if $m_2 = m_1 + 2e_F$},
            \\
            0\quad\text{otherwise}.
        \end{cases}
        $$
        Suppose instead that $m_1 \leq e_F$.
        For even integers $m_2$, define 
        $$
        c(m_2) = 2e_F - 2\Big\lfloor
        \frac{m_1+m_2}{4}
    \Big\rfloor.
        $$
        Then we have 
        $$
        \#\Et{(1^2)/E,m_2}{C_4/F,\mathcal{A}} = \begin{cases}
            \frac{1}{2}\cdot \# N_{E, c(m_2)}^{\mathcal{A}}\quad\text{if $m_2 = 3m_1 - 2$},
            \\
            \frac{1}{2}\cdot \Big( \# N_{E, c(m_2)}^{\mathcal{A}} - \# N_{E, c(m_2 - 2)}^{\mathcal{A}} \Big) \quad\text{if $3m_1 \leq m_2 \leq 4e_F - m_1 + 1$ and $m_2$ is even},
            \\
            \frac{1}{2}\cdot \Big(\#N^\mathcal{A}_E - \# N^\mathcal{A}_{E, c(m_2 - 2)}\Big)\quad\text{if $m_2 = 4e_F - m_1 + 2$},
            \\
            0\quad\text{otherwise}. 
        \end{cases}
        $$
    \end{enumerate}
\end{theorem} 
\begin{proof}
    Postponed.
\end{proof}
\begin{lemma}
    \label{lem-characterise-1^21^2-A}
    We have 
    $$
    \Et{(1^21^2)/F}{\mathcal{A}} = \big\{L_1 \times L_2 \in \Et{(1^21^2)/F}{} : L_1 \not \cong L_2\big\} \cup \big\{L_1 \times L_1 : L_1 \in \Et{(1^2)/F}{\mathcal{A}}\big\}.
    $$
\end{lemma}
\begin{proof}
    It is clear that the left-hand side is contained inside the right-hand side, and also that 
    $$
    \{L_1 \times L_1 : L_1 \in \Et{(1^2)/F}{\mathcal{A}}\} \subseteq \Et{(1^21^2)/F}{\mathcal{A}}.
    $$
    Finally, for distinct elements $L_1, L_2 \in \Et{(1^2)/F}{}$, the norm groups $N_{L_i/F}L_i^\times$ are distinct index $2$ subgroups of $F^\times$, and therefore $\Nm(L_1\times L_2) = F^\times$, so $L_1\times L_2 \in \Et{(1^21^2)/F}{\mathcal{A}}$.
\end{proof}
\begin{lemma}
    \label{lem-size-of-1^21^2-neq}
    Let $m$ be an integer. We have 
    $$
    \# \{L_1 \times L_2 \in \Et{(1^21^2)/F, m}{} : L_1 \not \cong L_2\}  = N_{(1^21^2)}^{\neq}(m)
    $$
\end{lemma}
\begin{proof}
    Theorem~\ref{thm-num-Cp-exts-with-disc-val} tells us that, for all $m_1$, we have 
    $$
    \# \Et{(1^2)/F,m_1}{} = \begin{cases}
        2(q-1)q^{\frac{m_1}{2} - 1}\quad\text{if $m_1$ is even with $2 \leq m_1 \leq 2e_F$},
        \\
        2q^{e_F}\quad\text{if $m_1 = 2e_F + 1$},
        \\
        0\quad\text{otherwise}. 
    \end{cases}
    $$
    We will use this fact without reference for the rest of this proof. For any $m$, the number we are looking for is equal to 
    \begin{equation}
        \frac{1}{2}\cdot \Big(\sum_{\substack{m_1 + m_2 = m}} \Big(\#\Et{(1^2)/F,m_1}{}\cdot \#\Et{(1^2)/F,m_2}{}\Big) -  \#\Et{(1^2)/F,m/2}{}\Big). 
    \end{equation}
    It is easy to see that this is $0$ unless one of the following is true:
    \begin{itemize}
        \item $4 \leq m \leq 4e_F$ and $m$ is even.
        \item $2e_F + 3 \leq m \leq 4e_F + 1$ and $m$ is odd.
        \item $m = 4e_F + 2$. 
    \end{itemize}
    The result follows by considering these cases separately.
\end{proof}

\begin{proof}
    [Proof of Theorem~\ref{thm-1^21^2-2-adic}]

    This is immediate from Lemmas~\ref{lem-characterise-1^21^2-A} and \ref{lem-size-of-1^21^2-neq}.
\end{proof}
\begin{lemma}
    \label{lem-disc-of-V4-subexts}
    Let $L/F$ be a $V_4$-extension with intermediate quadratic fields $E_1, E_2,$ and $E_3$. The following two statements are true:
    \begin{enumerate}
        \item We have $v_F(d_{L/F}) = \sum_{i=1}^3 v_F(d_{E_i/F})$.
        \item If $v_F(d_{E_1/F}) < v_F(d_{E_2/F})$, then $v_F(d_{E_3/F}) = v_F(d_{E_2/F})$. 
    \end{enumerate}
\end{lemma}
\begin{proof}
    The first statement is \cite[Lemma~2.17]{monnet2024counting}.
    For the second statement, suppose that $v_F(d_{E_1/F}) < v_F(d_{E_2/F})$. For each $i$, let $\chi_i:F^\times/F^{\times 2} \to C_2$ be the quadratic character associated to $E_i$, as in Lemma~\ref{lem-chars-to-Cp-exts-(p-1)-to-1}. Theorem~\ref{thm-conductor-discriminant} and Lemma~\ref{lem-cond-disc-for-Cp} tell us that 
    $$
    v_F(d_{E_i/F}) = \mff(E_i/F) = \mff(\chi_i),
    $$
    for each $i$. It is easy to see that $\chi_3 = \chi_1\chi_2$, so $\mff(\chi_3) = \mff(\chi_2)$, and the result follows. 
\end{proof}
\begin{lemma}
    \label{lem-num-V4-2^2-with-m-and-A}
    For each nonnegative integer $m$, we have 
    $$
    \# \Et{(2^2)/F, m}{V_4/F,\mathcal{A}} = \begin{cases}
        \frac{1}{2}\cdot \#\Et{(1^2)/F,m/2}{\mathcal{A}}\quad\text{if $2\mid v_F(\alpha)$ for all $\alpha \in \mathcal{A}$},
        \\
        0\quad\text{otherwise}. 
    \end{cases}
    $$
\end{lemma}
\begin{proof}
    Write $E_{\mathrm{ur}}$ for the unramified quadratic extension of $F$. Let $L \in \Et{(2^2)/F,m}{V_4/F}$. It follows from Lemma~\ref{lem-disc-of-V4-subexts} that there are exactly two elements $E \in \Et{(1^2)/F,m/2}{}$ with $L = E_{\mathrm{ur}}E$. For such $E$, we have
    $$
    \Nm L = \Nm E_{ur} \cap \Nm E = \{x \in \Nm E : 2\mid v_F(x)\}.
    $$
    Therefore, if there is some $\alpha \in \mathcal{A}$ with $2\nmid v_F(\alpha)$, then $\Et{(2^2)/F}{V_4/F,\mathcal{A}} = \varnothing$. On the other hand, if $2\mid v_F(\alpha)$ for all $\alpha\in \mathcal{A}$, then we have a $2$-to-$1$ surjection 
    $$
    \Et{(1^2)/F,m/2}{\mathcal{A}} \to \Et{(2^2)/F,m}{V_4/F,\mathcal{A}},\quad E \mapsto E_\mathrm{ur}E,
    $$
    and the result follows. 
\end{proof}
\begin{lemma}
    \label{lem-2^2-D4-exts-with-A}
    We have 
    $$
    \Et{(2^2)/F}{D_4/F,\mathcal{A}} = \begin{cases}
        \Et{(2^2)/F}{D_4/F}\quad\text{if $2\mid v_F(\alpha)$ for all $\alpha \in \mathcal{A}$},
        \\
        \varnothing \quad\text{otherwise}.
    \end{cases}
    $$
\end{lemma}
\begin{proof}
    Let $E/F$ be the unramified quadratic extension. By class field theory, every element $L \in \Et{(2^2)/F}{D_4/F}$ has $N_{L/F}L^\times = N_{E/F}E^\times$, and the result follows. 
\end{proof}
\begin{lemma}
    \label{lem-(2^2)-D_4-in-terms-of-quad-exts-of-E}
    Let $E/F$ be the unique unramified quadratic extension. For each nonnegative integer $m$, we have 
    $$
    \#\Et{(2^2)/F, m}{D_4/F} =
        \frac{1}{2}\cdot\Big(\# \Et{(1^2)/E,m/2}{} - \# \Et{(1^2)/E,m/2}{C_4/F} - \# \Et{(1^2)/E,m/2}{V_4/F}\Big).
    $$
\end{lemma}
\begin{proof}
    Using the tower law for discriminant, it is easy to see that there is a well-defined surjection 
    $$
    \Et{(1^2)/E,\frac{m}{2}}{} \setminus\Big(\Et{(1^2)/E,\frac{m}{2}}{C_4/F}\cup \Et{(1^2)/E,\frac{m}{2}}{V_4/F}\Big) \to \Et{(2^2)/F,m}{D_4/F}. 
    $$
    Moreover, \cite[Lemma~5.1, Part (2)]{monnet2024counting} tells us that this surjection is $2$-to-$1$.
\end{proof}
\begin{lemma}
    \label{lem-C4-exts-2-to-1-surjection}
    Let $E \in \Et{2/F}{C_2/F}$ and let $m_2$ be an integer such that $\Et{2/E, \leq m_2}{C_4/F}$ is nonempty. Let $\omega \in E^\times$ be such that $E(\sqrt{\omega}) \in \Et{2/E, \leq m_2}{C_4/F}$. If $m_2 \leq 2e_E$, then we have a $2$-to-$1$ surjection 
    $$
    \Big(U_E^{(2e_E - 2\lfloor\frac{m_2}{2}\rfloor)}E^{\times 2}\cap F^\times\Big)/F^{\times 2} \to \Et{2/E,\leq m_2}{C_4/F},\quad 
    t \mapsto E(\sqrt{\omega t}). 
    $$
    If $m_2 > 2e_E$, then we have a $2$-to-$1$ surjection 
    $$
    F^\times / F^{\times 2} \to \Et{2/E,\leq m_2}{C_4/F},\quad 
    t \mapsto E(\sqrt{\omega t}).
    $$
\end{lemma}
\begin{proof}
    By \cite[Proposition~1.2]{cohen-et-al}, there is a $2$-to-$1$ surjection 
    $$
    F^\times / F^{\times 2} \to \Et{2/E}{C_4/F},\quad t \mapsto E(\sqrt{\omega t}),
    $$
    and the result follows from \cite[Lemma~4.6 and Corollary~4.7]{monnet2024counting}.  
\end{proof}
\begin{lemma}
    \label{lem-UE2t-cap-F-unramified}
    Let $t \leq e_F$ be an integer and let $E/F$ be the unramified quadratic extension. Then 
    $$
    U_E^{(2t)}E^{\times 2} \cap F^\times = U_F^{(2t)}F^{\times 2}. 
    $$
\end{lemma}
\begin{proof}
    This is essentially \cite[Proposition~3.6]{cohen-et-al}. 
\end{proof}
\begin{lemma}
    \label{lem-(1^2)/E-C_4-and-V4/F-in-terms-of-(1^2)/F}
    Let $E/F$ be the unique unramified quadratic extension. For each nonnegative integer $m$, we have 
    $$
    \#\Et{(1^2)/E,m}{C_4/F} =
        \frac{1}{2}\cdot \#\Et{(1^2)/F, m}{}.
    $$
\end{lemma}
\begin{proof}
    If $\Et{2/E,m}{C_4/F}\neq \varnothing$, then either $m$ is even with $0 \leq m \leq 2e_F$, or $m = 2e_F + 1$. Suppose that $m$ is an even integer with $0 \leq m \leq 2e_F$. Let $\omega \in E^\times$ be such that $E(\sqrt{\omega})/E$ is unramified, hence a $C_4$-extension of $F$. By Lemma~\ref{lem-UE2t-cap-F-unramified}, we have 
    $$
    U_E^{(2e_E - m)}E^{\times 2} \cap F^\times = U_F^{(2e_F-m)}F^{\times 2}, 
    $$
    so Lemma~\ref{lem-C4-exts-2-to-1-surjection} gives a $2$-to-$1$ surjection 
    $$
    U_F^{(2e_F-m)}F^{\times 2}/F^{\times 2} \to \Et{2/E, \leq m}{C_4/F},\quad t \mapsto E(\sqrt{\omega t}).
    $$
    But \cite[Lemma~4.6 and Corollary~4.7]{monnet2024counting} tell us that, for $u \in F^\times / F^{\times 2}$, we have 
    $$
    v_F(d_{F(\sqrt{u})/F}) \leq m \iff u \in U_F^{(2e_F-m)}F^{\times 2}/F^{\times 2},
    $$
    and therefore 
    $$
    \#\Et{2/F,\leq m}{C_2/F} = \#(U_F^{(2e_F-m)}F^{\times 2}/F^{\times 2}) - 1,
    $$
    so
    $$
    \# \Et{2/E, \leq m}{C_4/F} = \frac{1}{2}\#\Et{2/F,\leq m}{C_2/F} + \frac{1}{2},
    $$
    and the result follows for $2 \leq m \leq 2e_F$. By Lemma~\ref{lem-C4-exts-2-to-1-surjection} and \cite[Lemma~4.6]{monnet2024counting}, there is a $2$-to-$1$ surjection
    $$
    \{[x] \in F^\times / F^{\times 2} : v_F(x) = 1\} \to \Et{(1^2)/E,2e_F+1}{C_4/F},\quad x \mapsto E(\sqrt{\omega x}),
    $$
    and a bijection 
    $$
    \{[x] \in F^\times / F^{\times 2} : v_F(x) = 1\} \to \Et{(1^2)/F, 2e_F + 1}{},\quad x\mapsto F(\sqrt{x}),
    $$
    and the result for $m = 2e_F + 1$ follows.
\end{proof}
\begin{lemma}
    \label{lem-num-V4-exts-half-num-C2-exts}
    Let $E/F$ be the unique unramified quadratic extension. For each nonnegative integer $m$, we have 
    $$
    \# \Et{(1^2)/E, m}{V_4/F} = \frac{1}{2} \cdot \# \Et{(1^2)/F, m}{}. 
    $$
\end{lemma}
\begin{proof}
    This follows easily from Lemma~\ref{lem-num-V4-2^2-with-m-and-A} and \cite[Lemma~5.1]{monnet2024counting}. 
\end{proof}
\begin{lemma}
    \label{lem-num-2^2-D4-with-norms-and-disc-val}
    If $\Et{(2^2)/F, m}{D_4/F,\mathcal{A}}$ is nonempty, then either $m$ is a multiple of $4$ with $4 \leq m \leq 4e_F$, or $m = 4e_F + 2$. If $m$ is a multiple of $4$ with $4 \leq m \leq 4e_F$, then  
    $$
    \#\Et{(2^2)/F, m}{D_4/F,\mathcal{A}} =\begin{cases} (q-1)\Big(
    (q+1)q^{\frac{m}{2}-2} - q^{\frac{m}{4} - 1}    
    \Big)\quad\text{if $2\mid v_F(\alpha)$ for all $\alpha \in \mathcal{A}$},
    \\
    0\quad\text{otherwise}. 
    \end{cases}
    $$
    If $m = 4e_F + 2$, then 
    $$
    \#\Et{(2^2)/F, m}{D_4/F,\mathcal{A}} = \begin{cases} q^{e_F}(q^{e_F} - 1) \quad\text{if $2\mid v_F(\alpha)$ for all $\alpha \in \mathcal{A}$},
        \\
        0\quad\text{otherwise}. 
        \end{cases}
    $$
\end{lemma}
\begin{proof}
    This follows easily from Theorem~\ref{thm-num-Cp-exts-with-disc-val} and Lemmas~\ref{lem-2^2-D4-exts-with-A}, \ref{lem-(2^2)-D_4-in-terms-of-quad-exts-of-E}, \ref{lem-(1^2)/E-C_4-and-V4/F-in-terms-of-(1^2)/F}, and \ref{lem-num-V4-exts-half-num-C2-exts}.
\end{proof}
\begin{lemma}
    \label{lem-2^2-D4-A-2-adic}
    If $2\mid v_F(\alpha)$ for all $\alpha \in \mathcal{A}$, then 
    $$
    \m\big(\Et{(2^2)/F}{D_4/F,\mathcal{A}}\big) = \frac{1}{2}\cdot\Big(q^{-2} - q^{-2e_F - 2} - \frac{1}{q^2+q+1}\big(q^{-1} - q^{-3e_F-1}\big) + q^{-3e_F - 2}\big(q^{e_F} - 1\big)\Big). 
    $$
\end{lemma}
\begin{proof}
    This follows easily from Lemma~\ref{lem-num-2^2-D4-with-norms-and-disc-val}. To eliminate the possibility of a manipulation error, we have checked the required summation numerically in the Python notebook in the Github repository \url{https://github.com/Sebastian-Monnet/Sn-n-ics-paper-checks}.
\end{proof}

\begin{lemma}
    \label{lem-norms-in-galois-tower}
    Let $M/L/K$ be a tower of quadratic extensions of $p$-adic fields. Suppose that $M/K$ is Galois and let $\beta \in L^\times$. Let $\alpha = N_{L/K}\beta$. Then 
    $$
    \alpha \in N_{M/K}M^\times \iff \beta \in N_{M/L}M^\times. 
    $$
\end{lemma}
\begin{proof}
    The $(\Leftarrow)$ direction is obvious. For $(\Rightarrow)$, suppose that $\alpha \in N_{M/K}M^\times$. Then $\alpha = N_{L/K}\theta$ for some $\theta \in N_{M/L}M^\times$. Since $N_{L/K}(\theta/\beta) = 1$, Hilbert's theorem 90 tells us that $\theta = \beta \frac{x}{\bar{x}}$ for some $x \in L^\times$. Writing $M = L(\sqrt{d})$ for $d \in L^\times$, we have 
    \begin{align*}
    (\beta, d)_L &= (d,\theta)_L(d,x)_L(d,\bar{x})_L 
    \\
    &= (d,x)_L (\bar{d},x)_L 
    \\
    &= 1,
    \end{align*}
    where the final equality comes from the fact that $M/K$ is Galois, so $L(\sqrt{d}) = L(\sqrt{\overline{d}}) = M$. 
\end{proof}
\begin{lemma}
    \label{lem-C4-bijection-with-norm-condition}
    Let $E \in \Et{2/F}{C_2/F, \mathcal{A}}$ be $C_4$-extendable, and let $m_2$ be an integer such that $\Et{2/E,\leq m_2}{C_4/F}$ is nonempty. If $m_2 \leq 2e_E$, then there is a $2$-to-$1$ surjection 
    $$
    \Big(U_E^{(2e_E - 2\lfloor\frac{m_2}{2}\rfloor)}E^{\times 2}\cap F^\times\Big)/ F^{\times 2} \cap N^{\mathcal{A}}_E \to \Et{2/E, \leq m_2}{C_4/F,\mathcal{A}}.  
    $$
    If $m_2 \geq 2e_E + 1$, then there is a $2$-to-$1$ surjection 
    $$
    N_E^\mathcal{A} \to \Et{2/E, \leq m_2}{C_4/F, \mathcal{A}}. 
    $$
\end{lemma}
\begin{proof}
    Let $\omega \in E^\times$ be the element we fixed in the definition of $N_E^\mathcal{A}$. If $m_2 \leq 2e_E$, then Lemma~\ref{lem-C4-exts-2-to-1-surjection} gives us a $2$-to-$1$ surjection 
    $$
    \Big(U_E^{(2e_E - 2\lfloor\frac{m_2}{2}\rfloor)}E^{\times 2}\cap F^\times\Big)/ F^{\times 2} \to \Et{2/E, \leq m_2}{C_4/F}, \quad t \mapsto E(\sqrt{\omega t}).
    $$
    If $m_2 \geq 2e_E + 1$, then Lemma~\ref{lem-C4-exts-2-to-1-surjection} gives a $2$-to-$1$ surjection 
    $$
    F^\times / F^{\times 2} \to \Et{2/E, \leq m_2}{C_4/F}, \quad t\mapsto E(\sqrt{\omega t}). 
    $$
    Let $t \in F^\times$ and let $L = E(\sqrt{\omega t})$. Let $\alpha \in \mathcal{A}$, and let $\widetilde{\alpha}\in E^\times$ be such that $N_{E/F}\widetilde{\alpha} = \alpha$. Lemma~\ref{lem-norms-in-galois-tower} tells us that $\alpha \in N_{L/F}L^\times$ if and only if $\widetilde{\alpha} \in N_{L/E}L^\times$. By Lemma~\ref{lem-norms-in-galois-tower} and properties of quadratic Hilbert symbols, we have 
    \begin{align*}
        \widetilde{\alpha} \in N_{L/E}L^\times &\iff (\widetilde{\alpha}, \omega t)_E  = 1 
        \\
        &\iff (\widetilde{\alpha}, t)_E = (\widetilde{\alpha}, \omega)_E 
        \\
        &\iff (\alpha, t)_F = \begin{cases}
        1 \quad \text{if $\alpha \in N_{E(\sqrt{\omega})/F}E(\sqrt{\omega})^\times,$}
        \\
        -1 \quad\text{otherwise.}
        \end{cases}
        \\
        &\iff \begin{cases}
            t \in N_\alpha \quad\text{if $\alpha \in N_\omega$},
            \\
            t \not \in N_\alpha \quad\text{otherwise}. 
        \end{cases}
    \end{align*} 
    It follows that $\mathcal{A}\subseteq N_{L/F}L^\times$ if and only if $t \in N_E^\mathcal{A}$, and the result follows. 
\end{proof}

\begin{lemma}
    \label{lem-size-of-Et-2/E-C4/F-A-for-half-ram}
    Let $E$ be the unramified quadratic extension of $F$, and let $m$ be a nonnegative integer. 
    If $m \leq 4e_F + 1$, then  
    $$
    \# \Et{(2^2)/F, \leq m}{C_4/F, \mathcal{A}} = \begin{cases}\frac{1}{2}\cdot \# N^\mathcal{A}_{E, 2e_F - 2\lfloor{\frac{m}{4}\rfloor}} - 1\quad\text{if $v_F(\alpha)\equiv 0\pmod{4}$ for all $\alpha\in\mathcal{A}$},
        \\
        \frac{1}{2}\cdot \# N^\mathcal{A}_{E, 2e_F - 2\lfloor{\frac{m}{4}\rfloor}} \quad\text{else if $2\mid v_F(\alpha)$ for all $\alpha\in\mathcal{A}$},
        \\
        0\quad\text{otherwise}. 
    \end{cases}
    $$
    If $m \geq 4e_F + 2$, then 
    $$
    \# \Et{(2^2)/F, \leq m}{C_4/F, \mathcal{A}} = \begin{cases}\frac{1}{2}\cdot \#N_E^\mathcal{A} - 1\quad\text{if $v_F(\alpha)\equiv 0\pmod{4}$ for all $\alpha\in\mathcal{A}$},
        \\
        \frac{1}{2}\cdot\#N_E^\mathcal{A} \quad\text{else if $2\mid v_F(\alpha)$ for all $\alpha\in\mathcal{A}$},
        \\
        0\quad\text{otherwise}. 
    \end{cases}
    $$
\end{lemma}
\begin{proof}
    Clearly, if $2\nmid v_F(\alpha)$ for any $\alpha \in \mathcal{A}$, then $\Et{(2^2)/F,\leq m}{C_4/F,\mathcal{A}} = \varnothing$. Therefore, we will assume that $2\mid v_F(\alpha)$ for all $\alpha\in\mathcal{A}$. By \cite[Lemma~5.1]{monnet2024counting}, the natural map 
    $$
    \Et{(1^2)/E,\leq \lfloor\frac{m}{2}\rfloor}{C_4/F} \to \Et{(2^2)/F,\leq m}{C_4/F}
    $$
    is a bijection. Let $m_2 = \lfloor \frac{m}{2}\rfloor$. Suppose that $m \leq 4e_F + 1$. Then Lemmas~\ref{lem-UE2t-cap-F-unramified} and \ref{lem-C4-bijection-with-norm-condition} tell us that 
    $$
    \# \Et{2/E,\leq m_2}{C_4/F, \mathcal{A}} = \frac{1}{2}\cdot \# N_{E, 2e_F - 2\lfloor\frac{m_2}{2}\rfloor}^{\mathcal{A}}.
    $$
    It is easy to see that $\lfloor \frac{m_2}{2}\rfloor = \lfloor \frac{m}{4}\rfloor$, and also that $\Et{2/E,\leq m_2}{C_4/F, \mathcal{A}}$ contains the unramified quadratic extension if and only if $v_F(\alpha) \equiv 0\pmod{4}$ for all $\alpha \in \mathcal{A}$. The result for $m \leq 4e_F + 1$ follows. 
    
    The argument for $m \geq 4e_F + 2$ is similar but easier, so we omit it. 
\end{proof}
\begin{corollary}
    \label{cor-size-of-Et-2^2-m-C4/F}
    For each nonnegative integer $m$, we have 
    $$
    \#\Et{(2^2)/F, m}{C_4/F,\mathcal{A}} = \begin{cases}
        \frac{1}{2}\cdot\big(
        \# N_{E, 2e_F - \frac{m}{2}}^{\mathcal{A}} - \# N_{E, 2e_F - \frac{m}{2} + 2}^{\mathcal{A}}    
        \big)\quad\text{if $4\mid m$ and $4 \leq m \leq 4e_F$,}
        \\
        \frac{1}{2}\cdot \big(
            \# N_{E}^{\mathcal{A}} - \# N_{E, 0}^{\mathcal{A}} 
        \big)\quad\text{if $m = 4e_F + 2$},
        \\
        0 \quad\text{otherwise}. 
    \end{cases}
    $$
\end{corollary}
\begin{proof}
    This follows easily from Lemma~\ref{lem-size-of-Et-2/E-C4/F-A-for-half-ram}. 
\end{proof}
\begin{proof}
    [Proof of Theorem~\ref{thm-2^2-2-adic}]

    We address the statements one by one. 
    \begin{enumerate}
        \item An $S_4$- or $A_4$-quartic extension has no proper intermediate fields, so it cannot have splitting symbol $(2^2)$.
        \item This is precisely Lemma~\ref{lem-2^2-D4-A-2-adic}.
        \item This follows from Corollary~\ref{cor-num-wild-Cp-exts-with-A-and-disc-val} and Lemma~\ref{lem-num-V4-2^2-with-m-and-A}.
        \item This follows from Corollary~\ref{cor-size-of-Et-2^2-m-C4/F}.
    \end{enumerate}
\end{proof}
In the special case where $\mathcal{A}$ is generated by a single element, we can write down a simple description of the sizes $\#N_{E,c}^{\mathcal{A}}$, and hence of the counts $\#\Et{(2^2)/F,m}{C_4/F,\mathcal{A}}$. We note these descriptions in Lemma~\ref{lem-size-of-N-c-A} and Corollary~\ref{cor-size-of-2^2-with-one-norm-elt}.
\begin{lemma}
    \label{lem-size-of-N-c-A}
    Let $\alpha \in F^\times\setminus F^{\times 2}$, let $d_\alpha = v_F(d_{F(\sqrt{\alpha})/F})$, and let $E/F$ be the unique unramified quadratic extension of $F$. For each nonnegative integer $c$, we have 
    $$
    \# N_{E,c}^{\langle \alpha \rangle} = \begin{cases}
        q^{e_F - \lceil\frac{c-1}{2}\rceil}\quad\text{if $c < d_\alpha$},
        \\
        2q^{e_F - \lceil\frac{c-1}{2}\rceil}\quad\text{if $d_\alpha \leq c \leq 2e_F$ and $v_F(\alpha) \equiv 0 \pmod{4}$},
        \\
        0 \quad\text{otherwise}. 
    \end{cases}
    $$
\end{lemma}
\begin{proof}
    From the definition of $N_E^\mathcal{A}$, it is easy to see that 
    $$
    N_E^{\langle \alpha\rangle} = \begin{cases}
        \overline{N}_\alpha^2 \quad\text{if $v_F(\alpha) \equiv 0 \pmod{4}$},
        \\
        (F^\times / F^{\times 2}) \setminus \overline{N}_\alpha^2 \quad\text{otherwise}. 
    \end{cases}
    $$

    By Lemma~\ref{lem-cond-disc-for-Cp}, we have $U_F^{(c)}F^{\times 2}/F^{\times 2} \subseteq \overline{N}^2_\alpha$ if and only if $c \geq d_\alpha$. By elementary linear algebra, it follows that 
    $$
    \# N_{E,c}^{\langle \alpha \rangle} = \begin{cases}
        \frac{1}{2}\cdot \#\Big(U_F^{(c)}F^{\times 2} / F^{\times 2} \Big)\quad\text{if $c < d_\alpha$},
        \\
        \#\Big(U_F^{(c)}F^{\times 2} / F^{\times 2} \Big) \quad\text{if $c \geq d_\alpha$ and $v_F(\alpha) \equiv 0 \pmod{4}$},
        \\
        0 \quad\text{otherwise}. 
    \end{cases}
    $$
    The result follows by Corollary~\ref{cor-dim-of-F-mod-Ut-and-pth-powers}.
\end{proof}
\begin{corollary}
    \label{cor-size-of-2^2-with-one-norm-elt}
    Let $\alpha \in F^\times \setminus F^{\times 2}$, and let $m$ be an integer. If $\Et{(2^2)/F, m}{C_4/F,\langle \alpha\rangle}$ is nonempty, then $v_F(\alpha)$ is even and either $m = 4e_ F+ 2$ or $m$ is a multiple of $4$ with $4\leq m \leq 4e_F$. If $v_F(\alpha)\equiv 0 \pmod{4}$ and $m$ is a multiple of $4$ with $4 \leq m \leq 4e_F$, then 
    $$
    \# \Et{(2^2)/F, m}{C_4/F,\langle \alpha\rangle} = \begin{cases}
        \frac{1}{2}q^{\frac{m}{4} - 1} (q-1) \quad\text{if $m > 4e_F - 2d_\alpha + 4$},
        \\
        \frac{1}{2}q^{\frac{m}{4} - 1}(q - 2)\quad\text{if $m = 4e_F - 2d_\alpha + 4$},
        \\
        q^{\frac{m}{4}-1}(q-1)\quad\text{if $m < 4e_F - 2d_\alpha + 4$}.
    \end{cases}
    $$
    If $v_F(\alpha) \equiv 2\pmod{4}$ and $m$ is a multiple of $4$ with $4 \leq m \leq 4e_F$, then 
    $$
    \# \Et{(2^2)/F, m}{C_4/F,\langle \alpha\rangle} = \begin{cases}
        \frac{1}{2}q^{\frac{m}{4} - 1} (q-1) \quad\text{if $m > 4e_F - 2d_\alpha + 4$},
        \\
        \frac{1}{2}q^{\frac{m}{4}}\quad\text{if $m = 4e_F - 2d_\alpha + 4$},
        \\
        0\quad\text{if $m < 4e_F - 2d_\alpha + 4$}.
    \end{cases}
    $$

    If $v_F(\alpha)$ is even and $m = 4e_F + 2$, then
    $$
    \#\Et{(2^2)/F,4e_F + 2}{C_4/F,\langle\alpha\rangle} = \begin{cases}
        \frac{1}{2}q^{e_F}\quad\text{if $d_\alpha > 0$},
        \\
        q^{e_F}\quad\text{if $v_F(\alpha)\equiv 2\pmod{4}$ and $d_\alpha = 0$},
        \\
        0 \quad\text{otherwise.}
    \end{cases}
    $$
\end{corollary}
\begin{proof}
    By \cite[Lemma~4.6]{monnet2024counting}, if $v_F(\alpha)$ is even then $d_\alpha$ is even. The result follows from Corollary~\ref{cor-size-of-Et-2^2-m-C4/F} and Lemma~\ref{lem-size-of-N-c-A}.
\end{proof}

\begin{lemma}
    \label{lem-1^4-G-for-G-S4-or-A4}
    For $G \in \{S_4,A_4\}$, we have 
    $$
    \Et{(1^4)/F}{G/F,\mathcal{A}} = \Et{(1^4)/F}{G/F}.
    $$
\end{lemma}
\begin{proof}
    Let $L \in \Et{(1^4)/F}{G/F}$, for $G\in\{S_4,A_4\}$. Then $L/F$ has no intermediate fields, so class field theory tells us that $N_{L/F}L^\times = F^\times$. 
\end{proof}
\begin{lemma}
    \label{lem-2-to-1-surjection-to-1^4-D4-exts}
    There is a natural $2$-to-$1$ surjection 
    $$
    \bigsqcup_{\substack{2m_1 + m_2 = m \\ m_1, m_2 > 0}} \bigsqcup_{E \in \Et{(1^2)/F, m_1}{\mathcal{A}}} \big\{L \in \Et{(1^2)/E, m_2}{} : L/F \textnormal{ not Galois}\big\} \to \Et{(1^4)/F, m}{D_4/F, \mathcal{A}}. 
    $$
\end{lemma}
\begin{proof}
    An element of the left-hand side may be written as a pair $(E,L)$, where $E \in \Et{(1^2)/F,m_1}{\mathcal{A}}$ and $L \in \Et{(1^2)/E,m_2}{}$. Using the tower law for discriminant, it is easy to see that there is a well-defined map 
    $$
    \Phi: \bigsqcup_{\substack{2m_1 + m_2 = m \\ m_1, m_2 > 0}} \bigsqcup_{E \in \Et{(1^2)/F, m_1}{\mathcal{A}}} \big\{L \in \Et{(1^2)/E, m_2}{} : L/F \text{ not Galois}\big\} \to \Et{(1^4)/F, m}{D_4/F, \mathcal{A}}, \quad (E,L)\mapsto L. 
    $$
    Let $L \in \Et{(1^4)/F, m}{D_4/F,\mathcal{A}}$. Since $L/F$ has Galois closure group $D_4$, it has a unique quadratic intermediate field $E$. It is easy to see that the pair $(E,L)$ is in the domain of $\Phi$, so  $\Phi$ is surjective. Let $L \in \Et{(1^4)/F,m}{D_4/F,\mathcal{A}}$. It is easy to see that 
    $$
    \Phi^{-1}(L) = \{(E,E(\sqrt{\alpha})), (E,E(\sqrt{\bar{\alpha}}))\},
    $$
    where $E/F$ is the unique quadratic subextension of $L/F$, and $\alpha \in E^\times$ is an element with $L = E(\sqrt{\alpha})$. Since $L/F$ is non-Galois, we have $E(\sqrt{\alpha}) \not \cong E(\sqrt{\bar{\alpha}})$ as extensions of $E$, and therefore the preimage $\Phi^{-1}(L)$ has exactly two elements, so we are done. 
\end{proof}
\begin{corollary}
    \label{cor-1^4-D4-with-A-in-terms-of-quadratics}
    We have 
    $$
    \#\Et{(1^4)/F,m}{D_4/F,\mathcal{A}} = \frac{1}{2}\sum_{\substack{2m_1 + m_2 = m \\ m_1, m_2 > 0}}\sum_{E \in \Et{(1^2)/F, m_1}{\mathcal{A}}}\Big(
    \#\Et{(1^2)/E,m_2}{} - \#\Et{(1^2)/E,m_2}{C_4/F} - \#\Et{(1^2)/E,m_2}{V_4/F}     
    \Big).
    $$
\end{corollary}
\begin{proof}
    This is immediate from Lemma~\ref{lem-2-to-1-surjection-to-1^4-D4-exts}.
\end{proof}

\begin{lemma}
    \label{lem-1^2/E-m2-V4}
    Let $E \in \Et{(1^2)/F,m_1}{}$ for some integer $m_1$. For each nonnegative integer $m_2$, we have 
    $$
    \# \Et{(1^2)/E, m_2}{V_4/F} = N^{V_4}(m_1, m_2),
    $$
    where $N^{V_4}$ is the function defined in Appendix~\ref{appendix-helpers}. 
\end{lemma}
\begin{proof}
    By the tower law for discriminant, every $L \in \Et{(1^2)/E,m_2}{V_4/F}$ has $v_F(d_{L/F}) = 2m_1 + m_2$. We will use this fact without reference throughout the proof. 
    Suppose that $m_2 < m_1$. By Lemma~\ref{lem-disc-of-V4-subexts}, we have a bijection
    $$
    \Et{(1^2)/F,m_2}{} \to \Et{(1^2)/E,m_2}{V_4/F},\quad E'\mapsto EE'.
    $$
    If $m_2 > m_1$, then similarly we obtain a $2$-to-$1$ surjection
    $$
    \Et{(1^2)/F,\frac{m_1+m_2}{2}}{} \to \Et{(1^2)/E,m_2}{V_4/F},\quad E\mapsto EE'.
    $$
    Using these two maps, the result for $m_1 \neq m_2$ follows from Theorem~\ref{thm-num-Cp-exts-with-disc-val}. 
    Finally, suppose that $m_1 = m_2$. Suppose that $L \in \Et{(1^2)/E, m_2}{V_4/F}$ and $L = EE'$ for some quadratic extension $E'/F$. Let $\chi,\chi' : F^\times / F^{\times 2} \to \F_2$ be the quadratic characters associated to $E$ and $E'$ respectively. By Theorem~\ref{thm-conductor-discriminant} and Lemma~\ref{lem-cond-disc-for-Cp}, we have 
    $$
    \mff(\chi) = v_F(d_{E/F}) = \min\big\{
        c : U_F^{(c)}F^{\times 2}/F^{\times 2}\subseteq \ker \chi    
    \big\},
    $$
    and similarly for $\chi'$ and $E'$. By Lemma~\ref{lem-disc-of-V4-subexts}, we have
    $$
    m_1 + \mff(\chi') + \mff(\chi\chi') = v_F(d_{L/F}) = 2m_1 + m_2 = 3m_1, 
    $$
    so 
    $$
    \mff(\chi') + \mff(\chi\chi') = 2m_1. 
    $$
    If $\mff(\chi') \neq m_1$, then Lemma~\ref{lem-disc-of-V4-subexts} tells us that
    $$
    \mff(\chi\chi') = \max\{m_1, \mff(\chi')\},
    $$
    so 
    $$
    \mff(\chi') + \max\{m_1, \mff(\chi')\} = 2m_1, 
    $$
    which is impossible. It follows that $\mff(\chi') = \mff(\chi\chi') = m_1$, so 
    $$
    \# \Et{(1^2)/E,m_1}{V_4/F} = \frac{1}{2}\cdot\#\Big\{\chi' : F^\times / U_F^{(m_1)}F^{\times 2} \to \F_2 \text{ }\Big| \text{ }\chi'|_{W_{m_1-1}} \not \in \{0, \chi|_{W_{m_1-1}}\}\Big\}.
    $$
    If $m_1 = 2e_F + 1$, then $W_{m_1-1} \cong C_2$ by Corollary~\ref{cor-size-of-Wi}, so $\Et{(1^2)/E,m_1}{V_4/F} = \varnothing$. If $m_1$ is even with $2 \leq m_1 \leq 2e_F$, then Corollary~\ref{cor-size-of-Wi} tells us that $\#W_{m_1-1} = q$, so there are $q-2$ possible restrictions $\chi'|_{W_{m_1-1}}$. Each such restriction lifts to 
    $$
    \#(F^\times / U_F^{(m_1-2)}F^{\times 2}) = 2q^{\frac{m_1}{2} - 1}
    $$
    characters, so we have 
    $$
    \#\Big\{\chi' : F^\times / U_F^{(m_1)}F^{\times 2} \to \F_2 \text{ }\Big| \text{ }\chi'|_{W_{m_1-1}} \not \in \{0, \chi|_{W_{m_1-1}}\}\Big\} = 2q^{\frac{m_1}{2}-1}(q-2),
    $$
    and the result for $m_1 = m_2$ follows. 
\end{proof}

\begin{lemma}
    We have 
    $$
    \#\Et{(1^4)/F,m}{D_4/F,\mathcal{A}} = \frac{1}{2}\cdot \sum_{0 < m_1 < m/2} \#\Et{(1^2)/F,m_1}{\mathcal{A}} \cdot\Big( N^{C_2}(m-2m_1) - N^{C_4}(m_1,m-2m_1) - N^{V_4}(m_1,m-2m_1)\Big),
    $$
    where the functions $N^{C_2}, N^{C_4}$, and $N^{V_4}$ are as defined in Appendix~\ref{appendix-helpers}. We can compute this quantity explicitly using Corollary~\ref{cor-num-wild-Cp-exts-with-A-and-disc-val}. 
\end{lemma}
\begin{proof}
    This follows easily from Theorem~\ref{thm-num-Cp-exts-with-disc-val}, Corollary~\ref{cor-1^4-D4-with-A-in-terms-of-quadratics}, Lemma~\ref{lem-1^2/E-m2-V4}, and \cite[Lemma~4.4]{monnet2024counting}. 
\end{proof}

\begin{lemma}
    \label{lem-1^4-V4}
    Let $m$ be an integer. If $\Et{(1^4)/F, m}{V_4/F,\mathcal{A}}$ is nonempty, then $m$ is an even integer with $6 \leq m \leq 6e_F + 2$. In that case, the number $\#\Et{(1^4)/F, m}{V_4/F, \mathcal{A}}$ is the sum of the following two quantities:
    \begin{enumerate}
        \item 
    $$
    \frac{1}{2}\cdot \sum_{\substack{m_1 < m_2 \\ m_1 + 2m_2 = m}} \Big(\#\Et{(1^2)/F,m_1}{\mathcal{A}}\Big)\Big(\#\Et{(1^2)/F,m_2}{\mathcal{A}}\Big).
    $$
    \item $$ \mathbbm{1}_{3\mid m}\cdot 
    \frac{2}{3(\#\overline{\mathcal{A}}^2)^2}\cdot q^{\frac{m}{3} - 2}\Big(q\# \overline{\mathcal{A}}^2_{m/3} - \#\overline{\mathcal{A}}^2_{m/3-1}\Big)\Big(q\# \overline{\mathcal{A}}^2_{m/3} - 2\#\overline{\mathcal{A}}^2_{m/3-1}\Big).
    $$
    \end{enumerate}
\end{lemma}
\begin{proof}
    The necessary conditions on $m$, namely that $m$ is even with $6 \leq m \leq 6e_F + 2$, come from \cite[Lemma~3.2]{monnet2024counting}. Assume that $m$ satisfies these conditions. Define the map 
    $$
    \Phi: \bigsqcup_{\substack{m_1 < m_2 \\ m_1 + 2m_2 = m}} \Et{(1^2)/F,m_1}{\mathcal{A}} \times \Et{(1^2)/F,m_2}{\mathcal{A}} \to \Et{(1^4)/F,m}{V_4/F,\mathcal{A}},\quad (E_1,E_2)\mapsto E_1E_2. 
    $$
    By Lemma~\ref{lem-disc-of-V4-subexts}, this map is well-defined and $2$-to-$1$, so 
    $$
    \#\im \Phi = \frac{1}{2}\cdot \sum_{\substack{m_1 < m_2 \\ m_1 + 2m_2 = m}} \Big(\#\Et{(1^2)/F,m_1}{\mathcal{A}}\Big)\Big(\#\Et{(1^2)/F,m_2}{\mathcal{A}}\Big).
    $$
    If $3 \nmid m$, then $\Phi$ is surjective, so we are done. Suppose that $3 \mid m$. Let $S$ be the set of $L \in \Et{(1^4)/F,m}{V_4/F,\mathcal{A}}$ such that every intermediate quadratic field $E$ of $L$ has $v_F(d_{E/F}) = m/3$. Then 
    $$
    \Et{(1^4)/F,m}{V_4/F,\mathcal{A}} = S \sqcup \im\Phi, 
    $$
    so 
    $$
    \# \Et{(1^4)/F,m}{V_4/F,\mathcal{A}} = \# S + \frac{1}{2}\cdot \sum_{\substack{m_1 < m_2 \\ m_1 + 2m_2 = m}} \Big(\#\Et{(1^2)/F,m_1}{\mathcal{A}}\Big)\Big(\#\Et{(1^2)/F,m_2}{\mathcal{A}}\Big).
    $$
    Let $\Sigma$ be the set of pairs $(\chi_1,\chi_2)$, where: 
    \begin{enumerate}
        \item $\chi_1$ and $\chi_2$ are quadratic characters $\chi_i:F^\times / U_F^{(m/3)}F^{\times 2} \to \F_2$.
        \item The restrictions of $\chi_1$ and $\chi_2$ to $U_F^{(m/3 - 1)}F^{\times 2} / U_F^{(m/3)}F^{\times 2}$ are nonzero and distinct. 
        \item For $i=1,2$, we have 
        $$
        \chi_i\Big(\mathcal{A}U_F^{(m/3)}F^{\times 2} / F^{\times 2} U_F^{(m/3)}\Big) = 0. 
        $$
    \end{enumerate}
    Then there is a natural $6$-to-$1$ surjection 
    $
    \Sigma \to S,
    $
    so 
    $$
    \# S = \frac{1}{6}\cdot \#\Sigma.
    $$
    Evaluating $\#S$ amounts to a simple linear algebra problem. To emphasise this simplicity, we define:
    \begin{enumerate}
        \item $$V = F^\times / U_F^{(m/3)}F^{\times 2}.$$
        \item $$M = U_F^{(m/3 - 1)}F^{\times 2} / U_F^{(m/3)}F^{\times 2}.$$
        \item $$N = \mathcal{A}U_F^{(m/3)}F^{\times 2} / U_F^{(m/3)}F^{\times 2}.$$
    \end{enumerate}
    Then $V$ is an $\F_2$-vector space, $M$ and $N$ are subspaces of $V$, and we are looking for pairs of linear transformations $\chi_1,\chi_2 : V \to \F_2$ such that the following two statements are true:
    \begin{enumerate}
        \item $\chi_i(N) = 0$ for $i=1,2$.
        \item The restrictions $\chi_1|_M$ and $\chi_2|_M$ are nonzero and distinct. 
    \end{enumerate}
    These correspond bijectively to pairs $\overline{\chi}_1,\overline{\chi}_2: V/N \to \F_2$ such that the restrictions $\overline{\chi}_1|_{(M+N)/N}$ and $\overline{\chi}_2|_{(M+N)/N}$ are nonzero and distinct. There are 
    $$
    \Big(\#\Big(\frac{M+N}{N}\Big) - 1\Big)\Big(\#\Big(\frac{M+N}{N}\Big) - 2\Big)
    $$
    possibilities for the pair $(\overline{\chi}_1|_{(M+N)/N},\overline{\chi}_2|_{(M+N)/N})$. Each of these lifts to 
    $$
    \Big(\#\Big(\frac{V}{M+N}\Big) \Big)^2 
    $$
    pairs $(\overline{\chi}_1,\overline{\chi}_2)$, so we have 
    $$
    \# \Sigma = \Big(\#\Big(\frac{V}{M+N}\Big) \Big)^2 \cdot \Big(\#\Big(\frac{M+N}{N}\Big) - 1\Big)\Big(\#\Big(\frac{M+N}{N}\Big) - 2\Big).
    $$
    We evaluate the sizes of the relevant vector spaces.
    \begin{enumerate}
        \item Corollary~\ref{cor-dim-of-F-mod-Ut-and-pth-powers} tells us that $$\# V = 2q^{\frac{m}{6}}.$$
        \item By the second and third isomorphism theorems for groups, we have 
        $$
        N = \frac{\mathcal{A} U_F^{(m/3)}F^{\times 2}}{U_F^{(m/3)}F^{\times 2}} \cong \frac{\overline{\mathcal{A}}^2}{\overline{\mathcal{A}}^2_{m/3}},
        $$
        so 
        $$
        \# N = \frac{\#\overline{\mathcal{A}}^2}{\#\overline{\mathcal{A}}^2_{m/3}}. 
        $$
        \item We have 
        $$
        M + N = \frac{\mathcal{A}U_F^{(m/3-1)}F^{\times 2}}{U_F^{(m/3)}F^{\times 2}},
        $$
        so 
        $$
        \# (M + N) = \Big[ 
        \mathcal{A} U_F^{(m/3 - 1)}F^{\times 2} : U_F^{(m/3 - 1)} F^{\times 2}    
        \Big] \cdot \# W_{m/3 - 1}. 
        $$
        By the second and third isomorphism theorems for groups, we have
        $$
        \frac{\mathcal{A}U_F^{(m/3-1)}F^{\times 2}}{U_F^{(m/3 - 1)}F^{\times 2}} \cong \frac{\overline{\mathcal{A}}^2}{\overline{\mathcal{A}}^2_{m/3 - 1}},
        $$
        so Corollary~\ref{cor-size-of-Wi} tells us that 
        $$
        \# (M+N) = q \cdot \frac{\#\overline{\mathcal{A}}^2}{\#\overline{\mathcal{A}}_{m/3 - 1}^2}.
        $$
    \end{enumerate}
    It follows that 
    $$
    \# \Sigma = \frac{4q^{\frac{m}{3} - 2}}{(\#\overline{\mathcal{A}}^2)^2}\cdot \Big(q\# \overline{\mathcal{A}}^2_{m/3} - \#\overline{\mathcal{A}}^2_{m/3-1}\Big)\Big(q\# \overline{\mathcal{A}}^2_{m/3} - 2\#\overline{\mathcal{A}}^2_{m/3-1}\Big).
    $$
\end{proof}

\subsection{Totally wildly ramified $C_4$-extensions}
Our goal in this subsection is to prove Theorem~\ref{thm-time-complexity-mass-all-quartics-with-norms}. So far, we have addressed every case except $\Et{(1^4)/F}{C_4/F,\mathcal{A}}$ for $2$-adic fields $F$. In this subsection, we give two methods for computing the pre-mass of this set, and analyse their time complexities.

\begin{lemma}
    \label{lem-UE2t-mod-squares-intersect-with-F}
    Let $F$ be a $2$-adic field, let $E \in \Et{(1^2)/F, m_1}{}$ for a positive integer $m_1$. Let $t$ be an integer with $0 \leq t \leq 2e_F -\frac{m_1}{2}$. We have 
    $$
    (E^{\times 2}U_E^{(2t)}\cap F^\times)/F^{\times 2} = \begin{cases}
        F^\times / F^{\times 2}\quad\text{if $t \leq \frac{m_1}{2} - 1$},
        \\
        U_F^{\Big(2\Big\lceil \frac{t - \frac{m_1}{2}}{2}\Big\rceil\Big)}F^{\times 2}/F^{\times 2} \quad\text{if $t \geq \frac{m_1}{2}$}. 
    \end{cases}
    $$
\end{lemma}
\begin{proof}
    The case $t=0$ is obvious. For $t \geq 1$, the result is essentially \cite[Proposition~3.11]{cohen-et-al}, combined with the fact that 
    $$
    U_F^{(2c)}F^{\times 2} = U_F^{(2c+1)}F^{\times 2}
    $$
    for all nonnegative integers $c$ with $0 \leq c < e_F$, by Corollary~\ref{cor-size-of-Wi}.
\end{proof}
\begin{lemma}
    \label{lem-equivalent-floors-with-m1-and-m2}
    Let $m_1$ and $m_2$ be integers, with $m_1$ even. Then 
    $$
    \Big\lfloor \frac{m_1}{4} + \frac{1}{2} \Big\lfloor \frac{m_2}{2} \Big\rfloor \Big\rfloor = \Big\lfloor \frac{m_1+m_2}{4}\Big\rfloor. 
    $$
\end{lemma}
\begin{proof}
    This is easy to see by writing $m_1 = 2k_1$ and $m_2 = 2k_2 + r$, for $r \in \{0, 1\}$. 
\end{proof}
\begin{lemma}
    \label{lem-conditions-to-be-all-or-nothing-or-other-thing}
    Let $F$ be a $2$-adic field and let $E \in \Et{(1^2)/F}{\mathcal{A}}$. If $-1 \not \in N_{E/F}E^\times$, then 
    $$
    \Et{(1^2)/E}{C_4/F} = \varnothing. 
    $$ 
    For the rest of the lemma, assume that $-1 \in N_{E/F}E^\times$. Let $m_1 = v_F(d_{E/F})$, and let $m_2$ be an integer. Let 
    $$
    c(m_2) = 2e_F - 2\Big\lfloor
    \frac{m_1+m_2}{4}
\Big\rfloor.
    $$
    If $m_1 \leq e_F$, then
    $$
    \#\Et{(1^2)/E,\leq m_2}{C_4/F,\mathcal{A}} = 
    \begin{cases}
        0 \quad\text{if $m_2 < 3m_1 - 2$},
        \\
        \frac{1}{2}\cdot \# N^\mathcal{A}_{E, c(m_2)} \quad\text{if $3m_1 - 2 \leq m_2 \leq 4e_F - m_1 + 1$},
        \\
        \frac{1}{2}\cdot\# N^\mathcal{A}_E \quad\text{if $m_2 \geq 4e_F - m_1 + 2$}. 
    \end{cases}
    $$
    If $m_1 > e_F$, then  
    $$
    \#\Et{(1^2)/E,\leq m_2}{C_4/F,\mathcal{A}} = \begin{cases}
        0\quad\text{if $m_2 < m_1 + 2e_F$},
        \\
        \frac{1}{2}\cdot\# N^\mathcal{A}_E \quad\text{if $m_2 \geq m_1 +2e_F$}.
    \end{cases}
    $$
\end{lemma}
\begin{proof}
    If $-1 \not \in N_{E/F}E^\times$, then $\Et{(1^2)/E}{C_4/F} = \varnothing$ by \cite[Corollary~4.9]{monnet2024counting}. Suppose that $-1 \in N_{E/F}E^\times$. Then \cite[Corollary~4.9]{monnet2024counting} and \cite[Lemma~4.21]{monnet2024counting} tell us that $\Et{(1^2)/E,\leq m_2}{C_4/F}$ is nonempty if and only if 
    $$
    m_2 \geq \begin{cases}
        3m_1 - 2 \quad\text{if $m_1 \leq e_F$,}
        \\
        2e_F + m_1\quad\text{if $m_1 > e_F$}. 
    \end{cases}
    $$
    Suppose that $m_1 \leq e_F$ and $m_2 \geq 3m_1 - 2$. If $m_2 \leq 4e_F - m_1 + 1$, then Lemmas~\ref{lem-C4-bijection-with-norm-condition}, \ref{lem-UE2t-mod-squares-intersect-with-F} and \ref{lem-equivalent-floors-with-m1-and-m2} give a $2$-to-$1$ surjection
    $$
    N_{E,c(m_2)}^\mathcal{A} \to \Et{2/E, \leq m_2}{C_4/F,\mathcal{A}}.
    $$
    If $m_2 \geq 4e_F - m_1 + 2$, then \cite[Lemma~4.4]{monnet2024counting} tells us that 
    $$
    \Et{2/E,\leq m_2}{C_4/F} = \Et{2/E, \leq 4e_F - m_1 + 2}{C_4/F},
    $$
    and we have 
    $$
    \#\Et{2/E, \leq 4e_F - m_1 + 2}{C_4/F} = \frac{1}{2}\cdot \# N_E^\mathcal{A}
    $$ 
    by Lemmas~\ref{lem-C4-bijection-with-norm-condition} and \ref{lem-UE2t-mod-squares-intersect-with-F}. 
    The result for $m_1 \leq e_F$ follows. The argument for $m_1 > e_F$ is similar but easier, so we omit it. 
\end{proof} 

\begin{proof}
    [Proof of Theorem~\ref{thm-1^4-2-adic}]
    We prove the statements one at a time:
    \begin{enumerate}
        \item This is immediate from Lemma~\ref{lem-1^4-G-for-G-S4-or-A4}.
        \item This follows easily from Theorem~\ref{thm-num-Cp-exts-with-disc-val}, Corollary~\ref{cor-1^4-D4-with-A-in-terms-of-quadratics}, Lemma~\ref{lem-1^2/E-m2-V4}, and \cite[Lemma~4.4]{monnet2024counting}. 
        \item This is precisely Lemma~\ref{lem-1^4-V4}.
        \item This is immediate from Lemma~\ref{lem-conditions-to-be-all-or-nothing-or-other-thing}.
    \end{enumerate}
\end{proof}

\begin{algorithm}
    \label{algo-omega-with-small-disc-val}
    Let $F$ be a $2$-adic field, let $m_1$ be an even integer with $m_1 \leq e_F$, and let $E \in \Et{(1^2)/F,m_1}{}$ be $C_4$-extendable. 
    \begin{enumerate}
        \item Take $d \in F^\times$ such that $E = F(\sqrt{d})$ and $v_F(d) = m_1$. 
        \item Take $a,b \in F^\times$ with $d = a^2 + 4b$, such that $v_F(a) = m_1/2$ and $v_F(b) = 1$.
        \item Define $\rho = \frac{a + \sqrt{d}}{2}$, which is a uniformiser of $E$.
        \item Choose $\omega \in E^\times$ such that $N_{E/F}\omega \in dF^{\times 2}$ and $v_E(\omega) = 0$.
        \item Take $\lambda \in F^\times$ such that 
        $$
        N_{E/F}\omega \equiv \lambda^2 \pmod{\p_F^{2e_F + 1 - m_1}}.
        $$
        \item Define 
        $$
        \omega_1 = \begin{cases}
            \omega \quad\text{if $v_E(\omega - \lambda) \geq m_1$},
            \\
            \frac{\omega b}{\rho^2}\quad\text{if $v_E(\omega - \lambda) = m_1 - 1$}. 
        \end{cases}
        $$
        \item Define 
        $$
        \omega_2 = \begin{cases}
            \omega_1 \quad\text{if $v_E(\omega_1 - \bar{\omega}_1) = 2m_1$},
            \\
            \omega_1(1 + \rho)^2 \quad\text{if $v_E(\omega_1 - \bar{\omega}_1) > 2m_1$},
        \end{cases}
        $$
        and 
        $$
        \lambda_2 = \begin{cases}
            \lambda \quad\text{if $v_E(\omega_1 - \bar{\omega}_1) = 2m_1$},
            \\
            \lambda(1 + a - b) \quad\text{if $v_E(\omega_1 - \bar{\omega}_1) > 2m_1$}.
        \end{cases}
        $$
        \item Write $\omega_2 = r_2 + s_2\rho$ for $r_2, s_2 \in F$, and define 
        $$
        q = \frac{r_2-\lambda_2}{s_2},\quad n = \frac{q^2 + b}{r_2}.
        $$
        \item Output 
        $
        \frac{\omega_2 n}{(q+\rho)^2}.
        $
    \end{enumerate}
\end{algorithm}
\begin{theorem}
    \label{thm-algo-for-small-disc-C4-works}
    Let $E, m_1$ be as in Algorithm~\ref{algo-omega-with-small-disc-val}. 
    \begin{enumerate} 
        \item All the steps of the algorithm are well-defined. 
        \item Let $\omega$ be the output of the algorithm. Then $E(\sqrt{\omega}) \in \Et{(1^2)/E,3m_1-2}{C_4/F}$. 
    \end{enumerate}
\end{theorem}
\begin{proof}
    This is essentially \cite[Proposition~3.15]{cohen-et-al}. We rewrite their proof using our notation in Appendix~\ref{appendix-algo}. 
\end{proof}
\begin{definition}
    Let $K$ be a field and let $M$ be a matrix with entries in $K$. A \emph{reduced row decomposition of} $M$ is a triple $(\widetilde{M}, T, T^{-1})$, where $\widetilde{M}$ is a matrix in reduced row echelon form and $T$ is a composition of elementary matrices with $M = T\widetilde{M}$. For computational efficiency, we consider the inverse matrix $T^{-1}$ to be part of the data of the reduced row decomposition. 
\end{definition}
\begin{lemma}
    \label{lem-time-complexity-reduced-row-echelon}
    Let $M$ be an $m\times n$ matrix defined over a field $K$. We can compute a reduced row decomposition of $M$ using $O(mn\min\{m,n\})$ field operations in $K$.
\end{lemma}
\begin{proof}
    For $i=1,\ldots, \min\{m,n\}$, the $i^\mathrm{th}$ step of Gaussian elimination (i.e. reducing the $i^\mathrm{th}$ column) requires $O(n)$ elementary row operations. Each elementary row operation requires $O(m)$ field operations in $K$, so we are done. 
\end{proof}
\begin{definition}
    \label{defi-B0E}
    Let $F$ be a $p$-adic field for some rational prime $p$. Let $E/F$ be a field extension with $[E : F]\leq 4$, let $m(X) \in \Q_p[X]$ be a monic degree $f_E$ polynomial that is irreducible over $\F_p[X]$, and let $\alpha \in E$ be a root of $m(X)$. Define the set 
    $$
    \mathcal{B}_0^E = \{1, \alpha, \ldots, \alpha^{f_E - 1}\}.
    $$
    We will assume that (as is the case with Magma's `FldPad' object class) the maximal unramified subextension of $E/\Q_p$ is defined as $\frac{\Q_p[X]}{(m(X))}$, so that $\mathcal{B}_0^E$ is part of the data of $E$.
\end{definition}
\begin{lemma}
    \label{lem-time-complexity-of-operations-in-residue-class-field}
    Let $F$ and $E$ be as in Definition~\ref{defi-B0E}. The following three statements are true:
    \begin{enumerate}
        \item The set $\mathcal{B}_0^E$ descends to an $\F_p$-basis for the vector space $\F_E$. 
        \item For any $x \in \co_E$, we can compute the $\mathcal{B}_0^E$-coefficients of $[x] \in \F_E$ with time complexity $O_F(1)$. 
        \item Suppose that we have an $\F_p$-linear transformation $\varphi: \F_E \to \F_E$ that we can evaluate with time complexity $O_F(t)$, for some function $t$. We can compute the matrix $[\varphi]_{\mathcal{B}_0^E}$ with time complexity $O_F(f_F\cdot t)$.
    \end{enumerate}
\end{lemma}
\begin{proof}
    Recall the polynomial $m(X) \in \Q_p[X]$ from Definition~\ref{defi-B0E}. The first statement follows from the fact that $\F_E$ is defined by $m(X)$ as an extension of $\F_p$. The second statement follows from the fact that (at least in Magma) $x$ is implemented as a power series in $\pi_E$ with coefficients in the maximal unramified subextension $E^{\mathrm{ur}}$ of $E$, and elements of $E^{\mathrm{ur}}$ are stored as $\Z_p$-linear combinations of $\mathcal{B}_0^E$. We can reduce all of these $\Z_p$-coefficients modulo $p$ with time complexity $O(f_F) \ll O_F(1)$. The third statement follows immediately from the second. 
\end{proof}
\begin{lemma}
    \label{lem-time-complexity-reduced-varphi-matrix}
    Let $F$ be a $p$-adic field for some rational prime $p$, and let $E/F$ be a field extension with $[E: F]\leq 4$. Recall the map 
    $$
    \varphi : \F_E \to \F_E,\quad [y] \mapsto \Big[y + \frac{\pi_E^{e_E}}{p}y^p\Big]
    $$
    from Algorithm~\ref{algo-c-alpha}. We can compute a reduced row decomposition of the matrix $[\varphi]_{\mathcal{B}_0^E}$ with time complexity $O_F(f_F^3 + f_F\log e_F + f_F\log p)$.
\end{lemma}
\begin{proof}
    A single evaluation of $\varphi$ has time complexity $O_F(\log e_F + \log p)$, so Lemma~\ref{lem-time-complexity-of-operations-in-residue-class-field} tells us that we can compute $[\varphi]_{\mathcal{B}_0^E}$ with time complexity $O_F(f_F\log e_F + f_F\log p)$. The result follows from Lemma~\ref{lem-time-complexity-reduced-row-echelon}.
\end{proof}
\begin{lemma}
    \label{lem-time-complexity-of-c-alpha-algo}
    Let $F$ be a $p$-adic field for some rational prime $p$ and let $\varphi:\F_F \to \F_F$ be the map from Algorithm~\ref{algo-c-alpha}. 
    If we have already computed a reduced row decomposition for $[\varphi]_{\mathcal{B}_0^F}$, then Algorithm~\ref{algo-c-alpha} can be run with time complexity $O_F([F:\Q_p])$.
\end{lemma}
\begin{proof}
    Steps $(1)$ and $(3)$ have time complexity $O_F(1)$. Consider the iteration in Step $(2)$. The steps with $p\nmid i$ have time complexity $O_F(1)$, and there are $O(e_F)$ such steps, so we can perform them all with time complexity $O_F(e_F)$. The steps with $p\mid i$ and $i < \frac{pe_F}{p-1}$ have time complexity $O_F(f_F\log p)$, since taking $p^\mathrm{th}$ roots in $\F_F$ is equivalent to raising to the power of $p^{f_F - 1}$. There are $O(e_F/p)$ such steps, so we can perform all of them with time complexity $O_F([F:\Q_p]\cdot \frac{\log p}{p})$. Thus, we can perform all the steps with $i < \frac{pe_F}{p-1}$ with time complexity $O_F([F:\Q_p])$.
    
    Assume that we have a primitive root modulo $p$ and a logarithm table for $\F_p$ with respect to this primitive root, so that we can perform any field operation in $\F_p$ with time complexity $O(1)$. This is a modest requirement, and we consider it to be part of any sensible implementation. Let $(\widetilde{M}, T, T^{-1})$ be our reduced row descomposition of $[\varphi]_{\mathcal{B}_0^F}$. Then the final step amounts to finding a vector $v \in \F_p^{f_F}$ with $\widetilde{M}v = T^{-1}[\frac{m_i - 1}{\pi_F^{e_F/(p-1)}p}]_{\mathcal{B}_0^E}$. This has time complexity $O(f_F^2) \ll O_F(1)$. 
\end{proof}
\begin{lemma}
    \label{lem-basis-for-units-mod-squares}
    Let $F$ be a $p$-adic field for some rational prime $p$, and let $E$ be an extension of $F$ with $[E:F]\leq 4$. Let $\mathcal{B}_{-1} = \{\pi_E\}$ and let $\mathcal{B}_0 = \mathcal{B}_0^E$. If $\mu_p\not\subseteq E$, then set $\mathcal{B}_{\frac{pe_E}{p-1}} = \varnothing$. Suppose instead that $\mu_p\subseteq E$. Let $u_{\frac{pe_E}{p-1}} \in \co_E^\times$ be an element such that $[u_{\frac{pe_E}{p-1}}] \in \F_E$ is not in the image of the map 
    $$
    \varphi: \F_E \to \F_E, \quad [y] \mapsto \Big[y + \frac{\pi_E^{e_E}}{p}y^p\Big],
    $$
    and define $\mathcal{B}_{\frac{pe_E}{p-1}} = \{1 + p\pi_E^{e_E/(p-1)}u_{\frac{pe_E}{p-1}}\}$. Define 
    $$
    \mathcal{B} = \mathcal{B}_{-1} \sqcup \mathcal{B}_{\frac{pe_E}{p-1}} \sqcup \bigsqcup_{\substack{1 \leq i \leq \lceil\frac{pe_E}{p-1}\rceil - 1 \\ p\nmid i}} \{1 + \pi_E^iu : u \in \mathcal{B}_0\}.
    $$
    The following two statements are true:
    \begin{enumerate}
        \item $\mathcal{B}$ is a system of representatives for a basis of the $\F_p$-vector space $E^\times / E^{\times p}$. 
        \item Assume that we have already computed a reduced row decomposition of $[\varphi]_{\mathcal{B}_0^E}$. For any element $\alpha \in E^\times$, we can compute the coefficients of $[\alpha] \in E^\times / E^{\times p}$ with respect to the basis induced by $\mathcal{B}$ with time complexity $O_F([F:\Q_p])$.
    \end{enumerate}
\end{lemma}
\begin{proof}
    Corollary~\ref{cor-dim-of-F-mod-Ut-and-pth-powers} tells us that the size of $\mathcal{B}$ equals the dimension of $E^\times / E^{\times p}$. Therefore, it suffices to prove that $\mathcal{B}$ spans $E^\times / E^{\times p}$. We will thus give an $O_F([F:\Q_p])$ algorithm for expressing $[\alpha]\in E^\times / E^{\times p}$ as a linear combination of $\mathcal{B}$, for any $\alpha \in E^\times$, thus proving both statements simultaneously. 

    Let $\alpha \in E^\times$. Without loss of generality, we may assume that $v_E(\alpha) \in \{0,1,\ldots, p-1\}$. Let $\alpha_0 = \frac{\alpha}{\pi_E^{v_E(\alpha)}}$. We will recursively define an element $\alpha_{i+1}$ for each $i = 0,1,2,\ldots, \lceil \frac{pe_E}{p-1}\rceil-1$. We claim that for each of these $i$, we have $\alpha_{i+1} \in U_E^{(i+1)}$. Clearly $\alpha_0 \in U_E^{(0)}$, so we have the base case for our induction. 
    \begin{itemize}
        \item Suppose that $p\mid i$. With time complexity $O_F(f_F\log p)$, we can find $[y_i] \in \F_E$ such that 
        $$
        y_i^p \equiv \frac{\alpha_i - 1}{\pi_E^i}\pmod{\p_E}.
        $$
        Then set $\alpha_{i+1} = \alpha_i / (1 + \pi_E^{i/p}y_i)^p$, so that $\alpha_{i+1} \in U_E^{(i+1)}$ by Lemma~\ref{lem-lift-to-pth-power-mod-n}. 
        \item Suppose that $p\nmid i$. Since $\big[\frac{\alpha_i-1}{\pi_E^i}\big] \in \F_E$, there are unique coefficients $\lambda_u^{(i)} \in \{0,1,\ldots, p-1\}$ such that 
        $$
        \Big[\frac{\alpha_i - 1}{\pi_E^i}\Big] = \sum_{u \in \mathcal{B}_0} \lambda_u^{(i)} [u], 
        $$
        and by Lemma~\ref{lem-time-complexity-of-operations-in-residue-class-field} we can determine the coefficients $\lambda_u^{(i)}$ with time complexity $O_F(1)$.
        By the natural isomorphism $\F_E \to U_E^{(i)}/U_E^{(i+1)}$, we have 
        $$
        \alpha_i \equiv \prod_{u\in\mathcal{B}_0}(1 + \pi_E^iu)^{\lambda_u^{(i)}}\pmod{\p_E^{i+1}},
        $$
        so we define $$\alpha_{i+1} = \frac{\alpha_i}{\prod_{u\in\mathcal{B}_0}(1 + \pi_E^iu)^{\lambda_u^{(i)}}} \in U_E^{(i+1)}.$$
    \end{itemize}
    Suppose that $\mu_p\not\subseteq E$. Then Corollary~\ref{cor-size-of-Wi} tells us that $\alpha_{\lceil \frac{pe_E}{p-1}\rceil} \in E^{\times p}$, and therefore
    $$
    \alpha E^{\times p} = \pi_E^{v_E(\alpha)} \cdot \prod_{\substack{1 \leq i \leq \lceil \frac{pe_E}{p-1}\rceil - 1 \\ p\nmid i}} (1 + \pi_E^iu)^{\lambda^{(i)}_u} E^{\times p},
    $$
    as required. 

    The coefficients $\lambda_u^{(i)}$ we have found so far were computed in $O(e_F)$ steps of time complexity $O_F(1)$, and $O(e_F/p)$ steps of time complexity $O_F(f_F\log p)$, so the algorithm so far has time complexity $O_F([F:\Q_p])$.

    Suppose instead that $\mu_p \subseteq E$, so $(p-1)\mid e_E$. Let $\lambda^{(\frac{pe_E}{p-1})}$ be the unique element of $\F_p$ with 
    $$
    \alpha_{\frac{pe_E}{p-1}} - \lambda^{(\frac{pe_E}{p-1})} u_{\frac{pe_E}{p-1}} \in \im \varphi. 
    $$
    Let $(\widetilde{M}, T, T^{-1})$ be our reduced row decomposition of $[\varphi]_{\mathcal{B}_0^E}$. Then $\lambda^{(\frac{pe_E}{p-1})}$ can be read off from the final entries of the vectors $T^{-1}[\alpha_{\frac{pe_E}{p-1}}]_{\mathcal{B}_0^E}$ and $T^{-1}[u_{\frac{pe_E}{p-1}}]_{\mathcal{B}_0^E}$, which can be computed with time complexity $O(f_F) \ll O_F(1)$. 
    
    By Corollary~\ref{cor-p-power-iff-p-power-mod-p-pi}, Lemma~\ref{lem-lift-to-pth-power-mod-n}, and Lemma~\ref{lem-im-of-coef-times-y^p+y-map}, we have 
    $$
    \alpha_{\frac{pe_E}{p-1}}/ (1 + p\pi_E^{e_E/(p-1)}u_{\frac{pe_E}{p-1}})^{\lambda^{(\frac{pe_E}{p-1})}} \in E^{\times p}.
    $$
    Thus, we have 
    $$
    \alpha E^{\times p} = \pi_E^{v_E(\alpha)} \cdot (1 + p\pi_E^{e_E/(p-1)}u_{\frac{pe_E}{p-1}})^{\lambda^{(\frac{pe_E}{p-1})}} \cdot \prod_{\substack{1 \leq i \leq \frac{pe_E}{p-1} - 1 \\ p\nmid i}}\prod_{u \in \mathcal{B}_0} (1 + \pi_E^iu)^{\lambda_u^{(i)}} E^{\times p},
    $$
    as required. 
\end{proof}

\begin{lemma}
    \label{lem-time-complexity-of-norm-equations}
    Let $F$ be a $2$-adic field and let $L/E/F$ be a tower of field extensions, where $L/E$ is quadratic and $[E:F]\leq 2$. Let $d \in E^\times$. We can do the following with time complexity $O_F([F:\Q_2]^3)$:
    \begin{enumerate}
        \item Determine whether $d \in N_{L/E}L^\times$. 
        \item If so, find an element $\omega \in L^\times$ such that $N_{L/E}\omega \in dE^{\times 2}$. 
    \end{enumerate}
\end{lemma}
\begin{proof}
    By Lemma~\ref{lem-time-complexity-reduced-varphi-matrix}, we can quickly compute a reduced row decomposition of $[\varphi]_{\mathcal{B}_0^E}$. Using Lemma~\ref{lem-basis-for-units-mod-squares}, we can quickly write down a set $\mathcal{B}\subseteq L^\times$, of size $2 + [L:\Q_2]$, that descends to a basis of $L^\times / L^{\times 2}$. Taking norms\footnote{Note that norms can be computed quickly since they are determinants of linear transformations in $2$ dimensions. }, we obtain a spanning set $\Nm\mathcal{B}$ for $N_{L/E}L^\times / E^{\times 2}$. 
    
    Again using Lemma~\ref{lem-basis-for-units-mod-squares}, fix a basis for $E^\times / E^{\times 2}$. Let $A \in \F_2^{(2 + [E:\Q_2]) \times (2 + [L:\Q_2])}$ be the matrix whose columns are the coordinates of the elements of $\Nm \mathcal{B}$, and let $v\in \F_2^{2+[E:\Q_2]}$ be the coordinate vector of $[d] \in E^\times / E^{\times 2}$. Note that, by Lemma~\ref{lem-basis-for-units-mod-squares}, $A$ and $v$ can be computed with time complexity $O_F([F:\Q_2]^2)$. It then suffices to perform Gaussian elimination on the augmented matrix $(\hspace{-0.4em}\begin{array}{c|c}
        A & v
     \end{array}\hspace{-0.4em})$. By Lemma~\ref{lem-time-complexity-reduced-row-echelon}, this can be done with time complexity $O_F([F:\Q_2]^3)$, so we are done.
\end{proof}

\begin{theorem}
    \label{thm-time-comlexity-of-finding-omega}
    Let $F$ be a $2$-adic field and let $\varphi : \F_F \to \F_F$ be the map from Algorithm~\ref{algo-c-alpha}. Assume that we have already computed a reduced row decomposition for $[\varphi]_{\mathcal{B}_0^F}$. Then Algorithm~\ref{algo-omega-with-small-disc-val} can be performed with time complexity $O_F([F:\Q_2]^3)$. 
\end{theorem}
\begin{proof}
    We compute the time complexity of each step of the algorithm, one by one. In Appendix~\ref{appendix-algo}, we give expanded descriptions of these steps. In our analysis here, we use these expanded descriptions without reference. 
    \begin{enumerate}
        \item Clearly this has time complexity $O_F(1)$.
        \item We need to solve the congurence $\frac{X^2}{d}- 1 \equiv 0\pmod{\p_F^{2e_F + 1 - m_1}}$. By Lemma~\ref{lem-time-complexity-of-c-alpha-algo}, we can do this with time complexity $O_F([F:\Q_2])$.
        \item This is clearly $O_F(1)$.
        \item By Lemma~\ref{lem-time-complexity-of-norm-equations}, this has time complexity $O_F([F:\Q_2]^3)$. 
        \item Using Algorithm~\ref{algo-c-alpha}, we can find $\lambda$ with time complexity $O_F([F:\Q_2])$, by Lemma~\ref{lem-time-complexity-of-c-alpha-algo}.
        \item[(6)-(9)] The remaining steps are all just computations in $F$, so they have time complexity $O_F(1)$. 
    \end{enumerate}
\end{proof}
\begin{lemma}
    \label{lem-time-complexity-colspan-basis}
    Let $K$ be a field, and let $m$, $n_1$, and $n_2$ be positive integers. For each $i \in \{1,2\}$, let $M_i$ be an $m \times n_i$ matrix with entries in $K$. We can compute a basis for
    $$
    \colspan(M_1) \cap \colspan(M_2)  
    $$
    using $O(m\cdot (n_1+n_2)\cdot \min\{m, n_1 + n_2\})$ field operations in $K$.
\end{lemma}
\begin{proof}
    We acknowledge the StackExchange answer \cite{SE-answer-colspan-basis} as the inspiration for our argument. For each $i$, let $r_i = \operatorname{rank}(M_i)$. For each $i$, Lemma~\ref{lem-time-complexity-reduced-row-echelon} tells us that, with $O(mn_i\min\{m,n_i\})$ field operations in $K$, we can use elementary column operations to replace $M_i$ with an $m \times r_i$ matrix with the same column span. Do this for both $i$, so that both linear transformations $M_i : K^{r_i} \to K^m$ are injective. Let $A$ be the matrix $(M_1 | -M_2)$, so that we have a linear transformation $A: K^{r_1 + r_2} \to K^m$. Lemma~\ref{lem-time-complexity-reduced-row-echelon} tells us that we may compute a reduced row decomposition of $A$ using $O(m\cdot (n_1+n_2)\cdot \min\{m, n_1 + n_2\})$ field operations in $K$. Using this decomposition, we may then quickly find a basis $\{v_i\}_i$ for $\ker A$. For each $i$, write 
    $$
    v_i = \begin{pmatrix}
        x_i \\ y_i 
    \end{pmatrix}
    $$
    for $x_i \in K^{r_1}$ and $y_i \in K^{r_2}$. Since $M_1$ and $M_2$ are injective, it is easy to see that $\{M_1x_i\}_i$ is a basis for $\colspan(M_1) \cap \colspan(M_2)$, so we are done. 
\end{proof}
\begin{lemma}
    \label{lem-time-complexity-for-size-of-NEAc}
    Let $F$ be a $2$-adic field, let $\mathcal{A}\subseteq F^\times$ be a finitely generated subgroup, and let $c$ be a nonnegative integer. Given choices for $\omega$ and $\mathcal{G}_4(\mathcal{A})$, the size $\# N_{E,c}^\mathcal{A}$ may be computed with either of the following two time complexities:
    \begin{enumerate} 
        \item $O_F(\#\mathcal{G}_4(\mathcal{A}) \cdot 2^{[F:\Q_2]} \cdot [F:\Q_2]^3)$.
        \item $O_F(2^{\#\mathcal{G}_4(\mathcal{A})}\cdot [F:\Q_2]^3)$. 
    \end{enumerate}
\end{lemma}
\begin{proof}
    The first algorithm is by brute-force. For each $x \in F^\times / F^{\times 2}$ and each $\alpha \in \mathcal{G}_4(\mathcal{A})$, Lemmas~\ref{lem-norms-in-galois-tower} and \ref{lem-time-complexity-of-norm-equations} tell us that we can check whether $\alpha \in N_\omega$ and whether $x \in \overline{N}_\alpha^2$ with time complexity $O_F([F:\Q_2]^3)$. Since $F^\times / F^{\times 2}$ has $2^{2 + [F:\Q_2]}$ elements, this first algorithm has the claimed time complexity. 

    We now describe the second algorithm. Using Lemma~\ref{lem-basis-for-units-mod-squares}, fix a basis for $F^\times / F^{\times 2}$. Also using Lemma~\ref{lem-basis-for-units-mod-squares}, for each $\alpha \in \mathcal{G}_4(\mathcal{A})$, we can write down a basis for $F(\sqrt{\alpha})^\times / F(\sqrt{\alpha})^{\times 2}$ and use it to obtain a generating set for $\overline{N}_\alpha^2$. We can do this for all $\alpha\in \mathcal{G}_4(\mathcal{A})$ with time complexity $O_F(\#\mathcal{G}_2(\mathcal{A})\cdot[F:\Q_2])$. Moreover, using Lemma~\ref{lem-basis-for-units-mod-squares}, we can quickly express these generating sets in terms of our fixed basis for $F^\times / F^{\times 2}$, and by Lemma~\ref{lem-time-complexity-reduced-row-echelon} we can reduce all of these generating sets to bases with time complexity 
    $$
    O_F(\#\mathcal{G}_2(\mathcal{A})\cdot [F:\Q_2]^3).
    $$ 
    Define the $\F_2$-vector subspace $V\subseteq F^\times / F^{\times 2}$ by  
    $$
    V = U_F^{(c)}F^{\times 2} / F^{\times 2} \cap \bigcap_{\alpha \in \mathcal{G}_4(\mathcal{A}) \cap N_\omega} \overline{N}_\alpha^2.
    $$
    By Lemmas~\ref{lem-norms-in-galois-tower}, \ref{lem-basis-for-units-mod-squares}, and \ref{lem-time-complexity-of-norm-equations}, we can compute the intersection $\mathcal{G}_4(\mathcal{A}) \cap N_\omega$ with time complexity $O_F(\#\mathcal{G}_4(\mathcal{A})\cdot [F:\Q_2]^3)$. 
    Taking successive intersections, Lemma~\ref{lem-time-complexity-colspan-basis} tells us that we can compute a basis of $V$ with time complexity 
    $$
    O_F(\#\mathcal{G}_4(\mathcal{A})\cdot [F:\Q_2]^3).
    $$
    Write $\{\alpha_1,\ldots, \alpha_m\} = \mathcal{G}_4(\mathcal{A}) \setminus N_\omega$,
    and for each $i$ let 
    $$
    U_i = \overline{N}_{\alpha_i}^2.
    $$
    Let $k$ be a positive integer and suppose that we have integers $i_j$ with $1 \leq i_1 < \ldots < i_k \leq m$. If we already have a basis for $V \cap \bigcap_{1\leq j \leq k-1}U_{i_j}$, then using Lemma~\ref{lem-time-complexity-colspan-basis}, we can compute a basis for $V \cap \bigcap_{1\leq j \leq k}U_{i_j}$ with time complexity $O_F([F:\Q_2]^3)$. Doing this for each of the $2^{\# \mathcal{G}_4(\mathcal{A})}$ possible tuples $(i_j)$, we can use the inclusion-exclusion principle to evaluate 
    $$
    \#N_{E,c}^\mathcal{A} = \# \Big(V \setminus \bigcup_i U_i\Big)
    $$
    with time complexity $O_F(2^{\#\mathcal{G}_4(\mathcal{A})}\cdot [F:\Q_2]^3)$, as required. 
\end{proof}

\begin{lemma}
    \label{lem-complexity-of-finding-top-half-of-tower}
    Let $F$ be a $2$-adic field, let $\mathcal{A}\subseteq F^\times$ be a finitely generated subgroup, and fix a choice of $\mathcal{G}_4(\mathcal{A})$. Let $E \in \Et{(1^2)/F}{}$. For each positive integer $m_2$, we can compute $\#\Et{(1^2)/E,m_2}{C_4/F,\mathcal{A}}$ with either of the following two time complexities:
    \begin{enumerate} 
        \item $O_F(e_F\cdot \#\mathcal{G}_4(\mathcal{A}) \cdot 2^{[F:\Q_2]} \cdot [F:\Q_2]^3)$.
        \item $O_F(e_F \cdot 2^{\#\mathcal{G}_4(\mathcal{A})}\cdot [F:\Q_2]^3)$. 
    \end{enumerate}
\end{lemma}
\begin{proof}
    By class field theory, we have 
    $$
    \Nm F\big(\sqrt{\langle \mathcal{A}, -1\rangle}\big) = \Nm F(\sqrt{-1}) \cap \bigcap_{\alpha \in \mathcal{G}_4(\mathcal{A})} \Nm F(\sqrt{\alpha}). 
    $$
    Let $d \in F^\times$ be such that $E = F(\sqrt{d})$. By Lemma~\ref{lem-time-complexity-of-norm-equations}, we can determine whether $d \in \Nm F\big(\sqrt{\langle\mathcal{A},-1\rangle}\big)$ with time complexity $O_F(\#\mathcal{G}_4(\mathcal{A})\cdot [F:\Q_2]^3)$. Suppose that $d \not \in \Nm F\big(\sqrt{\mathcal{A}, -1}\big)$. By symmetry of the quadratic Hilbert symbol, we have $\mathcal{A}\not\subseteq N_{E/F}E^\times$ or $-1 \not \in N_{E/F}E^{\times}$. By the tower law for norms and \cite[Corollary~4.9]{monnet2024counting}, this implies that
    $
    \Et{(1^2)/E,m_2}{C_4/F,\mathcal{A}} = \varnothing.
    $

    Suppose instead that $d \in \Nm F\big(\sqrt{\mathcal{A}, -1}\big)$, so that $E \in \Et{(1^2)/F}{\mathcal{A}}$ and $-1 \in N_{E/F}E^\times$. Let $m_1 = v_F(d_{E/F})$. If $m_1 \leq e_F$, then Theorem~\ref{thm-algo-for-small-disc-C4-works}, Lemma~\ref{lem-time-complexity-reduced-varphi-matrix}, and Theorem~\ref{thm-time-comlexity-of-finding-omega} tell us that we can compute $\omega \in E^\times$ such that $E(\sqrt{\omega}) \in \Et{(1^2)/E, 3m_1 - 2}{C_4/F}$ with time complexity $O_F([F:\Q_2]^3)$. If $m_1 > e_F$, then let $d \in F^\times$ be such that $E = F(\sqrt{d})$, and, again with time complexity $O_F([F:\Q_2]^3)$, let $\omega \in E^\times$ be such that $N_{E/F}\omega \in dF^{\times 2}$. In that case, \cite[Lemma~4.8]{monnet2024counting} tells us that $E(\sqrt{\omega}) \in \Et{(1^2)/E}{C_4/F}$. Moreover, \cite[Lemma~4.4]{monnet2024counting} tells us that $E(\sqrt{\omega})$ has minimal discriminant among elements of $\Et{(1^2)/E}{C_4/F}$. By Lemma~\ref{lem-conditions-to-be-all-or-nothing-or-other-thing}, we now just need to compute each size $\#N_{E,c}^\mathcal{A}$ for $O(e_F)$ values of $c$, so the result follows by Lemma~\ref{lem-time-complexity-for-size-of-NEAc}.
\end{proof}

\begin{lemma}
    \label{lem-time-complexity-2-adic-mass-of-tot-ram-C4}
    Let $F$ be a $2$-adic field and let $\mathcal{A}\subseteq F^\times$ be a finitely generated subgroup. Given a choice of $\mathcal{G}_4(\mathcal{A})$, the mass 
    $$
    \m\big(\Et{(1^4)/F}{C_4/F,\mathcal{A}}\big)
    $$
    can be computed with either of the following time complexities:
    \begin{enumerate}
        \item $$O_{F}\big(e_{F}\cdot \#\mathcal{G}_4(\mathcal{A})\cdot 2^{2[F:\Q_2]}\cdot [F:\Q_2]^3\big).$$
        \item $$O_{F}\big(e_{F}\cdot 2^{\#\mathcal{G}_4(\mathcal{A})}\cdot 2^{[F:\Q_2]}\cdot [F:\Q_2]^3\big).$$
    \end{enumerate}
\end{lemma}
\begin{proof}
    There is a natural bijection
    $$
    \bigsqcup_{\substack{2m_1 + m_2 = m \\ m_1, m_2 > 0}}\bigsqcup_{E \in \Et{(1^2)/F, m_1}{\mathcal{A}}} \Et{(1^2)/E,m_2}{C_4/F,\mathcal{A}} \longleftrightarrow \Et{(1^4)/F,m}{C_4/F,\mathcal{A}}. 
    $$
    Since $\#\Et{(1^2)/F}{C_2/F} = 4q^{e_F} - 2$, the result follows from Lemma~\ref{lem-complexity-of-finding-top-half-of-tower}.
\end{proof}
\begin{proof}
    [Proof of Theorem~\ref{thm-time-complexity-mass-all-quartics-with-norms}]
    It is clear that of all the quantities in Theorems~\ref{thm-n=4-for-odd-primes}, \ref{thm-1^21^2-2-adic}, \ref{thm-2^2-2-adic}, and \ref{thm-1^4-2-adic}, those in Theorem~\ref{thm-1^4-2-adic}, Part (4) are by far the most difficult to evaluate. Thus, the result follows immediately from Lemma~\ref{lem-time-complexity-2-adic-mass-of-tot-ram-C4}.
\end{proof}
\appendix

\section{Proof of Theorem~\ref{thm-algo-for-small-disc-C4-works}}
\label{appendix-algo}

Let $E,m_1$ be as in the setup of Algorithm~\ref{algo-omega-with-small-disc-val}. In this appendix, we will step through the algorithm, showing that each stage is well-defined, and eventually proving that the output has the desired property.

\begin{enumerate}
    \item Start by taking any $d$ with $E = F(\sqrt{d})$. Since $v_F(d_{E/F}) \leq e_F$, \cite[Lemma~4.6]{monnet2024counting} tells us that $v_F(d)$ is even, which means we can multiply by some even power of $\pi_F$ to get $v_F(d) = m_1$.  
    \item \cite[Lemma~4.6]{monnet2024counting} tells us that there is some $a \in F^\times$ such that $\frac{d}{a^2}\equiv 1\pmod{\p_F^{2e_F + 1 - m_1}}$, and moreover that there is no such $a$ for any higher power of $\p_F$. This implies that 
    $$
    v_F\Big(\frac{d}{a^2} - 1\Big) = 2e_F + 1 - m_1,
    $$  
    so 
    $$
    v_F(d - a^2) = 2e_F + 1. 
    $$
    Setting $b = \frac{d - a^2}{4}$, we obtain $a,b$ as required. 
    \item It is easy to see that $\rho^2 - a\rho - b = 0$, so the minimal polynomial of $\rho$ over $F$ is Eisenstein, and therefore $\rho$ is a uniformiser of $E$ with $\co_E = \co_F \oplus \co_F\cdot \rho$. 
    \item Since $-1 \equiv 1\pmod{\p_F^{e_F}}$, we have $v_F(d_{F(\sqrt{-1})/F}) \leq e_F + 1$ by \cite[Corollary~4.7]{monnet2024counting}, so $U_F^{(e_F+1)}\subseteq \Nm F(\sqrt{-1})$, and therefore 
    $$
    \frac{d}{a^2} \in U_F^{(2e_F + 1 - m_1)} \subseteq U_F^{(e_F + 1)} \subseteq \Nm F(\sqrt{-1}),
    $$
    which implies that $(-1,d)_F = 1$, and therefore \cite[Lemma~4.8]{monnet2024counting} and \cite[Corollary~4.9]{monnet2024counting} tell us that we may choose $\omega \in E^\times$ with $N_{E/F}\omega \in dF^{\times 2}$. Moreover, we may ensure that $v_E(\omega) = 0$, since $v_F(d)$ is even. 
    \item We know that $N_{E/F}\omega = dx^2$ for some $x \in F^\times$ with $v_F(x) = - m_1/2$. Setting $\lambda = ax$, it is easy to see that $N_{E/F}\omega \equiv \lambda^2\pmod{\p_F^{2e_F + 1 - m_1}}$. 
    \item We address Step 6 with a sequence of lemmas. 
    \begin{lemma}
        \label{lem-trace-of-int-valuation}
        For all $x \in \co_E$, we have $v_F(\Tr_{E/F}x) \geq \frac{m_1}{2}$.
    \end{lemma}
    \begin{proof}
        This follows easily from the fact that $x = s + t\rho$ for elements $s,t\in\co_F$. 
    \end{proof}
    \begin{lemma}
        \label{lem-omega-lambda-close}
        We have 
        $$
        v_E(\omega - \lambda) \geq m_1 - 1. 
        $$
    \end{lemma}
    \begin{proof}
        Let $\gamma = \omega - \lambda$. Define the sequence $(a_n)_{n\geq 0}$ as follows. Set $a_0 = 0$, and for each $n\geq 0$, define 
        $$
        a_{n+1} = \min\Big\{\Big\lfloor\frac{a_n}{2}\Big\rfloor + \frac{m_1}{2}, 2e_F + 1 - m_1\Big\}.
        $$
        We claim that $v_E(\gamma) \geq a_n$ for all $n$. The base case $n=0$ is clear. Suppose that $v_E(\gamma) \geq a_n$ for some $n$. Then $\gamma / \pi_F^{\lfloor \frac{a_n}{2}\rfloor} \in \co_E$, so it follows from Lemma~\ref{lem-trace-of-int-valuation} that 
        $$
        v_F(\Tr_{E/F}\gamma) \geq \frac{m_1}{2} + \Big\lfloor \frac{a_n}{2}\Big\rfloor. 
        $$ 
        Since $N_{E/F}\omega \equiv \lambda^2\pmod{\p_F^{2e_F + 1 - m_1}}$, we have 
        $$
        \lambda \Tr_{E/F}\gamma + N_{E/F}\gamma \equiv 0 \pmod{\p_F^{2e_F + 1 - m_1}},
        $$
        and it follows that $v_F(N_{E/F}\gamma) \geq a_{n+1}$. Since $E/F$ is totally ramified, we have 
        $$
        v_F(N_{E/F}\gamma) = v_E(\gamma),
        $$
        so indeed $v_E(\gamma) \geq a_{n+1}$, and by induction this is true for all $n$. 

        It is easy to see that if $a_n < m_1 - 1$, then $a_n < a_{n+1}$, so there is some $n$ with $a_n \geq m_1 - 1$, and therefore $v_E(\gamma) \geq m_1 - 1$, as required. 
    \end{proof}
    Lemma~\ref{lem-omega-lambda-close} tells us that $\omega_1$ is well-defined. 
    \begin{lemma}
        \label{lem-omega_1-has-properties}
        The following two statements are true:
        \begin{enumerate}
            \item $N_{E/F}\omega_1 = N_{E/F}\omega$.
            \item $\omega_1 \equiv \lambda \pmod{\p_E^{m_1}}$. 
        \end{enumerate}
    \end{lemma}
    \begin{proof}
        If $v_E(\omega - \lambda) \geq m_1$, then there is nothing to prove, so let us assume that $v_E(\omega - \lambda) = m_1 - 1$. The first claim follows from that fact that $N_{E/F}\rho = -b$. Write $\gamma = \omega - \lambda$. Since $v_E(\gamma) = m_1 - 1$, we have $\gamma / \pi_F^{m_1/2 - 1} = u + v\rho$ for elements $u,v \in \co_F$ with $v_F(u) \geq 1$ and $v_F(v) = 0$. We have 
        \begin{align*}
            N_{E/F}\omega - \lambda^2 &= \lambda\Tr_{E/F}\gamma + N_{E/F}\gamma 
            \\
            &= \lambda \pi_F^{\frac{m_1}{2} - 1}(2u + av) + \pi_F^{m_1-2}(u^2 + auv - bv^2) 
            \\
            &\equiv \lambda av \pi_F^{\frac{m_1}{2} - 1} - bv^2 \pi_F^{m_1-2}\pmod{\p_F^{m_1}}.
        \end{align*}
        We know that $N_{E/F}\omega \equiv \lambda^2\pmod{\p_F^{2e_F + 1 - m_1}}$, and $m_1 \leq e_F$, so in fact 
        $$
        N_{E/F}\omega \equiv \lambda^2\pmod{\p_F^{m_1}},
        $$
        and it follows that 
        $$
        \lambda av \pi_F^{-\frac{m_1}{2}} - bv^2\pi_F^{-1} \equiv 0 \pmod{\p_F}.
        $$
        Since $v_F(v) = 0$, it follows that
        $$
        v \equiv \frac{\lambda a}{b\pi_F^{\frac{m_1}{2}-1}}\pmod{\p_F},
        $$
        and therefore  
        $$
        \gamma \equiv \frac{\lambda a \rho}{b}\pmod{\p_E^{m_1}},
        $$
        so 
        $$
            \omega \equiv \lambda\Big(1 + \frac{a\rho}{b}\Big) \pmod{\p_E^{m_1}} ,
        $$
        and it follows that  
        \begin{align*}
        \omega_1 &= \frac{\omega b}{\rho^2} 
        \\
        &\equiv \lambda\Big(\frac{b}{\rho^2} + \frac{a}{\rho}\Big)\pmod{\p_E^{m_1}}
        \\
        &= \lambda \cdot \frac{b + a\rho}{\rho^2}
        \\
        &= \lambda. 
        \end{align*} 
    \end{proof}
    \item Write $\omega_1 = r_1 + s_1\rho$, for $r_1,s_1 \in \co_F$.
    \begin{lemma}
        The following two statements are true: 
        \begin{enumerate}
            \item $v_F(s_1)\geq \frac{m_1}{2}$. 
            \item $v_E(\omega_1 - \bar{\omega}_1) = 2v_F(s_1) + m_1$.
        \end{enumerate}
    \end{lemma}
    \begin{proof}
        Since $\omega_1 \equiv \lambda\pmod{\p_E^{m_1}}$, we have 
        $$
        (r_1 - \lambda) + s_1\rho \equiv 0 \pmod{\p_E^{m_1}},
        $$
        so $v_F(s_1)\geq \frac{m_1}{2}$. The second statement is obvious. 
    \end{proof}
    It follows that $\omega_2$ and $\lambda_2$ are well-defined and their definitions are equivalent to 
    $$
    \omega_2 = \begin{cases}
        \omega_1\quad\text{if $v_F(s_1) = \frac{m_1}{2}$},
        \\
        \omega_1(1+\rho)^2 \quad\text{if $v_F(s_1) > \frac{m_1}{2}$},
    \end{cases}
    $$
    and 
    $$
    \lambda_2 = \begin{cases}
        \lambda\quad\text{if $v_F(s_1) = \frac{m_1}{2}$},
        \\
        \lambda (1+a-b) \quad\text{if $v_F(s_1) > \frac{m_1}{2}$}.
    \end{cases}
    $$
    Write $\omega_2 = r_2 + s_2\rho$.
    \begin{lemma}
        \label{lem-lemma-omega-2-lambda-2-s2}
        We have 
        \begin{enumerate}
            \item $N_{E/F}\omega_2 \equiv \lambda_2^2\pmod{\p_F^{2e_F + 1 - m_1}}$.
            \item $\omega_2 \equiv \lambda_2\pmod{\p_E^{m_1}}$.
            \item $v_F(s_2) = \frac{m_1}{2}$. 
        \end{enumerate}
    \end{lemma}
    \begin{proof}
        If $v_F(s_1) = \frac{m_1}{2}$, then this is Lemma~\ref{lem-omega_1-has-properties}, so we will assume that $v_F(s_1) > \frac{m_1}{2}$. The first statement follows from Lemma~\ref{lem-omega_1-has-properties}, along with the fact that $N_{E/F}(1 + \rho) = 1 + a - b$. It is easy to see that 
        $$
        (1+\rho)^2 - N_{E/F}(1+\rho) = (1+\rho)\sqrt{d},
        $$
        so 
        $$
        (1+\rho)^2 \equiv 1 + a - b\pmod{\p_E^{m_1}},
        $$
        and the second statement follows. 
        Since $(1 + \rho)^2 = (1 + b) + (2+a)\rho$, we have 
        $$
        \omega_2 = (1 + b)\omega_1 + (2+a)\rho \omega_1, 
        $$
        and therefore 
        $$
        \bar{\omega}_2 = (1 + b)\bar{\omega}_1 + (2+a)\bar{\rho}\bar{\omega}_1,
        $$
        so
        $$
        \omega_2 - \bar{\omega}_2 = (1 + b)(\omega_1 - \bar{\omega}_1) + (2+a)(\rho \omega_1 - \bar{\rho}\bar{\omega}_1). 
        $$
        We know that $v_E(\omega_1 - \bar{\omega}_1) > 2m_1$, so 
        $$
        \omega_2 - \overline{\omega}_2 \equiv (2+a)(\rho \omega_1 - \bar{\rho}\bar{\omega}_1)\pmod{\p_E^{2m_1+1}}.
        $$
        Since $\omega_1 = r_1 + s_1\rho$, we have 
        $$
        \rho\omega_1 - \bar{\rho}\bar{\omega}_1 = (\rho - \bar{\rho})(r_1 + s_1(\rho + \bar{\rho})). 
        $$
        It is easy to see that $v_E(\rho - \bar{\rho}) = m_1$. Since $v_E(\omega_1) = 0$, we have $v_E(r_1) = 0$, so 
        $$
        v_E(r_1 + s_1(\rho + \bar{\rho})) = 0,
        $$
        and therefore 
        $$
        v_E(\rho\omega_1 - \bar{\rho}\bar{\omega}_1) = m_1.
        $$
        It is easy to see that $v_E(2+a) = m_1$, so $v_E(\omega_2 - \overline{\omega}_2) = 2m_1$, and therefore $$v_F(s_2) = \frac{m_1}{2}.$$ 
    \end{proof}
    \item We know that $s_2 \neq 0$ since $N_{E/F}\omega_2 \in dF^{\times 2}$, so $\omega_2 \not \in F^\times$. Similarly, $v_E(\omega_2) = 0$ so $r_2 \neq 0$. It follows that $q$ and $n$ are well-defined. 
    \item Since $q \in F$ and $\rho \not \in F$, we have $q + \rho \neq 0$, and therefore the output is well-defined. Since $n \in F^\times$, we have 
    $$
    N_{E/F}\Big(\frac{\omega_2n}{(q+\rho)^2}\Big) \in dF^{\times 2}. 
    $$
    \begin{lemma}
        \label{lem-val-r2-lambda2}
        We have $v_F(r_2 - \lambda_2) = \frac{m_1}{2}$.
    \end{lemma}
    \begin{proof}
        We have 
        $$
        N_{E/F}\omega_2 = r_2^2 + ar_2s_2 - bs_2^2,
        $$
        so 
        \begin{equation}
            \label{eqn-congruence-for-lemma-about-r2-and-lambda-2}
        (r_2^2 - \lambda_2^2) + ar_2s_2 - bs_2^2 \equiv 0 \pmod{\p_F^{2e_F + 1 - m_1}}.\tag{$*$}
        \end{equation}
        We know that $v_F(bs_2^2) = m_1 + 1$ and $v_F(ar_2s_2) = m_1$, so Equation~(\ref{eqn-congruence-for-lemma-about-r2-and-lambda-2}) implies that  
        $$
        v_F(r_2^2 - \lambda_2^2) = m_1. 
        $$
        Suppose for a contradiction that $v_F(r_2 - \lambda_2) \geq e_F$. Then $v_F(r_2 + \lambda_2) \geq e_F$, so 
        $$
        m_1 = v_F(r_2^2 - \lambda_2^2) \geq 2e_F > e_F,
        $$
        contradicting the fact that $m_1 \leq e_F$. Therefore, we have $v_F(r_2 - \lambda_2) < e_F$, and consequently $v_F(r_2 + \lambda_2) = v_F(r_2 - \lambda_2)$, so the result follows. 
    \end{proof}
    We also have 
    \begin{align*}
        \omega_2 n - (q + \rho)^2 &= (b - \rho^2) + \frac{\rho b s_2}{r_2} - q\rho \frac{\lambda_2+r_2}{r_2}
        \\
        &= (b - \rho^2) + \frac{\rho b s_2}{r_2} + \rho\frac{\lambda_2^2 - r_2^2}{r_2s_2} 
        \\
        &= \frac{\rho}{r_2s_2}\big(\lambda_2^2 - r_2^2 + \frac{r_2s_2}{\rho}(b - \rho^2) + bs_2^2\big)
        \\
        &= \frac{\rho}{r_2s_2}(\lambda_2^2 - r_2^2 + bs_2^2 - ar_2s_2) 
        \\
        &= \frac{\rho}{r_2s_2}(\lambda_2^2 - N_{E/F}\omega_2).
    \end{align*}
    Since $v_E(r_2) = 0$ and $v_E(s_2) = m_1$, we have $v_E(\frac{\rho}{r_2s_2}) = 1 - m_1$. Since
    $$
    v_E(\lambda_2^2 - N_{E/F}\omega_2) \geq 4e_F + 2 - 2m_1,
    $$
    it follows that  
    $$
    \omega_2 n \equiv (q + \rho)^2 \pmod{\p_E^{4e_F + 3 - 3m_1}}.
    $$
    Lemmas~\ref{lem-lemma-omega-2-lambda-2-s2} and \ref{lem-val-r2-lambda2} tell us that $v_F(q) = 0$, and therefore $v_E(q + \rho) = 0$, so it follows that
    $$
    \frac{\omega_2 n}{(q+\rho)^2} \equiv 1 \pmod{\p_E^{4e_F + 3 - 3m_1}}. 
    $$
    Theorem~\ref{thm-algo-for-small-disc-C4-works} then follows by \cite[Lemma~4.4]{monnet2024counting} and \cite[Lemma~4.6]{monnet2024counting}. 
\end{enumerate}

\section{Explicit helper functions}

\label{appendix-helpers}
\begin{itemize}
\item Let $p$ be an integer with $p \geq 2$, and let $q$ be a positive rational number. For integers $t$ with $t \geq 2$, define the functions $A(t)$ and $B(t)$ by
$$
    A(t) =  \begin{cases} 
        q^{1 - \lfloor\frac{t}{2}\rfloor}\cdot \frac{q^{\lfloor \frac{t}{2} \rfloor} - 1}{q-1}\quad\text{if $p=2$},
        \\
        q^{-p(p-2)}\cdot \frac{q^{(p-1)(p-2)} - 1}{q^{p-2} - 1} \cdot \frac{
            q^{-(p-1)^2\cdot\lfloor\frac{t}{p}\rfloor} - 1
        }{
            q^{-(p-1)^2} - 1
        }\quad\text{if $p\neq 2$},
    \end{cases}
    $$
    and 
    $$
    B(t) = \begin{cases} 0 \quad\text{if $p=2$} ,
        \\
        q^{- \lfloor\frac{t}{p}\rfloor} \cdot \frac{
        q^{-(p-2)(t+1)} - q^{-(p-2)(\lfloor\frac{t}{p}\rfloor p + 2)}
    }
    {
        q^{-(p-2)} - 1
    } \quad\text{if $p\neq 2$}. 
    \end{cases}
    $$
\item Define the explicit function $N_{(1^21^2)}^{\neq}$ by 
$$
N_{(1^21^2)}^{\neq}(m) = \begin{cases}
    2(q-1)^2 q^{\frac{m}{2} - 2}(\frac{m}{2} - 1) - \mathbbm{1}_{4\mid m}(q-1)q^{\frac{m}{4} - 1}\quad\text{if $4\leq m \leq 2e_F$ and $m$ is even},
    \\
    2(q-1)^2 q^{\frac{m}{2} - 2}(2e_F - \frac{m}{2} + 1) - \mathbbm{1}_{4\mid m}(q-1)q^{\frac{m}{4} - 1}\quad\text{if $2e_F + 2\leq m \leq 4e_F$ and $m$ is even},
    \\
    4(q-1)q^{\frac{m-1}{2} - 1}\quad\text{if $2e_F + 3 \leq m \leq 4e_F + 1$ and $m$ is odd},
    \\
    q^{e_F}(2q^{e_F} - 1)\quad\text{if $m = 4e_F + 2$},
    \\
    0\quad\text{otherwise}. 
    \end{cases}
$$

    \item $$
    N^{C_2}(m_2) = \begin{cases}
        2(q-1)q^{\frac{m_2}{2} - 1} \quad\text{if $0 \leq m_2 \leq 4e_F$ and $m_2$ is even},
        \\
        2q^{2e_F}\quad\text{if $m_2 = 4e_F + 1$},
        \\
        0\quad\text{otherwise}.
    \end{cases}
    $$
    \item Let $m_1$ be an even integer with $2 \leq m_1 \leq e_F$. For each integer $m_2$, define 
    $$
    N^{C_4}(m_1,m_2) = \begin{cases}
        q^{m_1-1}\quad\text{if $m_2 = 3m_1 - 2$},
        \\
        q^{\lfloor \frac{m_1+m_2}{4}\rfloor} - q^{\lfloor \frac{m_1+m_2-2}{4}\rfloor} \quad\text{if $3m_1 \leq m_2 \leq 4e_F - m_1$ and $m_2$ is even},
        \\
        q^{e_F} \quad\text{if $m_2 = 4e_F - m_1 + 2$},
        \\
        0 \quad\text{otherwise}. 
    \end{cases}
    $$
    Suppose that $m_1 = 2e_F + 1$ or $m_1$ is even with $e_F < m_1 \leq 2e_F$. Then define 
    $$
    N^{C_4}(m_1,m_2) = \begin{cases}
        2q^{e_F}\quad\text{if $m_2 = m_1 + 2e_F$},
        \\
        0\quad\text{otherwise}. 
    \end{cases}
    $$
    Finally, define $N^{C_4}(m_1,m_2) = 0$ for all other pairs of integers $(m_1,m_2)$. 
    \item Let $m_1$ be either $2e_F + 1$ or an even integer with $2 \leq m_1 \leq 2e_F$. Define 
    $$
    N^{V_4}(m_1,m_2) = \begin{cases}
        2(q-1)q^{\frac{m_2}{2} - 1} \quad\text{if $2 \leq m_2 < m_1$ and $m_2$ is even},
        \\
        (q-2)q^{\frac{m_1}{2} - 1}\quad\text{if $m_2 = m_1$ and $m_1$ is even},
        \\
        (q-1)q^{\frac{m_1+m_2}{4} - 1}\quad\text{if $m_1 < m_2 \leq 4e_F - m_1$ and $m_1 \equiv m_2 \pmod{4}$},
        \\
        q^{e_F}\quad\text{if $m_2 > m_1$ and $m_1 + m_2 = 4e_F + 2$},
        \\
        0\quad\text{otherwise}.
    \end{cases}
    $$
    Define $N^{V_4}(m_1,m_2) = 0$ for all other pairs of integers $(m_1,m_2)$. 
\end{itemize}

\printbibliography
\end{document}